\pdfoutput=1 

\documentclass[a4paper,twoside]{article}

\usepackage{amssymb}
\usepackage{amsmath}
\usepackage{amsthm}

\usepackage{microtype}

\usepackage{mathrsfs}

\usepackage{graphicx}

\usepackage{esint}                

\usepackage{enumitem}

\usepackage{hyperref}

\newtheoremstyle{citing}
  {3pt}
  {3pt}
  {\itshape}
  {}
  {\bfseries}
  {.}
  {.5em}
  {\thmnote{#3}}
\theoremstyle{citing}
\newtheorem*{citing}{}

\theoremstyle{definition}

\theoremstyle{remark}

\swapnumbers
\theoremstyle{plain}

\newtheorem{theorem}{Theorem}[section]
\newtheorem{lemma}[theorem]{Lemma}
\newtheorem{corollary}[theorem]{Corollary}

\theoremstyle{remark}
\newtheorem{remark}[theorem]{Remark}
\newtheorem*{remark*}{Remark}
\newtheorem{example}[theorem]{Example}

\theoremstyle{definition}
\newtheorem{definition}[theorem]{Definition}
\newtheorem{miniremark}[theorem]{}

\newcounter{counter1}
\newcounter{counter2}

\newcommand{\Var}{\mathbf{V}}     
\newcommand{\RVar}{\mathbf{RV}}   
\newcommand{\IVar}{\mathbf{IV}}   

\newcommand{\Lp}[1]{\mathbf{L}_{#1}}

\newcommand{\Lploc}[1]{\mathbf{L}_{#1}^{\mathrm{loc}}}

\newcommand{\Sob}[3]{\mathbf{W}_{#1}^{{#2},{#3}}}

\newcommand{\trunc}{\mathbf{T}}

\newcommand{\qspace}{\mathbf{Q}}

\newcommand{\nat}{\mathscr{P}}

\newcommand{\LM}{\mathscr{L}}

\newcommand{\HM}{\mathscr{H}}

\newcommand{\integers}{\mathbf{Z}}

\newcommand{\rel}{\mathbf{R}}

\newcommand{\grass}[2]{\mathbf{G}(#1,#2)}

\newcommand{\pp}{\mathbf{p}}
\newcommand{\qq}{\mathbf{q}}

\newcommand{\ocube}[2]{O(#1,#2)}

\newcommand{\oball}[2]{\mathbf{U}(#1,#2)}

\newcommand{\cball}[2]{\mathbf{B}(#1,#2)}

\newcommand{\density}{\boldsymbol{\Theta}}

\newcommand{\unitmeasure}[1]{\boldsymbol{\alpha}(#1)}

\newcommand{\besicovitch}[1]{\boldsymbol{\beta}(#1)}

\newcommand{\isoperimetric}[1]{\boldsymbol{\gamma}(#1)}

\newcommand{\cylinder}[4]{\mathbf{C} ( #1, #2, #3, #4 )}

\newcommand{\id}[1]{\mathbf{1}_{#1}}

\newcommand{\weakD}{\operatorname{\mathbf{D}}}

\newcommand{\derivative}[2]{{#1}\,\weakD{#2}}

\newcommand{\ud}{\ensuremath{\,\mathrm{d}}}

\DeclareMathOperator{\with}{:}
\newcommand{\classification}[3]{{#1} \cap \{ {#2} \with {#3} \}}
\newcommand{\eqclassification}[3]{{(#1)} \cap \{ {#2} \with {#3} \}}

\newcommand{\project}[1]{#1_\natural}

\newcommand{\eqproject}[1]{(#1)_\natural}

\newcommand{\perpproject}[1]{#1_\natural^\perp}

\newcommand{\lIm}{[}
\newcommand{\rIm}{]}

\newcommand{\Lbrack}{\boldsymbol{[}}
\newcommand{\Rbrack}{\boldsymbol{]}}

\newcommand{\vdim}{{m}}

\newcommand{\codim}{{n-m}}

\newcommand{\adim}{{n}}

\DeclareMathOperator{\without}{\sim}

\newcommand{\restrict}{\mathop{\llcorner}}

\newcommand{\class}[1]{#1}

\newcommand{\tint}[2]{{\textstyle\int_{#1}^{#2}}}
\newcommand{\tfint}[2]{{\textstyle\fint_{#1}^{#2}}}

\newcommand{\tsum}[2]{{\textstyle\sum_{#1}^{#2}}}

\DeclareMathOperator{\card}{card}

\newcommand{\Clos}[1]{\mathop{\mathrm{Clos}}#1}

\newcommand{\measureball}[2]{{#1}\,{#2}}

\DeclareMathOperator{\Nor}{Nor}     
\DeclareMathOperator{\Tan}{Tan}     
\DeclareMathOperator{\spt}{spt}     
\DeclareMathOperator{\im}{im}       
\DeclareMathOperator{\Int}{Int}     
\DeclareMathOperator{\Lip}{Lip}     
\DeclareMathOperator{\grad}{grad}   
\DeclareMathOperator{\trace}{trace} 
\DeclareMathOperator{\dmn}{dmn}     
\DeclareMathOperator{\dist}{dist}   
\DeclareMathOperator{\Hom}{Hom}     
\DeclareMathOperator{\Aff}{A}       
\DeclareMathOperator{\Der}{D}       
\DeclareMathOperator{\ap}{ap}       
\DeclareMathOperator*{\aplim}{\mathrm{ap}\, \lim}   
\DeclareMathOperator*{\aplimsup}{\mathrm{ap}\, \limsup}   
\DeclareMathOperator{\ach}{arcosh}  

\newcommand{\Lpnorm}[3]{{#1}_{({#2})}({#3})}
\newcommand{\eqLpnorm}[3]{{(#1)}_{({#2})}({#3})}

\newcommand{\orlicz}[3]{\tint{}{}{#1} \circ |{#2}| \ud {#3}}

\DeclareMathOperator{\Lap}{Lap}

\newcommand{\optional}[1]{}      

\theoremstyle{definition}
\swapnumbers

\begin{document}


\title{Decay rates for the quadratic and super-quadratic tilt-excess of
integral varifolds}
\author{S{\l}awomir Kolasi{\'n}ski \and Ulrich Menne}
\maketitle
\begin{abstract}
    This paper concerns integral varifolds of arbitrary dimension in an open
    subset of Euclidean space satisfying integrability conditions on their
    first variation.  Firstly, the study of pointwise power decay rates almost
    everywhere of the quadratic tilt-excess is completed by establishing the
    precise decay rate for two-dimensional integral varifolds of locally
    bounded first variation.  In order to obtain the exact decay rate, a
    coercive estimate involving a height-excess quantity measured in Orlicz
    spaces is established.  Moreover, counter-examples to pointwise power
    decay rates almost everywhere of the super-quadratic tilt-excess are
    obtained.  These examples are optimal in terms of the dimension of the
    varifold and the exponent of the integrability condition in most cases,
    for example if the varifold is not two-dimensional.  These examples also
    demonstrate that within the scale of Lebesgue spaces no local higher
    integrability of the second fundamental form, of an at least
    two-dimensional curvature varifold, may be deduced from boundedness of its
    generalised mean curvature vector.

    Amongst the tools are Cartesian products of curvature varifolds.
\end{abstract}
\paragraph{Keywords.}
Integral varifold, first variation, generalised mean curvature vector,
quadratric tilt-excess, super-quadratic tilt-excess, Orlicz space height-excess;
curvature varifold, second fundamental form; Cartesian product of varifolds.
    

\addcontentsline{toc}{section}{\numberline{}Introduction}
\section*{Introduction}

\subsection*{Overview}

Integral varifolds constitute an analytically tractable model for singular
geometric objects which admit appropriate notions of tangent plane and mean
curvature vector, see Almgren \cite{Almgren:Vari}, Allard \cite{MR0307015},
and Simon \cite{MR756417}. Due to good compactness properties they also arise
naturally as weak limits of smooth submanifolds of some ambient space and may
be used to represent solutions to geometrical variational problems.  Our
principal objective is the study of regularity properties entailed by
integrability conditions on the first variation of such varifolds by means of
decay rates of tilt-excess.  For this purpose the classical quadratic
tilt-excess and the super-quadratic tilt-excess which arises in the study of
area minimising integral currents are employed.  The three main results of
this study may be informally described as follows.

Firstly, for two-dimensional integral varifolds of locally bounded first
variation (i.e. $\vdim=2$ and $p=1$ in the general hypotheses below), the
optimal decay rate almost everywhere of the quadratic tilt-excess is
established, see Theorems~A and~B for the decay rate and its sharpness
respectively. For all other values of $(\vdim,p)$ the best decay rate amongst
\emph{powers} was determined in \cite{snulmenn.c2}.  The present result not only
fills this gap but in fact sharply exhibits the best decay rate amongst
\emph{all} rates, not just powers.  To obtain this precision, a coercive
estimate for the quadratic tilt-excess is derived which involves a height-excess
quantity measured in an Orlicz space occurring naturally in sharp embeddings of
Sobolev spaces. This seems to be the first time that a regularity estimate for
varifolds relies on Orlicz spaces; in fact, the only previous usage of Orlicz
spaces in the context of varifolds appears to be the Poincar{\'e} type embedding
results of Hutchinson, see \cite[Theorems 2 and 4]{MR1066398}.

Secondly, for at least two-dimensional integral varifolds with locally bounded
mean curvature and ``no boundary'' ($p = \infty$), negative results concerning
decay rates almost everywhere of super-quadratic tilt-excess are shown in
Theorem~C.  The importance of these examples stems from the fact that they imply
that -- even just almost everywhere and in co-dimension one -- there is no
analogue in the present situation of the ``main analytic estimate'' in Almgren's
proof of interior almost everywhere regularity of area minimising integral
currents in arbitrary codimension, see \cite{MR574247} and
\cite[\S3]{MR1777737}. (Of course, Almgren in fact proved a stronger Hausdorff
dimension estimate on the interior singular set.)  This provides a serious
obstacle to an, otherwise canonical, approach to the question of possible almost
everywhere regularity of stationary integral varifolds which is a key open
problem in varifold theory.

Thirdly, for one-dimensional integral varifolds of locally bounded first
variation ($p=1$), almost everywhere differentiability of the tangent plane
map (restricted to the set of points of approximate continuity) is proven, see
Theorem~D.  This implies in particular that such varifolds near almost every
point are representable as graphs of a finite number of Lipschitzian functions
with small constant.  These results as well as the estimates involved in
deriving them should prove useful in the study which parts of the structural
description of one-dimensional stationary varifolds with a uniform lower bound
on their density, see Allard and Almgren in \cite{MR0425741}, generalise to
locally bounded first variation.

In combination with previous results the preceding three main results
in particular yield a nearly complete picture concerning power decay rates
almost everywhere of quadratic and super-quadratic tilt-excess.


\subsection*{Known results}

In order to more formally describe the results of the present paper in the
context of known results, consider the following set of hypotheses; the notation
is explained in Section \ref{sec:notation}. Additionally, the terms ``[twice]
weakly differentiable'' are employed with their usual meaning, see
e.g.~\cite[p.~9--10]{snulmenn.decay}.
\begin{citing} [General hypotheses]
    Suppose $\vdim$ and $\adim$ are positive integers, $\vdim < \adim$, $1
    \leq p \leq \infty$, $U$ is an open subset of $\rel^\adim$, $V$ is an
    $\vdim$ dimensional integral varifold in $U$ whose first variation $\delta
    V$ is representable by integration\footnote{That is, in the terminology of
    Simon \cite[39.2]{MR756417}, $V$ is of locally bounded first variation.},
    the generalised mean curvature vector $\mathbf{h} (V,\cdot)$ of $V$
    belongs to $\Lploc{p} (\| V \|, \rel^\adim )$, and if $p > 1$ then
    \begin{gather*}
	( \delta V ) ( \theta ) = - \tint{}{} \mathbf{h} (V,z) \bullet \theta
	(z) \ud \| V \| z \quad \text{for $\theta \in \mathscr{D} ( U,
	\rel^\adim )$}.
    \end{gather*}
    Let $Z = U \cap \{ z \with \Tan^\vdim ( \| V \|, z ) \in \grass \adim
    \vdim \}$ and define $\tau : Z \to \Hom ( \rel^\adim, \rel^\adim )$ by
    $\tau (z) = \project{\Tan^\vdim ( \| V \|, z)}$ for $z \in Z$.
\end{citing}
The set $Z$ consists of all points such that the closed cone of $(\| V \|,
\vdim )$ approximate tangent vectors forms an $\vdim$ dimensional plane; for
these points $z$, $\tau (z)$ denotes the orthogonal projection of $\rel^\adim$
onto this plane.

The study of regularity properties is usually preceded by a more basic study
of the density ratio by means of the monotonicity identity, see Allard
\cite[\S 5]{MR0307015}, Simon \cite[\S 17]{MR756417} and \cite[\S
4]{snulmenn:tv.v2}, and its consequence, the isoperimetric inequality, see
Allard \cite[\S 7]{MR0307015} and Michael and Simon \cite[\S 2]{MR0344978}. In
particular, if $p > \vdim$ then $\density^\vdim ( \| V \|, \cdot )$ is an
upper semicontinuous real valued function whose domain is $U$, see Simon
\cite[17.8]{MR756417}. To which extent these properties persist if $p = \vdim$
is unclear; the cases $\vdim = 1$ and $\vdim = 2$ are treated in
\cite[4.8]{snulmenn:tv.v2} and Kuwert and Sch{\"a}tzle
\cite[Appendix]{MR2119722} respectively. From Allard \cite[3.5\,(1),
8.3]{MR0307015} one is at least assured that $\mathscr{H}^\vdim \restrict \spt
\| V \| \leq \| V \|$ if $p = \vdim$. If $p < \vdim$, one easily constructs
examples with $\spt \| V \| = U$, see \cite[14.1]{snulmenn:tv.v2}. Yet, there
are precise local estimates available on the size of the set of points where
the density ratio is small, see \cite[\S 2]{snulmenn.isoperimetric}.

In order to put the study of regularity properties into perspective, it is
instructive to consider as well the behaviour of the Laplace operator.
\begin{citing} [Model case]
    Suppose $\vdim$ and $\adim$ are positive integers, $\vdim < \adim$, $1
    \leq p \leq \infty$, $u \in \Lploc 1 ( \mathscr{L}^\vdim, \rel^\codim )$,
    the distributional Laplacian $T \in \mathscr{D}' ( \rel^\vdim, \rel^\codim
    )$ of $u$, defined by $T ( \phi ) = \tint{}{} u \bullet \Lap \phi \ud
    \mathscr{L}^\vdim$ for $\phi \in \mathscr{D} ( \rel^\vdim, \rel^\codim
    )$, is representable by integration, and if $p > 1$ then there exists $f
    \in \Lploc{p} ( \| V \|, \rel^\codim )$ satisfying
    \begin{gather*}
    	T ( \phi ) = \tint{}{} f(x) \bullet \phi(x) \ud \mathscr{L}^\vdim
	x \quad \text{for $\phi \in \mathscr{D} ( U, \rel^\codim )$}.
    \end{gather*}
\end{citing}
If $1 < p < \infty$ then $u$ is twice weakly differentiable and the
distributional partial derivatives up to second order of $u$ correspond to
functions in $\Lploc p ( \mathscr{L}^\vdim, \rel^\codim )$; this is
a~consequence of the usual a priori estimate based on the Calder{\'o}n Zygmund
inequality, see e.g.~\cite[3.5]{snulmenn.c2}, and convolution.

This implies differentiability results in Lebesgue spaces for the weak
derivative, see for instance Calder{\'o}n and Zygmund \cite[Theorem~12,
p.~204]{MR0136849}.

For an integral varifold $V$, the existence of a notion of first order
derivative, i.e.,~that $\tau$ is defined $\| V \|$ almost everywhere,  is a
simple consequence of its rectifiability, see Allard
\cite[3.5\,(1)]{MR0307015}.  However, this derivative behaves rather like an
approximate derivative than a weak derivative as is exemplified by the fact
that a Poincar{\'e} inequality only holds under additional hypotheses on the
first variation, see \cite[p.~372]{snulmenn.poincare}. More information on the
validity of Sobolev Poincar{\'e} type inequalities may be retrieved from
Hutchinson \cite{MR1066398}, \cite[\S4]{snulmenn.poincare} and
\cite[\S10]{snulmenn:tv.v2}.

As the Grassmann manifold is compact, $\tau$ belongs to $\Lploc q ( \| V \|,
\Hom ( \rel^\adim, \rel^\adim ) )$ for $1 \leq q \leq \infty$. Yet, it is
important to understand for which $q$ effective coercive estimates are
available. Classically, this is the case for $q = 2$, see Allard
\cite[8.13]{MR0307015} and its refinements Brakke \cite[5.5]{MR485012} and
\cite[4.10, 4.14]{snulmenn.decay}.
\begin{citing} [Question 1]
    Suppose $2 < q < \infty$ and $1 < p < \infty$.

    Do the general hypotheses imply that for $\|V\|$ almost all $c$ there
    exists $1 \leq \gamma < \infty$ such that there holds
    \begin{multline*}
	\big ( r^{-\vdim} \tint{\cball c{r/2}}{} | \tau(z)-\tau(c) |^q \ud \|
	V \| z \big )^{1/q} \\
	\leq \gamma \Big ( \big ( r^{-\vdim} \tint{\cball cr}{} |\tau
	(z)-\tau(c)|^2 \ud \| V \| z \big )^{1/2} + \big ( r^{p-\vdim}
	\tint{\cball cr}{} | \mathbf{h} (V,z)|^p \ud \| V \| z \big )^{1/p}
	\Big )
    \end{multline*}
    whenever $0 < r \leq \gamma^{-1}$? A similar question may be phrased for
    $q = \infty$ or $p \in \{ 1, \infty \}$.
\end{citing}
In case $\vdim \geq 2$ and $q = \infty$ such estimates are known to fail by
Brakke \cite[6.1]{MR485012}. In view of \ref{remark:holes-tilt-large} the
Question 1 is related to the possible size of the exceptional sets occurring
in approximations by $\qspace_Q ( \rel^\codim )$ valued functions. It is also
related to the second ``main analytic estimate'' -- so termed in Almgren's
announcement \cite[p.~6]{MR574247} -- of Almgren's regularity proof for area
minimising currents of arbitrary codimension, see Almgren \cite[3.29,
3.30]{MR1777737} and De~Lellis and Spadaro \cite[Theorem~7.1]{MR3283929}.

Passing from first order to second order quantities, the analogous property to
weak differentiability of the weak derivative of $u$ would be generalised $V$
weak differentiability of~$\tau$, or equivalently, $V$ being a curvature
varifold in the sense of Hutchinson, see \ref{def:curvature_varifold} and
\ref{remark:curvature_varifold}. Considering three half lines emanating from
the origin in $\rel^2$ at equal angles shows that even a stationary integral
varifold need not to be a curvature varifold, see Mantegazza \cite[3.4,
3.11]{MR1412686}. In view of \cite[4.8]{snulmenn.c2}, one may however still
define a notion of approximate second fundamental form, $\ap \mathbf{b}
(V,\cdot)$, of a varifold satisfying the general hypotheses such that
\begin{gather*}
    \ap \mathbf{b} (V,z) = \mathbf{b} (M,z) \quad \text{for $\| V \|$ almost
    all $z \in U \cap M$}
\end{gather*}
whenever $M$ is an $\vdim$ dimensional submanifold of $\rel^\adim$ of class
$2$. Since the corresponding approximate mean curvature vector is $\|V \|$
almost equal to $\mathbf{h} (V,\cdot)$ by \cite[4.8]{snulmenn.c2}, we assume
$\vdim \geq 2$ in the following question.
\begin{citing} [Question~2]
    Suppose $\vdim \geq 2$, $p = \infty$, and $0 < q < \infty$.

    Do the general hypotheses imply that for $\| V \|$ almost all $c$
    \begin{gather*}
        \tint{\cball cr}{} \| \ap \mathbf{b} (V,z) \|^q \ud \| V \| z < \infty
        \quad \text{for some $r > 0$}?
    \end{gather*}
\end{citing}
The existence proof in \cite[4.8]{snulmenn.c2} does not provide any integral
estimates of the resulting quantity and, in fact, no positive results are
known. Considering the scaling behaviour of the above integral, the example in
Brakke~\cite[6.1]{MR485012} shows that the answer is in the negative whenever $q
\geq 2$.

In case $V$ happens to be a curvature varifold, one may deduce
differentiability results for $\tau$ in Lebesgue spaces from general facts
about generalised $V$ weakly differentiable functions, see \cite[11.4,
15.9--15.12]{snulmenn:tv.v2}. The next question concerns to which extent these
properties persist for non-curvature varifolds.
\begin{citing} [Question 3]
    Suppose $0 < \alpha \leq 1$ and $2 \leq q < \infty$.

    Do the general hypotheses imply that
    \begin{gather*}
        \limsup_{r \to 0+} r^{-\alpha} \big ( r^{-\vdim} \tint{\cball cr}{} |
        \tau (z)- \tau (c) |^q \ud \| V \| z \big )^{1/q} < \infty
    \end{gather*}
    for $\| V \|$ almost all $c$? A similar question may be phrased for $q =
    \infty$.
\end{citing}
This may be seen as a pointwise H{\"o}lder condition with exponent $\alpha$ on
$\tau$ at~$c$ measured in a Lebesgue space with exponent $q$. If $p > \vdim$
and $p \geq 2$, proving uniform estimates for $$r^{-\alpha} \big ( r^{-\vdim}
\tint{\cball cr}{} | \tau (z) - \tau (c) |^2 \ud \| V \| z \big )^{1/2}$$ for
all $c$ in some relatively open subset of $\spt \| V \|$ and all $0 <r \leq
\varepsilon$ for some $\varepsilon > 0$ is -- via H{\"o}lder continuity of
$\tau$ -- the key to Allard's regularity theorem, see \cite[\S 8]{MR0307015}.
If $\vdim \geq 2$ and $p = \infty$, then $\tau$ may be discontinuous at points
in a set of positive $\| V \|$ measure and $\spt \| V \|$ may fail to be
associated to an $\rel^\codim$ valued or even $\qspace_Q ( \rel^\codim )$
valued function near those points by Brakke \cite[6.1]{MR485012}. Therefore,
the condition in Question~3 also acts as replacement for more classical
notions of regularity which are known to possibly fail on a set of positive
$\| V \|$ measure.

Evidently, if $\vdim \geq 2$ and $q = \infty$ the answer is in the negative
whenever $1 \leq p \leq \infty$ by Brakke \cite[6.1]{MR485012}. If $p < \vdim$,
the answer is in the negative whenever $\alpha q > \vdim p/(\vdim-p)$,
see~\cite[1.2\,(iv)]{snulmenn.isoperimetric}.

Turning to positive results, only the case $q=2$ has been systematically
studied. The sharpest known results were obtained in \cite{snulmenn.c2} building
partly on \cite{snulmenn.decay} and extending methods and results of Brakke
\cite[5.5, 5.7]{MR485012} and Sch\"atzle \cite[5.4]{MR1906780},
\cite[Theorem~5.1]{MR2064971}, and \cite[Theorem~3.1]{MR2472179}. Namely, if
$\vdim = 1$ or $\vdim = 2$ and $p > 1$ or $\vdim > 2$ and $p \geq 2
\vdim/(\vdim+2)$, then the answer is in the affirmative for $\alpha = 1$ and $q
= 2$, see \cite[5.2\,(2)]{snulmenn.c2}, and if $\vdim =2$ and $\alpha < 1$ or
$\sup \{ 2,p \} < \vdim$ and $\alpha = \vdim p/(2(\vdim-p)) < 1$ then the answer
is in the affirmative for $q = 2$, see \cite[5.2\,(1)]{snulmenn.c2}.

Therefore, the previously known results for Question 3 may be summarised as
follows. If $q = 2$, only the case $(\vdim,p,\alpha) = (2,1,1)$ was left open.
If $2 < q < \infty$, the gap between positive results and known counterexamples
was quite large. And if $q = \infty$, only the case $\vdim = 1$ remained
open. The initial motivation for the present work was to solve some of these
open cases of Question 3.

Finally, notice that Question~3 could be phrased for $1 \leq q < 2$ as well,
see \cite[p.~248, Problem (ii)]{snulmenn.isoperimetric}, and Question~2 could
include the case $p < \infty$. However, no effort has been  made to resolve
these additional cases in the present study.


\subsection*{Results of the present article} 
The results may be summarised as follows. All cases of Question~1 are answered;
in the negative if $\vdim \geq 2$ and in the affirmative if $\vdim = 1$, see
\ref{remark:no-gehring-improvement} and \ref{remark:sup-L1-est}. Question~2 is
answered in the negative if $q > 1$, see
\ref{remark:example-super-quadratic-tilt}. All cases of Question~3 except the
case $(\vdim,p,\alpha) = (2,1,2/q)$ for $2 < q < \infty$ are treated, see
\ref{thm:positive_result},
\ref{remark:sharpness_superquadratric}--\ref{remark:m=2,p=1}, and
\ref{thm:decay-1d}. Finally, for the case $(\vdim,p,q) = (2,1,2)$ of Question~3,
the precise decay rate is determined, see \ref{example:quadratic_tilt_excess}
and \ref{theorem:quadratic-tilt-decay}.

Beginning with the last item, the following theorem is established.
\begin{citing} [Theorem~A, see \ref{theorem:quadratic-tilt-decay}]
    Suppose $\vdim = 2$ and $p = 1$.  

    If $\adim$, $U$, $V$, and $\tau$ satisfy the conditions of the general
    hypotheses, then
    \begin{gather*}
	\lim_{r \to 0+} r^{-4} ( \log (1/r) )^{-1} \tint{ \cball cr}{} | \tau
	(z) - \tau (c) |^2 \ud \| V \| z = 0
    \end{gather*}
    for $\| V \|$ almost all $c$.
\end{citing}
This result is sharply complemented by the following negative result.
\begin{citing} [Theorem~B, see \ref{example:quadratic_tilt_excess}]
    Suppose $\vdim = 2$, $\adim = 3$, $p = 1$, $U = \rel^3$, and $\omega$ is a
    modulus of continuity.

    Then there exist $V$ and $\tau$ satisfying the conditions of the general
    hypotheses and $C$ with $\| V \| (C) > 0$ satisfying
    \begin{gather*}
	\limsup_{r \to 0+} \omega (r)^{-1} r^{-4} ( \log (1/r) )^{-1} \tint{
	\cball cr}{} | \tau (z) - \tau (c) |^2 \ud \| V \| z > 0
    \end{gather*}
    whenever $c \in C$.
\end{citing}
In fact, the varifold constructed in Theorem~B may be chosen to be associated to
the graph of a Lipschitzian function with small Lipschitz constant. Theorem~B in
particular answers the case $(\vdim,p,\alpha) = (2,1,1)$ of the part $q = 2$
of Question 3 in the negative.

\begin{citing} [Theorem~C, see \ref{example:quantitative_brakke_again} and
\ref{remark:second_fundamental_form_large}]
    Suppose $\vdim \geq 2$, $\adim = \vdim+1$, $p = \infty$, $U =
    \rel^\adim$, and $\omega$ is a modulus of continuity.
    
    Then there exist a curvature varifold $V \in \IVar_\vdim ( \rel^\adim )$ and
    $\tau$ satisfying the conditions of the general hypotheses, an $\vdim$
    dimensional submanifold $M$ of $\rel^\adim$ of class $\infty$ which is
    relatively open in $\spt \| V \|$, $\varepsilon > 0$, $B \subset \{ r \with
    r > 0 \}$ with $\inf B = 0$, and $C$ with $\| V \| (C)>0$ such that the
    following properties hold whenever $c \in C$.
    \begin{enumerate}
	\item \label{item:brakke:tilt} If $r \in B$, then $$\| V \| ( \cball
        cr \cap \{ z \with \| \tau(z)-\tau(c)\| \geq 1/3 \} ) \geq \omega (r)
        r^{\vdim+2} ( \log (1/r) )^{-2}.$$
	\item \label{item:brakke:holes} If $r \in B$ and $T = \im \tau (c)$,
        then $$\mathscr{H}^\vdim ( H(T,c,r) ) \geq 
        \omega(r) r^{\vdim+2} ( \log (1/r))^{-2},$$ where $H (T,c,r) = T \cap
        \cball{\project{T}(c)}{r} \without \project{T} \lIm \cylinder Tcrr
        \cap \spt \| V \| \rIm$.
    \item \label{item:brakke:2nd_fundamental_form} If $r > 0$ and $1 < q < \infty$,
        then
        $$\tint{M \cap \cball cr}{} | \mathbf{b} ( M,z) |^q \ud
        \mathscr{H}^\vdim z = \infty.$$
	\item \label{item:brakke:density} If $0 < r \leq \varepsilon$, then
        $$\| V \| ( \cball cr \cap \{ z \with \density^\vdim ( \| V \| , z )
        \leq \density^\vdim ( \| V \|, c) - 1 \} ) \geq \omega (r) r^\vdim.$$
    \end{enumerate}
\end{citing}
If $\omega$ satisfies the Dini condition and ``$\omega(r)$'' is replaced by
``$\omega(r)^2$'' in parts \eqref{item:brakke:tilt} and
\eqref{item:brakke:holes}, then one may take $B = \{ r \with 0 < r \leq
\varepsilon \}$, see \ref{example:quantitative_brakke} and
\ref{remark:example-super-quadratic-tilt}.

Part \eqref{item:brakke:tilt} answers Question~3 in the negative whenever $q >
2$ and $\alpha > 2/q$. Notice that the $\| V \|$ measure of sets of the form
considered in part \eqref{item:brakke:tilt} occurs as upper bound on the size of
the exceptional sets of the usual approximation by $\qspace_Q ( \rel^\codim)$
valued functions, \ref{remark:holes-tilt-large}. In part
\eqref{item:brakke:holes} we provide a lower bound on the size of ``holes'' of
the varifold. As these regions will always be part of one of the exceptional
sets of any approximation by $\qspace_Q (\rel^\codim)$ valued functions, this
yields a corresponding lower bound on the size of this set. For any $\delta >
0$, this lower bound also rules out an upper bound on the size of the
exceptional sets by a suitable multiple of
\begin{gather*}
	E^{1+\delta}, \quad \text{where $E = r^{-\vdim} \tint{\cball cr}{} |
	\tau (z)-\tau(c) |^2 \ud \| V \| z$}.
\end{gather*}
Such upper bound would be the natural analogue for varifolds satisfying the
general hypotheses with $\vdim \geq 2$ and $p =\infty$ of the aforementioned
second ``main analytic estimate'' for area minimising currents obtained by
Almgren in \cite[3.29\,(8)]{MR1777737}. Of course, our example does not
(obviously) preclude a possible bound involving $E ( \log (1/E))^{-2}$, for $E
> 0$, in place of $E^{1+\delta}$, or the original bound $E^{1+\delta}$ under
the additional hypothesis of stationarity.

Evidently, part \eqref{item:brakke:2nd_fundamental_form} provides a negative
answer to Question~2 whenever $q > 1$; in fact, even if the integral in
question is restricted to the ``regular part'' of~$V$. Part
\eqref{item:brakke:density} contains the observation that the regions of
significantly smaller density around a given point may be as large as the
approximate continuity of the density permits, see
\ref{remark:density-ap-lsc}.

If $\vdim \geq 2$, part \eqref{item:brakke:tilt} of Theorem~C leaves little
room for positive answers to Question~3 for $q > 2$ except of those which
follow from the boundedness of $\tau$ and the known positive results for $q
=2$, see \ref{thm:positive_result} and
\ref{remark:sharpness_superquadratric}--\ref{remark:m=2,p=1}. In fact, only
the case $(\vdim,p,\alpha) = (2,1,2/q)$ for $2 < q < \infty$ remains open.
This case is related to the question of availability of estimates of $\tau$ in
certain Lorentz spaces, see \ref{remark:m=2,p=1}.

The combination of positive and negative answers to Question~3 obtained so far
in particular implies that the answer to Question~1 is in the negative whenever
$\vdim \geq 2$, see~\ref{remark:no-gehring-improvement}.

Turning to the case $\vdim = 1$, a positive answer to Question~1 follows from
\ref{remark:sup-L1-est}.
\begin{citing} [Theorem~D, see \ref{thm:decay-1d} and
\ref{corollary:1-dim-Q-graph}]
    Suppose $\vdim = 1$ and $p =1$.

    If $\adim$, $U$, $V$, $Z$, and $\tau$ satisfy the general conditions, then
    there exists a subset $A$ of $Z$ with $\| V || ( A \without Z ) = 0$ such
    that the following two statements hold for $\| V \|$ almost all $z \in A$.
    \begin{enumerate}
	\item \label{item:1d:diff} The map $\tau | A$ is differentiable at~$z$
	relative to $A$ and
        \begin{gather*}
            \Der  ( \tau | A ) (z) = ( \| V \|, 1 ) \ap \Der  \tau (z).
        \end{gather*}
    \item \label{item:1d:representation} If $\varepsilon > 0$ then there exist a
        positive integer $Q$, $0 <r< \infty$, and $f_i : T \cap \cball{\project{T}
          (z)}{r} \to T^\perp \cap \cball{ \perpproject{T} (z)}{r}$ with $\Lip f_i
        \leq \varepsilon$ for $i = 1, \ldots, Q$ and
    	\begin{gather*}
	    \density^1 ( \| V\|, \zeta ) = \card \big \{ i \with f_i (
	    \project{T} ( \zeta ) ) = \perpproject{T} (\zeta) \big \}
        \end{gather*}
        for $\mathscr{H}^1$ almost all $\zeta \in \cylinder Tzrr$.
    \end{enumerate}
\end{citing}
Part \eqref{item:1d:diff} includes a positive answer to Question~3 for $\vdim
= 1$, $\alpha = 1$, and $q = \infty$. Part \eqref{item:1d:representation}
expresses the varifold as finite sum of graphs of Lipschitzian functions.


\subsection*{Outline of the proofs}

\subsubsection*{Theorem~A}
In order to explain the proof of Theorem~A, it is instructive to recall the
strategy of proof of similar results in \cite[5.2]{snulmenn.c2}. The proof rests
on an idea of Sch{\"a}tzle underlying \cite[Theorem 3.1]{MR2064971}. He realised
that in the presence of a coercive estimate, see Brakke \cite[5.5]{MR485012},
second order rectifiability of a varifold satisfying the general hypotheses with
$p = 2$ implies second order behaviour of the quadratic tilt-excess and the
quadratic height-excess. Having second order rectifiability at one's disposal
from \cite[4.8]{snulmenn.c2}, this procedure has been employed in
\cite[5.2\,(2)]{snulmenn.c2} with a refined coercive estimate.  Specifically,
Brakke's estimate \cite[5.5]{MR485012} was improved by introducing height-excess
quantities measured in $\Lp q ( \| V \| )$ with $q = \frac{2 \vdim}{\vdim-2}$ if
$\vdim > 2$, see \cite[4.10, 4.11]{snulmenn.decay}.  Employing an approximation
by $\qspace_Q ( \rel^\codim)$ valued functions and interpolation in \cite[5.7,
6.4]{snulmenn.decay} yielded~a coercive estimate involving approximate
height-excess quantities, that is height-excess that excludes an arbitrary set
of small $\| V \|$ measure, see \cite[9.5]{snulmenn.decay}. Finally, in all
cases in which this method implied a positive answer to Question~3 for $q=2$ and
$\alpha < 1$, the differentiation theory from \cite[3.7]{snulmenn.isoperimetric}
was employed in conjunction again with the second order rectifiability to
establish that the limit in question is actually equal zero $\| V \|$ almost
everywhere.

The results obtained by the above method in case $\vdim =2$ were not sharp as
the coercive estimate \cite[4.10]{snulmenn.decay} did only use Lebesgue spaces
so that subsequently the non-sharp embedding of weakly differentiable functions
with square-integrable weak derivative on $\oball ar$, for $a \in \rel^2$ and $0
< r < \infty$, into $\Lp q ( \mathscr{L}^2 \restrict \oball ar)$ for $q <
\infty$ needed to be employed. To be able to employ the sharp embedding into
Orlicz spaces, see \cite[8.27, 8.28]{MR2424078}, we therefore modify the above
procedure. In~particular, we obtain a coercive estimate involving the
appropriate Orlicz space in \ref{lemma:coercive-estimate} which takes the role
of \cite[4.10]{snulmenn.decay}. Moreover, to be able to proceed after obtaining
a weaker form of Theorem~A which results from replacing ``$\lim$'' and ``$=0$''
by ``$\limsup$'' and ``$<\infty$'' in its statement, the differentiation result
\cite[3.7]{snulmenn.isoperimetric} is adapted in \ref{thm:O_o} to include rates
which are not powers. It should also be noted that as the current proof does not
aim at decay \emph{estimates}, i.e.,~estimates for positive radii, such as
\cite[10.2]{snulmenn.decay} but only at decay \emph{rates}, i.e.,~the behaviour
as the radius approaches zero, we are able to encapsulate the usage of the
approximation by $\qspace_Q ( \rel^\codim )$ valued functions in the
construction of a real valued auxiliary function in
\ref{lemma:mini-lip-approx}. In comparison to \cite[9.5]{snulmenn.decay} where
the coercive estimate involving the approximate height-excess quantities was
derived as corollary to the more elaborate estimates in
\cite[9.4\,(9)]{snulmenn.decay} aiming at the proof of
\cite[10.2]{snulmenn.decay}, this greatly simplifies our derivation of the
corresponding estimates in \ref{theorem:quadratic-tilt-decay}.

\subsubsection*{Theorems B and C}
The qualitative construction principle for both theorems is that of Brakke
\cite[6.1]{MR485012}. The basis for obtaining quantitative information is
provided by the following variant of classical propositions, see
\ref{example:yet_another_cantor_set}.
\emph{If $\vdim$ is a positive integer, $0 < \lambda < 1$, and $\omega$ is a
  modulus of continuity, then there exist a countable disjointed family $G$ of
  cubes contained in the unit cube $C$ of $\rel^\vdim$, $B \subset \rel \cap \{
  r \with r > 0 \}$ with $\inf B = 0$, and $\varepsilon > 0$ such that
  $\mathscr{L}^\vdim$ almost all $a \in C \without \bigcup G$ have the following
  property: If $0 < r \leq \varepsilon$, then there exists a subfamily $H$ of
  $G$ with $\bigcup H \subset \cball ar$,
  \begin{gather*}
      \mathscr{L}^\vdim ( \textstyle{\bigcup H} ) \geq \omega (r) r^\vdim
      \quad \text{and} \quad
      \text{if $r \in B$ then $\card H = 1$.}
  \end{gather*}} Moreover, $B$ may be required to equal $\{ r \with 0 < r \leq
\varepsilon \}$ if and only if $\omega$ satisfies the Dini condition, see
\ref{example:cantor_set} and \ref{remark:topsoe}.  Denoting by $F$ the
collection of balls with centres and ``radii'' equal to those of the cubes in
$G$ and fixing an isometric injection of $\rel^\vdim$ into $\rel^{\vdim+1}$, we
associate to each member of $F$ an $\vdim$~dimensional submanifold
of~$\rel^{\vdim+1}$ of class~$1$ involving a~piece of a~catenoid, see
Figures~\ref{F:catenoid} and~\ref{F:bent-catenoid} on page~\pageref{F:catenoid}
and~\pageref{F:bent-catenoid}. The varifolds whose existence is asserted in
Theorems~B and~C then consist of the submanifolds corresponding to the members
of $F$ together with the image of $\rel^\vdim \without \bigcup F$ under the
injection [taken with multiplicity two in case of Theorem~C].

\subsubsection*{Theorem D} The key to the proof of Theorem~D is an a priori
estimate for ``weak subsolutions to Poisson's equation for the
Laplace-Beltrami operator on $V$'', see \ref{lemma:sup-L1-est} and
\ref{remark:laplace-beltrami}. This estimate is adapted from Allard and
Almgren, see~\cite[5\,(6)]{MR0425741}, and implies in particular a positive
answer to Question~1 if $\vdim = 1$ and $q = \infty$, see
\ref{remark:sup-L1-est}. Since a positive answer to Question~3 for $\vdim = 1$
and $q = 2$ is already known from \cite[5.2\,(2)]{snulmenn.c2}, part
\eqref{item:1d:diff} is now a consequence of suitable differentiation result,
see \ref{thm:differentiation}, which in turn is based on
\cite[3.1]{snulmenn.isoperimetric}. Part \eqref{item:1d:representation} then
follows by a suitable approximation by $\qspace_Q ( \rel^{\adim-1} )$ valued
functions, see \cite[3.15]{snulmenn.poincare}, and a selection theorem for
such function, see Almgren \cite[1.10]{MR1777737}.

\subsubsection*{Curvature varifolds} Our treatment of curvature varifolds
makes use of generalised weakly differentiable functions introduced in
\cite[8.3]{snulmenn:tv.v2}, in particular their differentiability properties,
see \cite[11.2, 11.4]{snulmenn:tv.v2}. This approach rests on the
characterisation of curvature varifolds in terms of such functions,
see~\cite[15.6]{snulmenn:tv.v2}. Otherwise, only some more elementary facts
are cited from~\cite{snulmenn:tv.v2}.

\subsection*{An open problem}
The cases of Question~3 remaining open are related to the following question
concerning the decay behaviour of tilt-excess quantities measures in the Lorentz
space with exponent $(2,\infty)$, see~\ref{remark:m=2,p=1}.
\begin{citing} [Question~4]
    Suppose $\vdim = 2$ and $p = 1$.

    Do the general hypotheses imply that
    \begin{gather*}
        \limsup_{r \to 0+} r^{-1} E(c,r) < \infty
        \quad
        \text{for $\| V \|$ almost all $c$},
    \end{gather*}
    where $E(c,r) = \sup \{ t r^{-\vdim/2} \| V \| ( \cball cr \cap \{ z \with |
    \tau (z)-\tau(c) | > t \} )^{1/2} \with 0 < t < \infty \}$?
\end{citing}
If the answer should turn out to be in the affirmative, it would imply Theorem~A
as a corollary as well as a positive answer to Question~3 in the remaining cases
$\alpha = 2/q$ for $2 < q < \infty$ both by interpolation based on the
boundedness of $\tau$. For the model case of the Laplace operator the procedure
used to prove Theorem~A could be adapted since coercive estimates in Lorentz
spaces are available for the Laplace operator; such estimates are reviewed for
instance in Dolzmann and M{\"u}ller~\cite[\S 3]{MR1354111}.

\subsection*{Organisation of the paper} 
In Section \ref{sec:notation} the notation is introduced. Sections
\ref{sec:lebesgue}--\ref{sec:differentiation} are of preparatory nature
supplying the density properties of the Lebesgue measure, properties of
Cartesian products of varifolds, and some differentiation theory for functions
on varifolds respectively. Sections
\ref{sec:qte-ex}--\ref{sec:quadratic_tilt_decay} treat the quadratic tilt-excess
whereas Sections \ref{sec:super_quadratic_tilt}--\ref{sec:sqte-decay} cover the
super-quadratic case. This includes, in Section \ref{sec:super_quadratic_tilt},
the example concerning the integrability of the second fundamental form of
curvature varifolds. Finally, one dimensional varifolds are considered in
Section \ref{sec:one_dimensional_decay}.

\subsection*{Acknowledgement}
The first author was partially supported by the Foundation for Polish Science
and partially by NCN Grant no.  2013/10/M/ST1/00416 while he was on leave from
Institute of Mathematics of the University of Warsaw.
 
It is a~pleasure to the authors to express their gratitude to Prof.~Guy David
for his valuable input to this paper and to thank Prof.~David Preiss for
informing them about the studies of Mejlbro and Tops{\o}e in \cite{MR0583557}
and Tops{\o}e in \cite{MR577961}.

\emph{AEI publication number:} AEI-2015-007


\section{Notation} \label{sec:notation}

The notation generally follows Federer \cite{MR41:1976} and Allard
\cite{MR0307015} with some modifications and additions described in \cite[\S
1]{snulmenn:tv.v2}. Here, concerning the paragraph ``Definitions in the text'' of
\cite[\S 1]{snulmenn:tv.v2}, only the notions of \emph{generalised weakly
differentiable function}, \emph{generalised weak derivative}, $\derivative
Vf$, and the associated space $\trunc (V,Y)$, see \cite[8.3]{snulmenn:tv.v2},
will be employed.

\paragraph{Less common symbols and terminology.} A family is disjointed if and
only if two distinct members have empty intersection, see
\cite[2.1.6]{MR41:1976}. Whenever $f$ is a function it is a subset of the
Cartesian product of its domain, $\dmn f$, and its image, $\im f$, see
\cite[p.~669]{MR41:1976}. The Lebesgue measure of the unit ball in
$\rel^\vdim$ will be denoted by $\unitmeasure \vdim$, see
\cite[2.7.16]{MR41:1976}. Whenever $\mu$ measures $X$ and $Y$ is a Banach
space, $\mathbf{A} ( \mu, Y)$ denotes the vectorspace of those $Y$ valued
$\mu$ measurable functions $f$ such that there exists a separable subspace
$Z$ of $Y$ satisfying $\mu ( f^{-1} \lIm Y \without Z \rIm ) = 0$, see
\cite[2.3.8]{MR41:1976}. In this case,
\begin{gather*}
    \Lpnorm{\mu}{p}{f} = \big ( \tint{}{} |f|^p \ud \mu \big )^{1/p} \quad
    \text{whenever $1 \leq p < \infty$}, \\
    \Lpnorm{\mu}{\infty}{f} = \inf \{ t \with \mu ( \{ x \with |f(x)| > t
    \} ) = 0 \}
\end{gather*}
whenever $f \in \mathbf{A} ( \mu, Y )$, see \cite[2.4.12]{MR41:1976}. Notice
also that no usage of equivalence classes of functions is made.

\paragraph{Modifications.} Extending Allard \cite[2.5\,(1)]{MR0307015},
whenever $M$ is a submanifold of $\rel^\adim$ of class~$2$ and $a \in M$ the
\emph{second fundamental form of $M$ at $a$} is the unique symmetric bilinear
mapping
\begin{gather*}
	\mathbf{b}(M,a) : \Tan (M,a) \times \Tan (M,a) \to \Nor (M,a)
\end{gather*}
such that whenever $g : M \to \rel^\adim$ is of class $1$ and $g(z) \in \Nor
(M,z)$ for $z \in M$,
\begin{gather*}
    u \bullet \langle v, \Der  g (a) \rangle = - \mathbf{b} (M,a) ( u,v )
    \bullet g (a) \quad \text{for $u,v \in \Tan (M,a)$}.
\end{gather*}

\paragraph{Additional notation.} By a~\emph{modulus of continuity} we mean
a~function $\omega : \{ t \with 0 \leq t \leq 1 \} \to \{ t \with 0 \leq t \leq
1 \}$ such that
\begin{gather*}
    \lim_{t \to 0+} \omega (t) = 0, \qquad \text{$\omega (t) = 0$ if and only
      if $t=0$ whenever $0 \leq t \leq 1$},
    \\
    \text{$\omega(s) \leq \omega(t)$ whenever $0 \leq s \leq t \leq 1$}.
\end{gather*}
Adapting Morrey \cite[p.~54]{MR0202511}, such $\omega$ is said to satisfy the
\emph{Dini condition} if and only if $\tint{0}{1} \omega(t)t^{-1} \ud \LM^1 t <
\infty$.  Following \cite[p.~8]{snulmenn.decay}, whenever $\adim,\vdim \in
\nat$, $\vdim < \adim$, $T \in \grass{\adim}{\vdim}$, $a \in \rel^\adim$, $0 < r
< \infty$, and $0 < h < \infty$ we define the \emph{closed cuboid} by
\begin{gather*}
    \cylinder Tarh = \rel^\adim \cap \big \{ z \with |\project T (z-a)| \le r
    \text{ and } |\perpproject T (z-a)| \le h \big \}  .
\end{gather*}

\paragraph{Definitions in the text.} The notions of \emph{curvature varifold}
and its \emph{second fundamental form} $\mathbf{b}(V,z)$, are explained in
\ref{def:curvature_varifold}. The Orlicz space seminorm $\mu_{(\Phi)}$ is
defined in \ref{def:orlicz}. The space $\qspace_Q ( Y )$ metrised by
$\mathscr{G}$ occurs in \ref{miniremark:almgren1}, the notion of \emph{affine}
$\qspace_Q ( Y )$ valued function and the corresponding seminorm $\| f \|$ are
explained in \ref{miniremark:almgren2}. Finally, the notions of \emph{affinely
approximable} and \emph{approximately affinely approximable} for $\qspace_Q (
Y)$ valued functions and the corresponding symbols $\Aff f$ and $\ap \Aff f$
are defined in \ref{miniremark:almgren2}.


\section{Density properties of Lebesgue measure} \label{sec:lebesgue} The
purpose of this section is to provide two examples of subsets of positive
Lebesgue measure whose complement is as large as possible near each point of
the set, see \ref{example:cantor_set} and
\ref{example:yet_another_cantor_set}. In this respect the size of the
complement is measured either by the Lebesgue measure or by the behaviour of
the distance function on the complement (equivalently by the size of cubes
contained in the complement). In the light of known positive results, the
examples obtained are sharp, see \ref{remark:topsoe} and
\ref{remark:vitali_lebesgue}. Using these examples, certain varifolds will be
constructed in Sections~\ref{sec:qte-ex} and~\ref{sec:super_quadratic_tilt} to
demonstrate that the tilt-excess decays proven in
Sections~\ref{sec:quadratic_tilt_decay} and~\ref{sec:sqte-decay} are sharp in
many cases.

Both sets are constructed by removing a suitable disjointed subcollection of
the family of all dyadic subcubes from the unit cube.

\begin{miniremark}
    \label{miniremark:family_of_cubes}
    Here, we collect some useful terminology for later reference.

    Suppose $\vdim$ is a positive integer. Define \emph{open cubes} by
    \begin{gather*}
        C(a,r) = \classification{\rel^\vdim}{x}{ \text{$a_j < x_j <
        a_j+r$ for $j=1, \ldots, \vdim$} }
    \end{gather*}
    whenever $a \in \rel^\vdim$ and $0 < r < \infty$. Note that $a \notin
    C(a,r)$. We shall work with the following definition of dyadic cubes.
    Whenever $i$ is a~nonnegative integer define $W(i)$ to consist of all open
    cubes
    \begin{gather*}
        C ( 2^{-i} a, 2^{-i} ) \subset C(0,1)
    \end{gather*}
    corresponding to $a \in \integers^\vdim$ with $0 \leq a_j \leq 2^i-1$
    for $j = 1, \ldots, \vdim$. Let
    \begin{gather*}
        Z = {\textstyle\bigcup} \{ W(i) \with i = 0, 1, 2, \ldots \},
        \quad
        N = {\textstyle \bigcup_{i=1}^\infty \big ( ( \Clos C(0,1) ) \without \bigcup W(i) \big )}.
    \end{gather*}
    Note that $W(i)$ is disjointed and $W(i) \cap W(j) = \varnothing$ if
    $i \neq j$ and $\LM^\vdim (N)=0$. Observe
    \begin{gather*}
        \text{either $Q \subset R$ or $R \subset Q$} \quad
        \text{whenever $Q,R \in Z$ and $Q \cap R \neq \varnothing$}.
    \end{gather*}
\end{miniremark}
\begin{example}
    \label{example:cantor_set}
    Suppose $\vdim$, $W$, $Z$, and $N$ are as in
    \ref{miniremark:family_of_cubes}, $\omega$ is a modulus of
    continuity satisfying the Dini condition, and $0 \leq \lambda < 1$.

    Then there exist $0 < \varepsilon \leq 1$ and a disjointed subfamily
    $G$ of $Z$ such that
    \begin{gather*}
        A = {\textstyle C(0,1) \without ( N \cup \bigcup G)}
    \end{gather*}
    satisfies the following two conditions:
    \begin{enumerate}
        \item \label{item:cantor_set:measure} $\LM^\vdim (A)
        \geq \lambda$.
        \item \label{item:cantor_set:holes} If $a \in A$ and $0 < r
        \leq \varepsilon$, then there exists $Q \in G$ with
        \begin{gather*}
             Q \subset \oball{a}{r} \quad \text{and} \quad
             \LM^\vdim (Q) \geq \omega (r) r^\vdim.
        \end{gather*}
    \end{enumerate}
\end{example}
\begin{proof} [Construction]
    Choose $0 < s \leq 2^{-1} \vdim^{-1/2}$ such that $\omega ( 2
    \vdim^{1/2} s) \leq 2^{-2\vdim} \vdim^{-\vdim/2}$, define $\phi :
    \classification{\rel}{t}{ 0 \leq t \leq 1 } \to \rel$ by
    \begin{gather*}
        \phi ( t ) = 2^{2\vdim} \vdim^{\vdim/2} \omega ( 2 \vdim^{1/2}
        t ) \quad \text{for $0 \leq t \leq s$}, \qquad \phi ( t ) = 1
        \quad \text{for $s < t \leq 1$},
    \end{gather*}
    and note that $\phi$ is a modulus of continuity satisfying the Dini
    condition. Choose a positive integer $k$ such that
    \begin{gather*}
        2^{-k} \leq s, \quad ( \log 2 )^{-1} \tint{0}{2^{1-k}} \phi
        (t) t^{-1} \ud \LM^1 t \leq 1-\lambda.
    \end{gather*}

    Define a sequence $\beta_i$ of integers by the requirement
    \begin{gather*}
        2^{\vdim (i-\beta_i+1)} > \phi ( 2^{-i} ) \geq 2^{\vdim (
        i-\beta_i) }.
    \end{gather*}
    For $k \leq i \in \integers$ observe that $\beta_{i+1} \geq \beta_i \geq
    i$ and inductively define families $F_i$ to consist of the open cubes
    $C(b,2^{-\beta_i})$, see \ref{miniremark:family_of_cubes}, corresponding
    to all $C(b,2^{-i}) \in W(i)$ satisfying the following condition:
    \begin{gather*}
    	\text{If $j$ is an integer, $k \leq j \leq i-1$, and $R \in F_j$ then
	$C(b,2^{-\beta_i}) \cap R = \varnothing$}.
    \end{gather*}
    Let $\varepsilon = \vdim^{1/2} 2^{1-k}$ and note that $\varepsilon
    \leq 1$. Define
    \begin{gather*}
        G = {\textstyle \bigcup} \{ F_i \with k \leq i \in \integers \},
        \quad
        A = {\textstyle C(0,1) \without ( N \cup \bigcup G)}.
    \end{gather*}

    Notice that $F_i$ is disjointed for $k \leq i \in \integers$ and $G$ is a
    disjointed subfamily of $Z$. Estimating
    \begin{gather*}
        \tsum{i=k}{\infty} 2^{\vdim (i-\beta_i)} \leq
        \tsum{i=k}{\infty} \phi ( 2^{-i} ) \leq \tsum{i=k}{\infty}
        (\log 2)^{-1} \tint{2^{-i}}{2^{1-i}} \phi (t) t^{-1} \ud
        \LM^1 t \leq 1-\lambda,
    \end{gather*}
    and $\card F_i \leq \card W(i) = 2^{i\vdim}$ whenever $k \leq i \in
    \integers$, one obtains
    \begin{gather*}
        \LM^\vdim ( {\textstyle\bigcup G} ) \leq
        \tsum{i=k}{\infty} 2^{\vdim(i-\beta_i)} \leq 1-\lambda, \quad
        \LM^\vdim ( A ) \geq \lambda.
    \end{gather*}

    Suppose $a \in A$ and $0 < r \leq \varepsilon$.

    There exist $i$ and $b$ such that
    \begin{gather*}
        k \leq i \in \integers, \quad 2^{-i} \leq \vdim^{-1/2} r \leq
        2^{1-i}, \quad {\textstyle a \in S = C (b,2^{-i}) \in W(i)}.
    \end{gather*}
    The proof will be concluded by showing: \emph{There exists $Q = C
    (c,t) \in G$ having the property
    \begin{gather*}
        t \geq 2^{-\beta_i} \quad \text{and} \quad Q \cap S \neq
        \varnothing
    \end{gather*}
    and this property implies
    \begin{gather*}
        Q \subset \oball{a}{r}, \quad \LM^\vdim (Q) \geq
        \omega (r) r^\vdim.
    \end{gather*}}
    Concerning the existence of $Q$, if $C(b,2^{-\beta_i}) \cap R \neq
    \varnothing$ for some integer $j$ with $k \leq j \leq i-1$ and some $R \in
    F_j$, then one may take $Q = R$, and otherwise one may take $Q =
    C(b,2^{-\beta_i}) \in F_i$. Concerning the implication of the property,
    estimate
    \begin{gather*}
        2^{-i} \leq s, \quad \phi ( 2^{-i} ) \geq 2^{2\vdim}
        \vdim^{\vdim/2} \omega (r), \quad t^\vdim \geq 2^{-\vdim
        \beta_i} > \phi (2^{-i}) 2^{\vdim(-i-1)} \geq \omega (r)
        r^{\vdim}
    \end{gather*}
    and note that $Q \subset S$, since $a \in S \without Q$ and $S,Q \in
    Z$, hence $Q \subset \oball{a}{r}$.
\end{proof}
\begin{remark} \label{remark:topsoe}
    Whenever $\omega$ violates the Dini condition and $\lambda > 0$ there do
    not exist $\varepsilon$, $G$ and $A$ as in \ref{example:cantor_set} as may
    be verified by applying Tops{\o}e \cite[Theorem 3]{MR577961} with $\| x
    \|$, $N$, $\mathscr{B}$, $\psi (r)$, and $\delta_0$ replaced by $\sup \{
    |x_i| \with i = 1, \ldots, \vdim \}$, $\vdim$, $G$, $\omega (r)^{1/\vdim}
    r/2$, and $\varepsilon$.
\end{remark}
\begin{remark} \label{remark:mejlbro_topsoe}
    The basic construction principle is that of Mejlbro and Tops{\o}e
    \cite[Theorem 2]{MR0583557} who established the sharpness of the Dini
    condition in a similar context.
\end{remark}
\begin{example}
    \label{example:yet_another_cantor_set}
    Suppose $\vdim$, $W$, $Z$, and $N$ are as in
    \ref{miniremark:family_of_cubes}, $\omega$ is a modulus of continuity, and
    $0 \leq \lambda < 1$.

    Then there exist $0 < \varepsilon \leq 1$, $B \subset \rel \cap \{ r
    \with r > 0 \}$ with $\inf B = 0$, and a disjointed subfamily $G$ of
    $Z$ such that
    \begin{gather*}
        {\textstyle A = C(0,1) \without \left (
        N \cup \bigcup G \right )} \quad \text{with} \quad
        \LM^\vdim (A) \geq \lambda
    \end{gather*}
    satisfying the following condition: If $a \in A$ and $0 < r \leq
    \varepsilon$, then there exists a subset $H$ of $G \cap \{ Q \with Q
    \subset \oball{a}{r} \}$ such that
    \begin{gather*}
        \LM^\vdim ( {\textstyle \bigcup H} ) \geq \omega (r) r^\vdim
        \quad \text{and} \quad
        \text{$\card H = 1$ if $r \in B$}.
    \end{gather*}
\end{example}
\begin{proof} [Construction]
    Choose $s$ with
    \begin{gather*}
        0 < s \leq 2^{-1} \vdim^{-1/2},
        \quad
        \omega ( 2 \vdim^{1/2} s) \leq 2^{-2\vdim} \vdim^{-\vdim/2},
    \end{gather*}
    define $\phi : \rel \cap \{ t
    \with 0 \leq t \leq 1 \} \to \rel$ by
    \begin{gather*}
        \phi (t) = 2^{2\vdim} \vdim^{\vdim/2} \omega ( 2 \vdim^{1/2} t)
        \quad
        \text{for $0 \leq t \leq s$},
        \qquad
        \phi (t) = 1
        \quad
        \text{for $s < t \leq 1$},
    \end{gather*}
    and note that $\phi$ is a modulus of continuity. Inductively choose
    sequences $\alpha_i$ and $\beta_i$ of nonnegative integers subject to
    the conditions
    \begin{gather*}
        \phi ( 2^{-\alpha_i} ) \leq ( 1- \lambda) 2^{-i},
        \qquad
        \alpha_{i+1} > \beta_i \geq \alpha_i,
        \\
        2^{\vdim ( \alpha_{i+1} - \beta_{i+1} + 1)} 
        > \phi (2^{-\alpha_i} )
        \geq 2^{\vdim ( \alpha_{i+1} - \beta_{i+1} )}
    \end{gather*}
    whenever $i$ is a positive integer and, in case $i > 1$, define families
    $F_i$ to consist of the open cubes $C(b,2^{-\beta_i})$, see
    \ref{miniremark:family_of_cubes}, corresponding to all $C
    (b,2^{-\alpha_i}) \in W ( \alpha_i )$ satisfying the following condition:
    \begin{gather*}
	\text{If $j$ is an integer, $1 < j \leq i-1$, and $R \in F_j$ then $C
	(b,2^{-\beta_i}) \cap R = \varnothing$}.
    \end{gather*}
    Define $\varepsilon = \inf \{ 2^{-\alpha_1} \vdim^{1/2}, 2 \vdim^{1/2} s
    \}$ and note that $\varepsilon \leq 1$. Let
    \begin{gather*}
        G = {\textstyle \bigcup} \{ F_i \with 1 < i \in \integers \},
        \quad
        {\textstyle A = C(0,1) \without \left ( N \cup \bigcup G \right )},
        \\
        B = \{ \vdim^{1/2} 2^{1-\alpha_i} \with 1 < i \in \integers \}.
    \end{gather*}
    Observe that $G$ is disjointed. Since $\card F_i \le 2^{\vdim \alpha_i}$
    for any positive integer~$i$, we have
    \begin{gather*}
        \LM^\vdim({\textstyle \bigcup G})
        \le \tsum{i = 2}{\infty} 2^{\vdim(\alpha_i - \beta_i)}
        \le \tsum{i = 2}{\infty} \phi(2^{-\alpha_{i-1}})
        \le 1-\lambda  .
    \end{gather*}
    Hence $\LM^\vdim(A) \ge \lambda$.

    Suppose $a \in A$ and $0 < r \leq \varepsilon$.

    There exist $k$ and $c$ satisfying
    \begin{gather*}
        \alpha_1 < k \in \integers, \quad 2^{-k} < \vdim^{-1/2} r \leq
        2^{1-k}, \quad a \in S = C (c,2^{-k}) \in W(k)
    \end{gather*}
    and $i$ such that
    \begin{gather*}
        1 < i \in \integers \quad \text{and} \quad \alpha_{i-1} < k
        \leq \alpha_i.
    \end{gather*}
    Defining $I = W (\alpha_i) \cap \{ Q \with Q \subset S \}$, one notes
    that
    \begin{gather*}
        \card I = 2^{\vdim(\alpha_i-k)},
    \end{gather*}
    in particular $\card I = 1$ if $r \in B$. Moreover, one concludes
    \begin{gather*}
        C (b,2^{-\alpha_i}) \in I \quad \text{implies} \quad \text{$C
        (b,2^{-\beta_i}) \subset R \subset S$ for some $R \in G$};
    \end{gather*}
    in fact, either $C (b,2^{-\beta_i}) \cap R \neq \varnothing$ for some
    integer $j$ with $1 < j \leq i-1$ and some $R \in F_j$, hence
    \begin{gather*}
        R \cap S \neq \varnothing, \quad C (b,2^{-\beta_i}) \subset R
        \subset S
    \end{gather*}
    as $\beta_i \geq \beta_j$ and $a \in S \without R$, because $a \in A \subset
    C(0,1) \without \bigcup G$, or
    \begin{gather*}
        C (b,2^{-\beta_i}) \in F_i \subset G, \quad C (b,2^{-\beta_i})
        \subset C (b,2^{-\alpha_i}) \subset S.
    \end{gather*}
    Consequently, there exists a subset $H$ of $G \cap \{ R \with R
    \subset S \}$ such that
    \begin{gather*}
        {\textstyle \bigcup \{ C (b,2^{-\beta_i}) \with C
        (b,2^{-\alpha_i}) \in I \} \subset \bigcup H}, \quad
        \text{$\card H = 1$ if $r \in B$}.
    \end{gather*}
    Noting $2^{-1} \vdim^{-1/2} r \leq \inf \{ 2^{-\alpha_{i-1}}, s \}$,
    one infers
    \begin{gather*}
        \LM^\vdim ( {\textstyle \bigcup H} ) \geq 2^{\vdim (
        \alpha_i-\beta_i-k)} \geq 2^{-2\vdim} \vdim^{-\vdim/2} \phi (
        2^{-\alpha_{i-1}} ) r^\vdim \geq \omega (r) r^\vdim.
    \end{gather*}
    Since $S \subset \oball{a}{r}$, the conclusion follows.
\end{proof}
\begin{remark} \label{remark:vitali_lebesgue}
    From the classical differentiation theory of Vitali and Lebesgue, see for
    instance \cite[2.8.17, 2.9.11]{MR41:1976}, it is evident that the lower
    bound on $\LM^\vdim ( \oball{a}{r} \without A )$ exhibited here is the
    known optimal one, see Tolstoff \cite[Th\'eor\`eme 3, p.~263]{MR0013777}.
    In order to demonstrate the sharpness of our results in
    \ref{theorem:quadratic-tilt-decay}, the additional property ``$\card H =
    1$ if $r \in B$'' will be employed in \ref{example:quadratic_tilt_excess}.
    Clearly, $B$ may not be required to equal $\{ r \with 0 < r \leq
    \varepsilon \}$ if $\omega$ violates the Dini condition and $\lambda > 0$
    by \ref{remark:topsoe}.
\end{remark}


\section{Cartesian product of varifolds}
\label{sec:varifold-theory}
The purpose of this section is to establish basic properties of the
Cartesian product of varifolds. The construction preserves rectifiability,
integrality and maps curvature varifolds to curvature varifolds, see
\ref{thm:product_varifold}. We also note a version of the coarea formula for
rectifiable varifolds, see
\ref{miniremark:rect_varifold}\,\eqref{item:rect:varifold:coarea}.

The proof of the rectifiability of the Cartesian product of rectifiable
varifolds needs to take into account that \emph{$(\mathscr{H}^\vdim,\vdim)$
rectifiable sets do not possess a similar stability property}, see
\ref{remark:product_rectifiable_sets}. Our treatment of curvature varifolds is
based on the characterisation of such varifolds in terms of generalised weakly
differentiable functions obtained in~\cite[15.6]{snulmenn:tv.v2}.

In the present paper only products of varifolds with planes are employed.
More general products will be required in the study of the geodesic distance
on the support of the weight measure of certain varifolds, see \cite[\S
6]{snulmenn:sobolev.v2}.

\begin{lemma} \label{lemma:second_fundamental_form}
    Suppose $\vdim, \adim \in \nat$, $\vdim \leq \adim$, $M$ is an $\vdim$
    dimensional submanifold of $\rel^\adim$ of class $2$, $\tau : M \to
    \Hom ( \rel^\adim, \rel^\adim )$ is defined by $\tau (z) = \project{\Tan
    (M,z)}$ for $z \in M$, and $\Hom ( \rel^\adim, \rel^\adim )$ is normed by
    $\| \cdot \|$.

    Then the following three statements hold.
    \begin{enumerate}
	\item \label{item:second_fundamental_form:b_from_tau} If $z \in M$ and
	$u,v \in \Tan (M,z)$, then $\mathbf{b} (M,z) (u,v) = \langle u,
	\langle v, \Der \tau (z) \rangle \rangle$.
	\item \label{item:second_fundamental_form:tau_from_b} If $z \in M$ and
	$u,v,w \in \rel^\adim$ then
	\begin{align*}
	    & \langle v, \langle u, \Der  \tau (z) \circ \tau (z) \rangle
	    \rangle \bullet w \\
	    & \quad = \mathbf{b} (M,z) ( \langle u, \tau(z) \rangle, \langle
	    v, \tau (z) \rangle ) \bullet w + \mathbf{b} (M,z) (  \langle u,
	    \tau (z) \rangle, \langle w, \tau (z) \rangle ) \bullet v.
	\end{align*}
	\item \label{item:second_fundamental_form:norm_equiv} If $z \in M$,
	then $\| \mathbf{b} (M,z) \| = \| \Der  \tau (z) \circ \tau (z) \|$.
    \end{enumerate}
\end{lemma}

\begin{proof}
    Define $\nu : M \to \Hom ( \rel^\adim, \rel^\adim )$ by $\nu (z) =
    \id{\rel^\adim} - \tau (z)$ for $z \in M$. Differentiating the equations
    $\tau (z) \circ \tau (z) = \tau (z)$ and $\nu (z) \circ \tau (z) = 0$
    for $z \in M$, one obtains for $z \in M$ and $u \in \Tan (M,z)$ that
    \begin{gather*}
        \tau (z) \circ \langle u, \Der  \tau (z) \rangle \circ \tau (z) = 0,
        \quad
        \nu (z) \circ \langle u, \Der  \tau (z) \rangle \circ \nu (z) = 0.
    \end{gather*}

    In order to prove \eqref{item:second_fundamental_form:b_from_tau}, suppose
    $z \in M$ and $u,v \in \Tan (M,z)$, notice that $\langle u, \langle  v,
    \Der \tau (z) \rangle \rangle \in \Nor (M,z)$, and differentiate the
    equation $\langle u, \tau (\zeta) \rangle \bullet g ( \zeta ) = 0$ for
    $\zeta \in M$, to obtain $\langle u, \langle v,\Der \tau (z) \rangle
    \rangle \bullet g(z) = - u \bullet \langle v, \Der g(z) \rangle $.
    Expressing
    \begin{align*}
	\langle v, \langle u, \Der  \tau (z) \circ \tau (z) \rangle \rangle
	\bullet w & = \langle \langle v, \tau (z) \rangle, \langle u, \Der
	\tau (z) \circ \tau (z) \rangle \rangle \bullet \langle w, \nu (z)
	\rangle \\
	& \phantom{=} \ + \langle \langle v, \nu (z) \rangle, \langle u, \Der
	\tau (z) \circ \tau (z) \rangle \rangle \bullet \langle w, \tau (z)
	\rangle
    \end{align*}
    for $z \in M$ and $u,v,w \in \rel^\adim$,
    \eqref{item:second_fundamental_form:tau_from_b} follows from the symmetry
    of $\langle u, \Der  \tau (z) \circ \tau (z) \rangle$ and
    \eqref{item:second_fundamental_form:b_from_tau}. Finally, noting
    \begin{gather*}
	| \langle v, \tau (z) \rangle || \langle w, \nu (z) \rangle | + |
	\langle v, \nu (z) \rangle | | \langle w, \tau (z) \rangle | \leq
	|v||w|
    \end{gather*}
    for $z \in M$ and $v,w \in \rel^\adim$ by H{\"o}lder's inequality,
    \eqref{item:second_fundamental_form:norm_equiv} follows from
    \eqref{item:second_fundamental_form:b_from_tau} and
    \eqref{item:second_fundamental_form:tau_from_b}.
\end{proof}

\begin{remark}
    Items
    \ref{lemma:second_fundamental_form}\,\eqref{item:second_fundamental_form:b_from_tau}\,\eqref{item:second_fundamental_form:tau_from_b}
    are in analogy with Hutchinson \cite[5.1.1\,(i)\,(ii)]{MR825628}.
\end{remark}
\begin{definition} \label{def:curvature_varifold}
    Suppose $\vdim, \adim \in \nat$, $\vdim \leq \adim$, and $U$ is an open
    subset of $\rel^\adim$.

    Then $V$ is called \emph{$\vdim$ dimensional curvature varifold in $U$} if
    and only if the following three conditions are satisfied:
    \begin{enumerate}
    	\item $V$ is an $\vdim$ dimensional integral varifold in $U$.
	\item $\| \delta V \|$ is a Radon measure absolutely continuous with
	respect to $\| V\|$.
	\item If $Y = \Hom ( \rel^\adim, \rel^\adim ) \cap \{ \sigma \with
	\sigma = \sigma^\ast \}$, $Z = U \cap \{ z \with \Tan^\vdim ( \|V \|,
	z ) \in \grass \adim \vdim \}$, and $\tau : Z \to Y$ is defined by
	$\tau (z) = \project{\Tan^\vdim ( \| V \|, z )}$ for $z \in Z$, then
	$\tau$ is a generalised $V$ weakly differentiable function.
    \end{enumerate}
    In this case one defines for $z \in Z \cap \dmn \derivative V \tau$ the
    \emph{second fundamental form of~$V$ at~$z$} by
    \begin{gather*}
        \mathbf{b} (V,z) : \Tan^\vdim ( \| V \|, z ) \times \Tan^\vdim ( \| V \|, z ) \to \rel^\adim,
        \\
	\mathbf{b} (V,z) (u,v) = \langle u, \langle v, \derivative{V}{\tau}
	(z) \rangle \rangle \quad \text{whenever $u,v \in \Tan^\vdim ( \| V
	\|, z )$}.
    \end{gather*}
\end{definition}
\begin{remark} \label{remark:curvature_varifold} In view of
    \cite[15.4--15.6]{snulmenn:tv.v2} the preceding definition of curvature
    varifold is equivalent to Hutchinson's original definition in
    \cite[5.2.1]{MR825628}. Recalling \cite[4.8]{snulmenn.c2} and
    \cite[11.2]{snulmenn:tv.v2}, one infers that, for $\| V \|$ almost all $z$, the
    second fundamental form of $V$ at $z$ is a symmetric bilinear map with
    values in $\Nor^\vdim ( \| V \|, z)$ which is related to
    $\derivative{V}{\tau} (z)$ as $\mathbf{b}(M,z)$ is related to $\Der  \tau (z)
    \circ \tau (z)$ in
    \ref{lemma:second_fundamental_form}\,\eqref{item:second_fundamental_form:tau_from_b}\,\eqref{item:second_fundamental_form:norm_equiv};
    in fact, if $M$ is an $\vdim$ dimensional submanifold of $\rel^\adim$ of
    class~$2$ and $\sigma : M \to \Hom ( \rel^\adim , \rel^\adim )$ satisfies
    $\sigma (z) = \project{\Tan (M,z)}$ for $z \in M$, then 
    \begin{gather*}
        \Tan(M,z) = \Tan^\vdim(\|V\|,z) \quad \text{and} \quad \Der  \sigma (z) = (
        \| V \|, \vdim ) \ap \Der  \tau (z)
    \end{gather*}
    for $\| V\|$ almost all $z \in U \cap M$ by \cite[2.8.18, 2.9.11,
    3.2.17]{MR41:1976} and Allard \cite[3.5\,(2)]{MR0307015}. Notice also
    that, for $\| V\|$ almost all $z$,
    \begin{gather*}
        \mathbf{b} (V,z) \circ 
        ( \project{\Tan^\vdim ( \| V \|, z )} \times \project{\Tan^\vdim ( \| V \|, z )} )
    \end{gather*}
    corresponds to the generalised second fundamental form at $(z,\Tan^\vdim (
    \| V \|, z) )$ in the sense of Hutchinson \cite[5.2.5]{MR825628}.
\end{remark}
\begin{miniremark}
    \label{miniremark:traces}
    Suppose $V$ and $W$ are finite dimensional vectorspaces, $f \in \Hom (
    V,W)$, and $g \in \Hom (W,V)$. Then $\trace ( g \circ f ) = \trace ( f \circ
    g)$; in fact, the argument of \cite[1.4.5]{MR41:1976} for the case $V=W$
    applies to the present case as well.
\end{miniremark}
\begin{miniremark}
    \label{miniremark:rect_varifold}
    Suppose $\vdim, \adim \in \nat$, $\vdim \leq \adim$, $U$ is an open subset
    of $\rel^\adim$, and $V \in \RVar_\vdim (U)$. Then the following two
    statements hold.
    \begin{enumerate}
	\item
        \label{item:rect_varifold:c1}
        There exist sequences of compact subset $C_j$ of $\vdim$ dimensional
        submanifolds of $U$ of class $1$ and $0 < d_j < \infty$ such that
    	\begin{gather*}
            V(k) = \tsum{j=1}{\infty} d_j \tint{C_j}{} k (z,\Tan^\vdim
            (\mathscr{H}^\vdim \restrict C_j,z)) \ud \mathscr{H}^\vdim z
        \end{gather*}
        for $k \in \mathscr{K} ( U \times \grass \adim \vdim )$.
	\item \label{item:rect:varifold:coarea} If $\vdim \geq \mu \in \nat$,
        $M = \{ z \with 0 < \density^\vdim ( \| V \|, z ) < \infty \}$, and $f
        : U \to \rel^\mu$ is Lipschitzian, then $f$ is $(\| V \|, \vdim)$
        differentiable at $\| V \|$ almost all $z$ and there holds
        \begin{align*}
            & \tint{}{} g(z) {\textstyle \big \| \bigwedge_\mu (
              \| V \|, \vdim ) \ap \Der f(z) \big \|} \ud \| V \| z \\
            & \qquad = \tint{}{} \tint{M \cap f^{-1} \lIm \{ y \} \rIm}{}
            g(z) \density^\vdim ( \| V \|, z ) \ud
            \mathscr{H}^{\vdim-\mu} z \ud \mathscr{L}^\mu y
        \end{align*}
        whenever $g$ is a $\| V \|$ integrable $\overline \rel$ valued
        function, where $\infty \cdot 0 = 0$.
    \end{enumerate}
    \eqref{item:rect_varifold:c1} may be verified by means of \cite[2.8.18,
    2.9.11, 3.2.17, 3.2.29]{MR41:1976}. The first half of
    \eqref{item:rect:varifold:coarea} is implied by
    \cite[4.5\,(2)]{snulmenn.decay}. Concerning the second half of
    \eqref{item:rect:varifold:coarea}, one notices that $g|C_j$ is
    $\mathscr{H}^\vdim \restrict C_j$ integrable and
    \begin{gather*}
        \density^\vdim ( \| V \|, z ) = \tsum{j \in J(z)}{} d_j \quad
        \text{for $\mathscr{H}^\vdim$ almost all $z \in U$}, \\
        ( \| V \|, \vdim ) \ap \Der f(z) = ( \mathscr{H}^\vdim \restrict C_j, \vdim
        ) \ap \Der f(z) \quad \text{for $\mathscr{H}^\vdim$ almost all $z \in C_j$},
    \end{gather*}
    where $J(z) = \nat \cap \{j \with z \in C_j \}$, by Allard
    \cite[2.8\,(4a), 3.5\,(2)]{MR0307015} and obtains
    \begin{align*}
	& {\textstyle \tint{C_j}{} g (z) \big \| \bigwedge_\mu (
	\mathscr{H}^\vdim \restrict C_j, \vdim ) \ap \Der f(z) \big \| \ud
	\mathscr{H}^\vdim z} \\
	& \qquad = \tint{}{} \tint{C_j \cap f^{-1} \lIm \{ y \} \rIm}{} g (z)
	\ud \mathscr{H}^{\vdim-\mu} z \ud \mathscr{L}^\mu y
    \end{align*}
    from \cite[2.10.35, 3.2.22\,(3)]{MR41:1976}, hence multiplying by $d_j$
    and summing over $j$ by means of \cite[2.10.25]{MR41:1976} yields the
    conclusion since $M$ and $\bigcup_{j=1}^\infty C_j$ are
    $\mathscr{H}^\vdim$ almost equal by Allard \cite[3.5\,(1b)]{MR0307015} and
    the equation for $\density^\vdim ( \| V \|, z )$.
\end{miniremark}
\begin{theorem} \label{thm:product_varifold}
    Suppose for $i \in \{ 1,2 \}$, $\vdim_i, \adim_i \in \nat$, $\vdim_i
    \leq \adim_i$, $p_i : \rel^{\adim_1} \times \rel^{\adim_2} \to
    \rel^{\adim_i}$ satisfy $p_i (z_1,z_2) = z_i$ for $(z_1,z_2) \in
    \rel^{\adim_1} \times\rel^{\adim_2}$, $U_i$ are open subsets of
    $\rel^{\adim_i}$, $V_i \in \Var_{\vdim_i} ( U_i )$, and $W \in
    \Var_{\vdim_1+\vdim_2} ( U_1 \times U_2 )$ satisfies
    \begin{gather*}
    	W(k) = \tint{}{} k ((z_1,z_2), S_1 \times S_2) \ud (V_1 \times V_2)
	((z_1,S_1),(z_2,S_2))
    \end{gather*}
    for $k \in \mathscr{K} \big ( ( U_1 \times U_2 ) \times
    \grass{\rel^{\adim_1} \times \rel^{\adim_2}}{ \vdim_1+\vdim_2 } \big )$.

    Then the following eight statements hold.
    \begin{enumerate}
	\item \label{item:product_varifold:weight} There holds $\| W \| = \|
	V_1 \| \times \| V_2 \|$.
	\item \label{item:product_varifold:rectifiable} If $V_1$ and $V_2$ are
	rectifiable, so is $W$ and, for $\| W \|$ almost all $(z_1,z_2)$,
	\begin{gather*}
	    \Tan^{\vdim_1+\vdim_2} ( \| W \|, (z_1,z_2)) = \Tan^{\vdim_1} ( \|
	    V_1 \| , z_1 ) \times \Tan^{\vdim_2} ( \| V_2 \|, z_2 ), \\
	    \density^{\vdim_1+\vdim_2} ( \| W \|, (z_1,z_2) ) =
	    \density^{\vdim_1} ( \| V_1 \|, z_1 ) \density^{\vdim_2} ( \|V_2\|
	    , z_2 ).
	\end{gather*}
	\item \label{item:product_varifold:integral} If $V_1$ and $V_2$ are
	integral, so is $W$.
	\item \label{item:product_varifold:planes} If $S_i \in \grass
	{\adim_i}{\vdim_i}$ for $i \in \{1,2\}$ and $h \in \Hom (
	\rel^{\adim_1} \times \rel^{\adim_2}, \rel^{\adim_1} \times
	\rel^{\adim_2} )$, then
    	\begin{gather*}
	    \eqproject{S_1 \times S_2} = p_1^\ast \circ \eqproject{S_1} \circ
	    p_1 + p_2^\ast \circ \eqproject{S_2} \circ p_2, \\
	    \eqproject{S_1 \times S_2} \bullet h = \eqproject{S_1} \bullet (
	    p_1 \circ h \circ p_1^\ast ) + \eqproject{S_2} \bullet ( p_2
	    \circ h \circ p_2^\ast ).
	\end{gather*}
	\item \label{item:product_varifold:variation} If $\theta \in
	\mathscr{D} (U_1 \times U_2, \rel^{\adim_1} \times \rel^{\adim_2} )$,
	then
	\begin{align*}
	    ( \delta W ) ( \theta ) & = \tint{}{} ( \delta V_1)_{z_1} ( p_1 (
	    \theta (z_1,z_2))) \ud \| V_2 \| z_2 \\
	    & \phantom{=}\ + \tint{}{} ( \delta V_2)_{z_2} ( p_2 ( \theta
	    (z_1,z_2))) \ud \| V_1 \| z_1.
	\end{align*}
	\item \label{item:product_varifold:variation_measure} If $\| \delta
	V_i \|$ are Radon measures for $i \in \{ 1,2 \}$, then
	\begin{gather*}
	    \| \delta W \| \leq \| \delta V_1 \| \times \| V_2 \| + \| V_1 \|
	    \times \| \delta V_2 \|
	\end{gather*}
	and, for $\theta \in \Lp{1} \big ( \| \delta V_1 \| \times \| V_2 \| +
	\| V_1 \| \times \| \delta V_2 \|, \rel^{\adim_1} \times
	\rel^{\adim_2} \big )$, the equation in
	\eqref{item:product_varifold:variation} holds.
	\item \label{item:product_varifold:tv_again} If, for $i \in \{ 1,2
	\}$, $V_i$ are rectifiable, $\| \delta V_i \|$ are Radon measures,
	$Y_i$ are finite dimensional normed vectorspaces and $f_i \in \trunc
	(V_i, Y_i)$, and $f : \dmn f_1 \times \dmn f_2 \to Y_1 \times Y_2$
	satisfies
	\begin{gather*}
	    f(z_1,z_2) = (f_1(z_1),f_2(z_2)) \quad \text{for $z_1 \in \dmn
	    f_1$ and $z_2 \in \dmn f_2$},
	\end{gather*}
	then $f \in \trunc (W, Y_1 \times Y_2)$ and, for $\| W \|$ almost all
	$(z_1,z_2)$,
	\begin{gather*}
	    \derivative{W}{f} (z_1,z_2) (u_1,u_2) = ( \derivative{V_1}{f_1}
	    (z_1)(u_1), \derivative{V_2}{f_2} (z_2)(u_2) )
	\end{gather*}
	whenever $u_1 \in \rel^{\adim_1}$ and $u_2 \in \rel^{\adim_2}$.
	\item \label{item:product_varifold:curvature} If $V_1$ and $V_2$ are
	curvature varifolds, then so is $W$ and, for $\| W \|$ almost all
	$(z_1,z_2)$,
	\begin{gather*}
	    \mathbf{b} (W,(z_1,z_2)) ( (u_1,u_2), (v_1,v_2)) = ( \mathbf{b}
	    (V_1,z_1) (u_1,u_2), \mathbf{b} (V_2,z_2) (u_2,v_2) )
	\end{gather*}
	whenever $u_1,v_1 \in \rel^{\adim_1}$ and $u_2,v_2 \in
	\rel^{\adim_2}$.
    \end{enumerate}
\end{theorem}
\begin{proof} [Proof of
    \eqref{item:product_varifold:weight}]
    \eqref{item:product_varifold:weight} follows from Fubini's theorem.
\end{proof}
\begin{proof} [Proof of \eqref{item:product_varifold:rectifiable}]
    If $C_i$ are compact subsets of $\vdim_i$ dimensional submanifolds of $U_i$
    of class $1$ for $i \in \{ 1,2 \}$, then $C_1 \times C_2$ is
    $(\mathscr{H}^{\vdim_1+\vdim_2},\vdim_1+\vdim_2)$ rectifiable with
    \begin{gather*}
	( \mathscr{H}^{\vdim_1} \restrict C_1 ) \times ( \mathscr{H}^{\vdim_2}
	\restrict C_2 ) = \mathscr{H}^{\vdim_1+\vdim_2} \restrict ( C_1 \times
	C_2 ),
    \end{gather*}
    by \cite[3.2.23]{MR41:1976} and, for $\mathscr{H}^{\vdim_1+\vdim_2}$
    almost all $(z_1,z_2) \in C_1 \times C_2$,
    \begin{align*}
        & \Tan^{\vdim_1} ( \mathscr{H}^{\vdim_1} \restrict C_1, z_1 ) \times
        \Tan^{\vdim_2} ( \mathscr{H}^{\vdim_2} \restrict C_2, z_2 ) \\
        & \qquad = \Tan^{\vdim_1+\vdim_2} ( \mathscr{H}^{\vdim_1+\vdim_2}
        \restrict ( C_1 \times C_2 ), (z_1,z_2) )
    \end{align*}
    by \cite[2.8.18, 2.9.11]{MR41:1976}. Therefore the assertion may be verified
    by means of \ref{miniremark:rect_varifold}\,\eqref{item:rect_varifold:c1},
    Allard \cite[2.8\,(4a), 3.5\,(2)]{MR0307015}, \cite[2.1.1\,(11)]{MR41:1976},
    and \eqref{item:product_varifold:weight}.
\end{proof}
\begin{proof} [Proof of \eqref{item:product_varifold:integral}]
    This is a consequence of \eqref{item:product_varifold:rectifiable} and
    Allard \cite[3.5\,(1c)]{MR0307015}.
\end{proof}
\begin{proof} [Proof of \eqref{item:product_varifold:planes}\,\eqref{item:product_varifold:variation}\,\eqref{item:product_varifold:variation_measure}]
    The first equation of \eqref{item:product_varifold:planes} is obvious and
    implies the second equation by \ref{miniremark:traces}.
    \eqref{item:product_varifold:variation} is a consequence of
    \eqref{item:product_varifold:planes} and Fubini's theorem.
    \eqref{item:product_varifold:variation} implies
    \eqref{item:product_varifold:variation_measure}.
\end{proof}
\begin{proof} [Proof of \eqref{item:product_varifold:tv_again}]
    Assume $\dim ( Y_1 \times Y_2 ) \geq 1$. First, consider the case $\dim
    (Y_1 \times Y_2) \geq 2$. Define $F : p_1^{-1} \lIm \dmn
    \derivative{V_1}{f_1} \rIm \cap p_2^{-1} \lIm \dmn \derivative{V_2}{f_2}
    \rIm \to \Hom ( \rel^{\adim_1} \times \rel^{\adim_2}, Y_1 \times Y_2 )$ by
    \begin{gather*}
    	F (z_1,z_2)(u_1,u_2) = ( \derivative{V_1}{f_1} (z_1)(u_1),
	\derivative{V_2}{f_2}(z_2)(u_2) )
    \end{gather*}
    whenever $z_1 \in \dmn \derivative{V_i}{f_i}$ and $u_i \in \rel^{\adim_i}$
    for $i \in \{ 1,2 \}$. By \cite[8.4]{snulmenn:tv.v2} it sufficient to prove
    \begin{align*}
	& ( \delta W ) ( ( \gamma \circ f ) \theta ) \\
	& \qquad = \tint{}{} \gamma (f(z)) \project{T} \bullet \Der  \theta
	(z) \ud W (z,T) + \tint{}{} \langle \theta (z), \Der  \gamma ( f(z))
	\circ F (z) \rangle \ud \| W \| z
    \end{align*}
    whenever $\theta \in \mathscr{D} (U_1 \times U_2, \rel^{\adim_1} \times
    \rel^{\adim_2} )$ and $\gamma \in \mathscr{D} ( Y_1 \times Y_2, \rel )$.
    Recalling \cite[2.15, 3.1]{snulmenn:tv.v2},\footnote{The topologies on
    $\mathscr{D} ( Y_1 \times Y_2, \rel )$ considered in
    \cite[2.13]{snulmenn:tv.v2} and \cite[4.1.1]{MR41:1976} differ but possess
    the same convergent sequences, see \cite[2.15, 2.17\,(3)]{snulmenn:tv.v2}.}
    one may assume additionally that for some $\gamma_i \in \mathscr{D}
    (Y_i,\rel)$ for $i \in \{1,2 \}$, there holds
    \begin{gather*}
	\gamma (y_1,y_2) = \gamma_1 (y_1) \gamma_2 (y_2) \quad \text{for
	$(y_1,y_2) \in Y_1 \times Y_2$}.
    \end{gather*}
    In this case one computes, noting
    \begin{align*}
    	& ( \delta V_i )_{z_i} ( \gamma_i (f_i(z_i)) p_i ( \theta (z_1,z_2)) )
	\\
	& \qquad = \tint{}{} \gamma_i (f_i(z_i)) \eqproject{S_i} \bullet ( p_i
	\circ \Der  \theta (z_1,z_2) \circ p_i^\ast ) \ud V_i (z_i,S_i) \\
	& \qquad \phantom = \ + \tint{}{} \langle p_i (\theta (z_1,z_2)), \Der 
	\gamma_i ( f_i (z_i) ) \circ \derivative{V_i}{f_i} (z_i) \rangle \ud
	\| V_i \| z_i
    \end{align*}
    whenever $i,j \in\{ 1,2 \}$, $i \neq j$ and $z_j \in \dmn f_j$, and using
    \eqref{item:product_varifold:planes} and
    \eqref{item:product_varifold:variation_measure},
    \begin{align*}
    	& ( \delta W ) ( ( \gamma \circ f ) \theta ) \\
	& \qquad = \tint{}{} \gamma_2 ( f_2 (z_2) ) ( \delta V_1 )_{z_1} (
	\gamma_1 ( f_1 (z_1)) p_1 ( \theta (z_1,z_2)) ) \ud \| V_2 \| z_2 \\
	& \qquad \phantom = \ + \tint{}{} \gamma_1 ( f_1 (z_1) ) ( \delta V_2
	)_{z_2} ( \gamma_2 ( f_2 (z_2)) p_2 ( \theta (z_1,z_2)) ) \ud \| V_1
	\| z_1 \\
	& \qquad = \tint{}{} \gamma (f(z)) \project{T} \bullet \Der  \theta
	(z) \ud W (z,T) + \tint{}{} \langle \theta (z), \Der  \gamma ( f(z))
	\circ F (z) \rangle \ud \| W \| z.
    \end{align*}
    If $\dim ( Y_1 \times Y_2 ) = 1$, one may assume $\dim Y_1 = 1$, hence
    $\dim Y_2 = 0$, and similarly consider $\gamma_1 \in \mathscr{E} (Y_1,
    \rel )$ such that $\spt \Der  \gamma_1$ is compact and $\gamma_2 = 1$.
\end{proof}
\begin{proof} [Proof of \eqref{item:product_varifold:curvature}]
    Define $Y_i = \Hom ( \rel^{\adim_i}, \rel^{\adim_i} ) \cap \{ \sigma \with
    \sigma = \sigma^\ast \}$ and functions $f_i \in \trunc ( V_i, Y_i)$ by $f_i
    (z_i) = \project{\Tan^{\vdim_i} ( \| V_i \|, z_i )}$ whenever $z_i \in U_i$
    and $\Tan^{\vdim_i} ( \| V_i \|, z_i ) \in \grass {\adim_i}{\vdim_i}$ for $i
    \in \{ 1, 2 \}$. Associate $f \in \trunc (W,Y_1 \times Y_2 )$ to $f_1$ and
    $f_2$ as in \eqref{item:product_varifold:tv_again}. Define $Y$ to be the
    vectorspace of symmetric endomorphisms of $\rel^{\adim_1} \times
    \rel^{\adim_2}$ and let the linear map $L : Y_1 \times Y_2 \to Y$ be defined
    by
    \begin{gather*}
    	L ( \sigma_1, \sigma_2 ) = p_1^\ast \circ \sigma_1 \circ p_1 + p_2^\ast
	\circ \sigma_2 \circ p_2 \quad \text{for $\sigma_1 \in Y_1$ and
	$\sigma_2 \in Y_2$}.
    \end{gather*}
    In view of \eqref{item:product_varifold:rectifiable} and
    \eqref{item:product_varifold:planes}, one infers
    \begin{gather*}
	\project{\Tan^{\vdim_1+\vdim_2} ( \| W \|, (z_1,z_2))} = L ( f
	(z_1,z_2)) \quad \text{for $\| W \|$ almost all $(z_1,z_2)$}.
    \end{gather*}
    Therefore $W$ is a curvature varifold by
    \eqref{item:product_varifold:weight},
    \eqref{item:product_varifold:integral},
    \eqref{item:product_varifold:variation_measure}, and
    \cite[8.6]{snulmenn:tv.v2} and, by \eqref{item:product_varifold:tv_again},
    there holds, for $\| W \|$ almost all $(z_1,z_2)$,
    \begin{gather*}
	\derivative{W}{f} (z_1,z_2) (u_1,u_2) = L ( \derivative{V_1}{f_1}
	(z_1)(u_1), \derivative{V_2}{f_2} (z_2)(u_2) )
    \end{gather*}
    for $u_1 \in \rel^{\adim_1}$ and $u_2 \in \rel^{\adim_2}$. Recalling
    \eqref{item:product_varifold:rectifiable}, the equation
    for the second fundamental form of $W$ now follows.
\end{proof}
\begin{remark}
    \label{remark:product_rectifiable_sets}
    The behaviour of rectifiable varifolds described in
    \eqref{item:product_varifold:rectifiable} is in contrast with the more
    subtle behaviour of $(\mathscr{H}^\vdim,\vdim)$ rectifiable sets; in fact,
    if $\vdim, \adim \in \nat$ and $\vdim < \adim$ there exist compact subsets
    $C_1$ and $C_2$ of $\rel^\adim$ such that
    \begin{gather*}
	\mathscr{H}^\vdim ( C_1 ) = \mathscr{H}^\vdim (C_2) =0, \quad
	\mathscr{H}^l ( C_1 \times C_2 ) = \infty
    \end{gather*}
    whenever $0 \leq l < \vdim + \adim$, see \cite[2.10.29,
    3.2.24]{MR41:1976}.
\end{remark}
\begin{remark}
    Concerning \eqref{item:product_varifold:tv_again} and
    \eqref{item:product_varifold:curvature}, notice that, for general $W$,
    neither is membership in $\trunc (W, Y_1 \times Y_2 )$ implied by membership
    of the component functions in $\trunc (W,Y_i)$ nor is $\trunc (W, Y_1 \times
    Y_2 )$ closed under addition, see \cite[8.25]{snulmenn:tv.v2}.
\end{remark}


\section{Differentiation results}
\label{sec:differentiation}

In the present section a differentiation theorem for measures relative to
varifolds is provided.  It slightly generalises
\cite[3.1]{snulmenn.isoperimetric} so as to become applicable in the proof of
the sharp decay rate almost everywhere for the quadratic tilt-excess of
certain two-dimensional varifolds in Section~\ref{sec:quadratic_tilt_decay}.
Additionally, a corollary for use in the study of one-dimensional varifolds in
Section~\ref{sec:one_dimensional_decay} is noted.

Results of the type considered here occur for instance in Calder{\'o}n and
Zygmund \cite[Theorem~10, p.~189]{MR0136849} in which paper a differentiation
theory of higher order in Lebesgue spaces is developed. The classical Rademacher
theorem is contained in that theory as special case, see Calder{\'o}n and
Zygmund \cite[Theorem~12, p.~204]{MR0136849}.

\begin{miniremark}
    \label{miniremark:situation_general}
    Suppose $\vdim, \adim \in \nat$, $\vdim \leq \adim$, $1 \leq p \leq \infty$,
    $U$ is an open subset of $\rel^\adim$, $V \in \Var_\vdim ( U)$, $\| \delta V
    \|$ is a Radon measure, $\density^\vdim ( \| V \|, z) \geq 1$ for $\| V \|$
    almost all $z$. If $p > 1$, then suppose additionally that $\mathbf{h} ( V,
    \cdot ) \in \Lploc{p} ( \| V \|, \rel^\adim)$ and
    \begin{gather*}
	\delta V (\theta) = - \tint{}{} \mathbf{h} (V,z) \bullet \theta(z) \ud
	\| V \| z \quad \text{for $\theta \in \mathscr{D} ( U, \rel^\adim)$}.
    \end{gather*}
    Therefore $V \in \RVar_\vdim (U)$ by Allard \cite[5.5\,(1)]{MR0307015}.
    If~$p = 1$ let $\psi = \| \delta V \|$. If $1 < p < \infty$ define a Radon
    measure $\psi$ over $U$ by the requirement $\psi (B) = \tint{B}{} |
    \mathbf{h} ( V, z ) |^p \ud \| V \| z$ whenever $B$ is Borel subset of~$U$.
\end{miniremark}
\begin{theorem}
    \label{thm:O_o}
    Suppose $\vdim$, $\adim$, $p$, $U$ and $V$ are as in
    \ref{miniremark:situation_general}, $1 \le p \le \vdim$, $\omega$ is a
    modulus of continuity, $\mu$ measures $U$ with $\mu(U \without \spt \|V\|) =
    0$, and $Z$ is a $\|V\|$ measurable set with $\mu(Z) = 0$. In case $p <
    \vdim$ suppose additionally that there exist $1 < q \leq \infty$ and $f \in
    \Lploc{q} (\|V\|)$ such that
	\begin{gather*}
        \liminf_{r \to 0+} r^{(1-1/q)\vdim p/(p-\vdim)} \omega (r) > 0, \\
        \mu (B) = \tint{B}{} f \ud \|V\| \quad \text{whenever $B$ is Borel subset of $U$.}
    \end{gather*}

    Then for $\HM^\vdim $ almost all $z \in Z$
    \begin{gather*}
	\limsup_{r \to 0+}  r^{-\vdim} \omega(r)^{-1} \measureball{\mu}{\cball
	zr} \quad \text{equals either $0$ or $\infty$}.
    \end{gather*}
\end{theorem}
\begin{proof}
    \setcounter{equation}{0}
    For $i \in \nat$ let $G_i$ denote the set of all $z \in \spt \| V \|$ such
    that either $\oball z{1/i} \not \subset U$ or
    \begin{gather*}
        \measureball{\| \delta V \|}{\cball zr } > (2 \isoperimetric{\vdim})^{-1} \|V\|(\cball zr)^{1-1/\vdim}
        \quad
        \text{for some $0 < r < 1/i$}.
    \end{gather*}
    Notice that $G_{i+1} \subset G_i$ and that $G_i$ is relatively open in $\spt
    \| V \|$ by an argument analogous to \cite[2.9.14]{MR41:1976}.

    We start with some preliminary observations. Consider the \emph{case $p <
      \vdim$}. If $q < \infty$, define $\nu$ to be the Radon measure over $U$
    characterised by the requirement $\nu (B) = \tint{B}{} f^q \ud \| V \|$
    whenever $B$ is a Borel subset of $U$. Using H{\"o}lder's inequality, one
    may employ \cite[2.9, 2.10]{snulmenn.isoperimetric} with $m$, $n$, $\mu$,
    $s$, $\varepsilon$, and $\Gamma$ replaced by $\codim$, $\vdim$, $\| V \|$,
    $\vdim$, $(2 \isoperimetric{\vdim} )^{p/(p-\vdim)}$, and $5 \vdim \gamma (
    \vdim )$ to see that for $\HM^\vdim$ almost all $z \in U$ there exists an
    $i \in \nat$ such that
    \begin{gather*}
        \density^{\vdim^2/(\vdim-p)} ( \| V \| \restrict G_i, z ) = 0  .
    \end{gather*}
    This implies, by another use of H{\"o}lder's inequality, that for $\HM^\vdim$
    almost all $z \in U$ there exists an $i \in \nat$ such that
    \begin{gather*}
        \density^{\vdim(1+(1-1/q)p/(\vdim-p))} ( \mu \restrict G_i, z) = 0  ,
    \end{gather*}
    since, if $q < \infty$, then $\density^{\ast \vdim} ( \nu, z ) < \infty$ for
    $\mathscr{H}^\vdim$ almost all $z \in U$ due
    to~\cite[2.10.19\,(3)]{MR41:1976}. Observe that if $z \in \spt \|V\|
    \without \bigcup_{i=1}^{\infty} G_i$, then, according to
    \cite[2.5]{snulmenn.isoperimetric}, $\density_*^\vdim(\|V\|,z) > 0$;
    hence, $\spt \| V \| \cap \{ z \with \density^\vdim ( \| V \|, z ) = 0
    \} \subset \bigcap_{i=1}^\infty G_i$. Recalling that $G_i$ are relatively
    open in $\spt \|V\|$ and $\mu(U \without \spt \|V\|) = 0$, one infers that
    \begin{gather*}
        \density^\vdim( \| V \|, z ) = 0 
        \quad \text{implies} \quad
        \density^{\vdim(1+(1-1/q)p/(\vdim-p))} ( \mu, z) = 0
    \end{gather*}
    for $\mathscr{H}^\vdim$ almost all $z \in U$. In \emph{case $p = \vdim$},
    one notices that
    \begin{gather*}
        \density^\vdim_\ast ( \| V \|, z ) \geq 1/2
        \quad \text{whenever $z \in \spt \| V \|$},
        \\
        {
          \textstyle \spt \| V \| \cap
          \big \{ z \with \| \delta V \| ( \{ z \} ) < (2\isoperimetric{\vdim})^{-1} \big \}
          \subset U \cap \bigcup_{i=1}^\infty \big\{ c \with G_i \cap \cball{c}{1/i} = \varnothing \big\}  ,
        }
        \\
        \spt \| V \| \cap  \{ z \with \| \delta V \| ( \{ z \} ) \ge (2\isoperimetric{\vdim})^{-1}\} \text{ is countable}
    \end{gather*}
    by \cite[4.8\,(4), 7.6]{snulmenn:tv.v2} and
    \cite[2.5]{snulmenn.isoperimetric}
    
    Therefore, one may assume in both cases that $Z \subset \{ z \with
    \density^{\ast \vdim} ( \| V \|, z ) > 0 \}$. As this implies that $Z$ is
    a~union of countably many $\mathscr{H}^\vdim$~measurable sets of finite
    $\mathscr{H}^\vdim$~measure by~\cite[2.10.19\,(3)]{MR41:1976}, one may also
    assume that $Z$ is compact.
    
    Define sets
    \begin{gather*}
        Z_j = Z \cap \{ z \with \measureball{\mu}{\cball zr} \leq j r^{\vdim} \omega(r+) \text{ for all } 0 < r < 1/j \}
    \end{gather*}
    whenever $j \in \nat$ and $1/j < \dist ( Z, \rel^\adim \without U)$ whose
    union contains
    \begin{gather*}
        Z \cap \Big\{ z \with \limsup_{r \to 0+} r^{-\vdim} \omega(r)^{-1} \measureball{\mu}{\cball zr} < \infty \Big\}
    \end{gather*}
    and observe that the sets $Z_j$ are compact, see~\cite[2.9.14]{MR41:1976}.
    Hence, it is sufficient to prove for each $j \in \nat$ with $1/j < \dist ( Z,
    \rel^\adim \without U )$ that
    \begin{gather*}
        \lim_{r \to 0+} r^{-\vdim} \omega(r)^{-1} \measureball{\mu}{\cball cr}
        = 0 \quad \text{for $\HM^\vdim $ almost all $c \in Z_j$}.
    \end{gather*}
    In fact, this equality will be proven whenever $c \in Z_j$ and $j \in \nat$
    with $1/j < \dist (Z, \rel^\adim \without U )$ satisfy for some $i \in \nat$
    with $i \geq 2j$ that $\density^\vdim ( \|V\| \restrict U \without Z_j, c )
    = 0$ and
    \begin{gather*}
        \density^{\vdim^2/(\vdim-p)} ( \|V\| \restrict G_i, c ) = 0 
        \quad
        \text{if $p < \vdim $},
        \qquad G_i \cap \cball{c}{1/i} = \varnothing
        \quad \text{if $p = \vdim$}
    \end{gather*}
    as $\HM^\vdim$ almost all $c \in Z_j$ do according to
    \cite[2.10.19\,(4)]{MR41:1976} and the second paragraph of this proof.
    
    For such $c$, $j$, and $i$, suppose $0 < \varepsilon \leq ( 6 \vdim
    \isoperimetric{\vdim} )^{-1}$, $0 < r < 1/i$, and
    \begin{gather*}
	\| V \| ( \cball{c}{2r} \without Z_j ) < \varepsilon^\vdim r^\vdim.
    \end{gather*}
    Whenever $\zeta \in \cball cr \cap ( \spt \|V\| ) \without ( G_i \cup
    Z_j)$ and $s = \dist (\zeta,Z_j)$ there exists $z \in Z_j$ with $s =
    |\zeta-z|$ and one infers, using \cite[2.5]{snulmenn.isoperimetric} with
    $n$, $m$, $U$, $\mu$, $\varepsilon$, and $\varrho$ replaced by $\vdim$,
    $\codim$, $\oball{\zeta}{1/i}$, $\| V \|$, $(2 \isoperimetric{\vdim}
    )^{-1}$, and $s/2$, that
    \begin{gather*}
        s \leq |\zeta-c| \leq r < 1/i \leq 1/(2j),
        \quad
        \cball \zeta{s/2} \subset \cball z{3s/2} \cap \cball c{2r} \without Z_j,
        \\
        ( 4 \vdim \isoperimetric{\vdim} )^{-\vdim} s^\vdim
        \leq \measureball{\| V \|}{\cball{\zeta}{s/2}} 
        \leq \| V \| ( \cball{c}{2r} \without Z_j ) 
        < \varepsilon^\vdim r^\vdim,
        \quad
        3s/2 < r,
        \\
        \measureball{\mu}{\cball \zeta{s/2}} 
        \leq \measureball{\mu}{\cball z{3s/2}}
        \leq j (3s/2)^{\vdim} \omega((3s/2)+)
        \leq \gamma \omega (r) \measureball{\|V\|}{\cball \zeta{s/2}},
    \end{gather*}
    where $\gamma = j ( 6 \vdim \isoperimetric{\vdim } )^{\vdim}$. Therefore the
    Besicovitch-Federer covering theorem yields the existence of~countable
    disjointed families $F_1, \ldots, F_{\besicovitch{\adim}}$ of closed balls
    such that
    \begin{gather*}
        \cball cr \cap ( \spt \|V\| ) \without ( G_i \cup Z_j )
        \subset {\textstyle\bigcup\bigcup \{ F_k \with k = 1, \ldots,
        \besicovitch{\adim} \}} \subset \cball c{2r} \without Z_j, \\
        \mu (S) \leq \gamma \omega(r) \| V \| (S) \quad \text{whenever
        $S \in F_k$ and $k = 1, \ldots, \besicovitch{\adim}$}  .
    \end{gather*}
    Recalling $\mu(Z) = 0$, the conclusion follows from
    \begin{align*}
        \mu ( \cball cr \without G_i )
        & = \mu ( \cball cr \cap ( \spt \|V\| ) \without ( G_i \cup Z_j ) )
        \\
        & \leq \tsum{k=1}{\besicovitch{\adim}} \tsum{S \in F_k}{} \mu (S) 
        \leq \tsum{k=1}{\besicovitch{\adim}} \gamma \omega (r) \tsum{S \in F_k}{} \| V \| (S)
        \\
        & \le \besicovitch{\adim } \gamma \omega (r) \|V\| ( \cball c{2r} \without Z_j)
        \leq \besicovitch{\adim} \gamma \varepsilon^\vdim r^\vdim \omega (r)  ,
    \end{align*}
    since $\varepsilon$ can be chosen arbitrary small.
\end{proof}
\begin{remark}
    The preceding theorem is a slight generalisation of
    \cite[3.1]{snulmenn.isoperimetric} which treated the case that $\omega
    (r)$ equals a positive power of $r$; see
    \cite[3.2--3.4]{snulmenn.isoperimetric} for comments on earlier
    developments and the sharpness of certain hypotheses. The case $q=1$ is
    excluded since in this case there is no modulus of continuity satisfying
    the condition on the limit inferior.
\end{remark}

\begin{corollary}
    \label{thm:differentiation}
    Suppose $\vdim$, $\adim$, $p$, $U$, and $V$ are as in
    \ref{miniremark:situation_general}, $p=\vdim$, $g$ maps $\| V \|$ almost all
    of $U$ into $\{ t \with 0 \leq t \leq \infty \}$, $\dmn g \subset \spt
    \|V\|$, and $0 < q < \infty$.

    Then, for $\| V \|$ almost all $z$, either
    \begin{enumerate}
	\item $\limsup_{\zeta \to z} |\zeta-z|^{-q} g(\zeta) = \infty$, or
	\item $g(z) = 0$ and $\lim_{\zeta \to z} |\zeta-z|^{-q} g(\zeta) =
	0$.\footnote{Our usage of limits is affected by the requirement
	    \begin{gather*}
		\limsup_{z \to c} g( z) = \lim_{\varepsilon \to 0+} \sup \{
		g(z) \with z \in \cball c\varepsilon \cap \dmn g \},
	\end{gather*}
	in particular $\limsup_{z \to c} g ( z ) = - \infty$ if $c \notin
	\Clos \dmn g$.}
    \end{enumerate}
\end{corollary}
\begin{proof}
    Define $C = \{ z \with g(z) = 0 \}$ and $Z = U \cap \{ \limsup_{\zeta \to
    z} |\zeta-z|^{-q} g(\zeta) < \infty \}$. Then $Z$ is a Borel subset of $U$
    and $Z \cap ( \dmn g ) \without C$ is countable; in fact, $Z$ is the union
    of the relatively closed subsets of $U$ defined by
    \begin{gather*}
    	Z_i = U \cap \{ z \with \text{$g( \zeta ) \leq i | \zeta-z |^q$
	whenever $\zeta \in \dmn g$ and $0 < | \zeta-z | < 1/i$} \}
    \end{gather*}
    for $i \in \nat$ and $Z_i \cap ( \dmn g ) \without C$ is contained in the
    set of isolated points of $Z_i$. Therefore $\| V \| ( Z \without C ) = 0$
    and $C \cap Z$ is $\| V \|$ measurable. Hence, applying \ref{thm:O_o} with
    $\omega (r)$, $\mu (B)$, and $Z$ replaced by $r$, $( \sup ( \{ 0 \} \cup g
    \lIm B \rIm))^{(\vdim+1)/q}$, and $C \cap Z$, one derives
    \begin{gather*}
    	\limsup_{\zeta \to z} | \zeta-z |^{-q} g ( \zeta ) = 0 \quad \text{for
	$\| V \|$ almost all $z \in Z$}.
    \end{gather*}
    As $\dmn g$ does not contain isolated points, the conclusion follows.
\end{proof}
\begin{remark}
    Clearly, the special case \cite[3.1]{snulmenn.isoperimetric} of
    \ref{thm:O_o} would also be sufficient for the proof of the corollary.
\end{remark}


\section{Quadratic tilt-excess, an example}
\label{sec:qte-ex}
The purpose of this section is to exhibit a two dimensional integral
varifold whose first variation is representable by integration in
order to render the decay estimates of
Section~\ref{sec:quadratic_tilt_decay} sharp. The varifold constructed
for that purpose in~\ref{example:quadratic_tilt_excess} will in fact
be associated to the graph of a Lipschitzian function with small
Lipschitz constant.

The basic building block of this example is a varifold consisting of a~piece of
a~sphere, a~piece of a~catenoid and a~smooth join to a~plane,
see~\ref{lemma:ball_catenoid}. Suitably rescaled copies of~this varifold will
then be used to fill the holes of a~set previously constructed
in~\ref{example:yet_another_cantor_set}.
\begin{figure}[!htb]
  \centering
  \includegraphics[width=\textwidth, keepaspectratio=true]{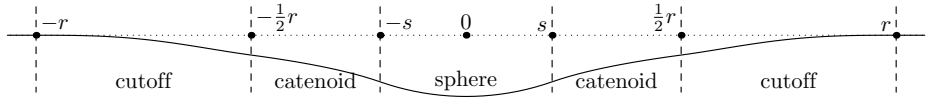}
  \caption{Rotating the solid line around the vertical axis illustrates the
    support of the varifold constructed in \ref{lemma:ball_catenoid}.}

  \label{F:catenoid}
\end{figure}
\begin{miniremark} [see \protect{\cite[5.1.9]{MR41:1976}}]
    \label{miniremark:pqT}
    Suppose $1 < \adim \in \nat$ and $\pp : \rel^\adim \to
    \rel^{\adim-1}$ and $\qq : \rel^\adim \to \rel$ satisfy
    \begin{gather*}
        \pp (z) = (z_1,\ldots,z_{\adim-1}) \quad \text{and} \quad \qq (z) =
	z_\adim \quad \text{whenever $z = (z_1,\ldots,z_\adim) \in
	\rel^\adim$}.
    \end{gather*}
    Then the statements of
    Allard \cite[8.9]{MR0307015} may be supplemented as follows.
    \begin{enumerate}
	\item \label{item:pqT:norm_equiv} If $S, T \in \grass \adim{\adim-1}$,
	then $| \project S - \project T | = 2^{1/2} \| \project S - \project T
	\|$.
	\item \label{item:pqT:graphs} If $L \in \Hom ( \rel^{\adim-1}, \rel
	)$, $S = \rel^\adim \cap \{ z \with L ( \pp ( z )) = \qq (z) \}$, and
	$T = \im \pp^\ast$, then $| L | = \| L \|$ and $\| \project S -
	\project T \| = ( 1 + \| L \|^2 )^{-1/2} \| L \|$.
    \end{enumerate}
    In fact, if $v \in \mathbf{S}^{\adim-1}$ and with $S = \{ z \with z
    \bullet v = 0 \}$ then \eqref{item:pqT:norm_equiv} is implied by
    \begin{gather*}
	2^{-1} | \project S - \project T |^2 = \project T \bullet \perpproject
	S = | \project T (v) |^2 = \| \project{T} \circ \perpproject{S} \|^2 =
	\| \project{S} - \project{T} \|^2
    \end{gather*}
    and in case of \eqref{item:pqT:graphs} one may take $v = ( 1 + \| L \|^2
    )^{-1/2} ( \pp^\ast ( L^\ast (1)) - \qq^\ast (1) )$.
\end{miniremark}

\begin{miniremark}
    \label{miniremark:radial_function}
    If $1 < \adim \in \nat$, $I$ is an open subset of $\{ t \with 0 < t <
    \infty \}$, and $g : I \to \rel$ is of class $2$, then $N = \rel^\adim
    \cap \{ z \with \qq (z) = g ( |\pp (z)| ) \}$, see \ref{miniremark:pqT},
    is an $\adim-1$ dimensional submanifold of $\rel^\adim$ of class $2$ and
    if $z \in N$ and $t = | \pp (z) |$ then
    \begin{align*}
	| \mathbf{h} (N,z) | & = ( 1 + g'(t)^2 )^{-1/2} \big | (\adim-2)
	t^{-1} g' ( t ) + (1+g'(t)^2)^{-1} g'' ( t ) \big | \\
	& \leq (\adim-2) t^{-1} |g'(t)| + |g''(t)|, \\
	\| \mathbf{b} (N,z) \| & = ( 1 + g'(t)^2 )^{-1/2} \sup \{ t^{-1}
	|g'(t)|, (1+g'(t)^2 )^{-1} |g''(t)| \} \\
	& \leq \sup \{ t^{-1} |g'(t)|, | g''(t)| \}
    \end{align*}
    as may be verified using the formulae occurring in
    \cite[p.~356--357,~388--391]{MR1814364}.
\end{miniremark}

\begin{miniremark} \label{miniremark:arcosh}
    We will employ the area cosinus hyperbolicus, $\ach : \{ t \with 1 \leq
    t < \infty \} \to \rel$, given by $\ach (t) = \log \big ( t +
    (t^2-1)^{1/2} \big )$ for $1 \leq t < \infty$. Notice that
    \begin{gather*}
	\ach' (t) = (t^2-1)^{-1/2}, \quad \ach''(t) = -t (t^2-1)^{-3/2}
    \end{gather*}
    for $1 < t < \infty$, hence for $3/2 \leq t < \infty$ also that
    \begin{gather*}
	\log t \leq \ach(t) \leq 3 \log t, \quad 1/t \leq \ach'(t) \leq 3/t,
	\quad - 3/t^2 \leq \ach''(t) < 0.
    \end{gather*}
    Moreover, $\ach ( t ) - \ach ( s ) \geq \log ( t/s)$ for $1 \leq s \leq t
    < \infty$.
\end{miniremark}

\begin{lemma} \label{miniremark:catenoid}
    Suppose $N = \rel^3 \cap \{ z \with | \qq (z) | = \ach ( | \pp (z) | )
    \}$, see \ref{miniremark:pqT} and \ref{miniremark:arcosh}.

    Then $N$ is a $2$ dimensional submanifold of $\rel^3$ of class $\infty$,
    $\mathbf{h} (N, z ) = 0$ for $z \in N$, and whenever $1 \leq r < \infty$
    there holds
    \begin{gather*}
	\HM^2 \big (N \cap \pp^{-1}\lIm \cball 0r \rIm \big ) = 2
	\unitmeasure{2} \big (\ach(r) + r (r^2 - 1)^{1/2} \big ), \\
	\tint{N \cap \pp^{-1}\lIm \cball 0r \rIm}{} | \project{\Tan(N,z)} -
	\pp^\ast \circ \pp |^2 \ud \HM^2 z = 8 \unitmeasure{2} \ach(r)  .
    \end{gather*}
\end{lemma}
\begin{proof}
    The asserted equations are readily verified by means of
    \ref{miniremark:pqT}--\ref{miniremark:arcosh}.
\end{proof}
\begin{remark}
    The surface $N$ is known as the catenoid, see e.g.~\cite[p.~18]{MR852409}.
\end{remark}
\begin{lemma}
    \label{lemma:ball_catenoid}
    Suppose $\adim = 3$, $\pp$ and $\qq$ are related to $\adim$ as in
    \ref{miniremark:pqT}, and $2 \leq s \leq r/2 < \infty$.

    Then there exist $h : \rel^2 \to \rel$ of class $1$ and $V \in \IVar_2 (
    \rel^3 )$ satisfying
    \begin{gather*}
	\Lip h \leq \Gamma/s, \quad \Lip \Der h< \infty, \quad \| V \| = \HM^2
	\restrict \im ( \pp^\ast + \qq^\ast \circ h ), \\
	\| V \| \restrict \pp^{-1} \lIm \rel^2 \without \oball 0r \rIm = \HM^2
	\restrict \pp^\ast \lIm \rel^2 \without \oball 0r \rIm, \\
	1 \leq \unitmeasure{2}^{-1} r^{-2} \| V \| \big ( \pp^{-1} \lIm
	\oball{0}{r} \rIm \big ) \leq 1 + \Gamma r^{-2} \log r, \quad \| \delta
	V \| ( \rel^3 ) \leq \Gamma, \\
	\text{$\| \delta V \|$ is absolutely continuous with respect to $\| V
	\|$}, \\
	\text{$\density^2 ( \| V \|,z ) = 1$ and $| \qq (z)| \leq 3 \log r$
	whenever $z \in \spt \| V \|$}, \\
	\tint{\pp^{-1} \lIm \oball{0}{r} \rIm \times \grass{3}{2}}{} |
	\project{S} - \pp^\ast \circ \pp |^2 \ud V (z,S) \geq \log (r/(2s)),
    \end{gather*}
    where $\Gamma$ is a universal, positive, finite number.
\end{lemma}
\begin{proof}
    Abbreviate $f_1 = \ach$, see \ref{miniremark:arcosh}. Define $d : \{
    \sigma \with 1 < \sigma < \infty \} \to \rel$ by
    \begin{gather*}
        d(\sigma) = f_1(\sigma) + \sigma/f_1'(\sigma)
    \end{gather*}
    for $1 < \sigma < \infty$ and note that $d(\sigma)-\sigma^2$ is a
    nondecreasing as a function of $\sigma$ with $d(\sigma)-\sigma^2 \to -1$
    as $\sigma \to 1+$, hence
    \begin{gather*}
        d(\sigma)-\sigma^2 \geq -1 \quad \text{for $1 < \sigma < \infty$}.
    \end{gather*}
    Define $f_2 : \{ t \with -s^2 < t < s^2 \} \to \rel$ by
    \begin{gather*}
        f_2 (t) = a(s)- \big ( s^4 - t^2 \big )^{1/2} \quad
        \text{for $-s^2 < t < s^2$},
    \end{gather*}
    hence $f_2 (t) \geq -1$ for $-s^2 < t < s^2$. Choose $\gamma \in
    \mathscr{E} ( \rel, \rel )$ with $0 \leq \gamma \leq 1$ and
    \begin{gather*}
    	\gamma (t) = 1 \quad \text{if $t \leq 1/2$}, \qquad \gamma(t) = 0
	\quad \text{if $t \geq 1$}, \qquad -3 \leq \gamma'(t) \leq 0
    \end{gather*}
    whenever $t \in \rel$. Noting
    \begin{gather*}
        s^4-s^2 = s^2/f_1'(s)^2 > 0, \quad f_2(s) = f_1(s), \quad
        f_2'(s) = f_1'(s),
    \end{gather*}
    one defines a function $g : \rel \to \rel$ of class $1$ with $\Lip g'
    < \infty$ by
    \begin{gather*}
        g (t) = f_2 (|t|) - f_1 (r) \quad \text{if $|t| \leq s$},
        \qquad g (t) = ( f_1(|t|) - f_1(r) ) \gamma ( |t|/r ) \quad
        \text{else}
    \end{gather*}
    whenever $t \in \rel$.

    One computes
    \begin{gather*}
        \begin{aligned}
            g'(t) & = f_1'(t) \gamma(t/r) + (f_1(t)-f_1(r)) r^{-1}
            \gamma'(t/r), \\
            g''(t) & = f_1''(t) \gamma(t/r) + 2 f_1'(t) r^{-1}
            \gamma'(t/r) + (f_1(t)-f_1(r)) r^{-2} \gamma''(t/r)
        \end{aligned}
    \end{gather*}
    for $s < t < \infty$, hence, taking into account \ref{miniremark:arcosh},
    there exists a positive, finite number $\Delta_1$ determined by $\gamma$
    such that
    \begin{gather*}
        - \Delta_1 \log r \leq g(t) \leq 0
        \quad \text{for $0 \leq t < \infty$},
        \\
        0 \leq g'(t) \leq \Delta_1 \inf \{ 1/s, 1/t \}
        \quad \text{for $0 \leq t < r$},
        \\
        g'(t) \geq 1 / t
        \quad \text{for $s \leq t \leq r/2$},
        \qquad |g''(t)| \leq \Delta_1 r^{-2}
        \quad \text{for $r/2 < t < r$}.
    \end{gather*}
    Define $h : \rel^2 \to \rel$ by
    \begin{gather*}
        h(x) = g(|x|) \quad \text{for $x \in \rel^2$}
    \end{gather*}
    and note that $h$ is of class $1$ with
    \begin{gather*}
        \Lip h \leq \Delta_1/s, \quad \Lip \Der h < \infty.
    \end{gather*}
    Let $M = \rel^3 \cap \{ z \with \qq(z) = h(\pp(z)) \}$, define $V \in
    \IVar_2 ( \rel^3 )$ by
    \begin{gather*}
	V(k) = \tint{M}{} k(z,\Tan(M,z)) \ud \HM^2 z \quad \text{for $k \in
	\mathscr{K} ( \rel^3 \times \grass{3}{2} )$}
    \end{gather*}
    and observe
    \begin{gather*}
	\delta V(\theta) = - \tint{M}{} \mathbf{h}(M,z) \bullet \theta(z) \ud
	\HM^2 z \quad \text{for $\theta \in \mathscr{D} (\rel^3, \rel^3 )$}.
    \end{gather*}
    Since $M \cap \pp^{-1} \lIm \oball{0}{s} \rIm$ is a~piece of a~sphere of
    radius~$s^2$ one computes
    \begin{gather*}
        | \mathbf{h}(M,z) | = 2 s^{-2}
        \quad \text{whenever $z \in M \cap \pp^{-1} \lIm \oball{0}{s} \rIm$}, \\
        \tint{M \cap \pp^{-1} \lIm \oball{0}{s} \rIm}{} |
        \mathbf{h}(M,z) | \ud \HM^2 z
        = 4 \unitmeasure{2} s / (s + (s^2 - 1)^{1/2}) \leq 4 \unitmeasure 2
    \end{gather*}
    and since $M \cap \pp^{-1} \lIm \cball{0}{r/2} \without \cball{0}{s} \rIm$
    is a~piece of a~catenoid one obtains
    \begin{gather*}
        \mathbf{h}(M,z) = 0
        \quad
	\text{whenever $z \in M \cap \pp^{-1} \lIm \cball{0}{r/2} \without
	\cball{0}{s} \rIm$}
    \end{gather*}
    from \ref{miniremark:catenoid}. Moreover, recalling $2 \le s \le r/2$, one
    estimates
    \begin{gather*}
        \begin{aligned}
	    0 \leq \HM^2 \big ( M \cap \pp^{-1} \lIm \oball{0}{r} \rIm \big )
	    - \unitmeasure{2} r^2 & = 2 \unitmeasure 2 \tint{0}{r} \big ( ( 1
	    + g'(t)^2)^{1/2} -1 \big ) t \ud \LM^1 t \\
            & \leq \unitmeasure 2 \tint{0}{r} g'(t)^2 t \ud \LM^1 t
            \leq 2 \unitmeasure{2} \Delta_1^2 \log r.
        \end{aligned}
    \end{gather*}
    From \ref{miniremark:radial_function} and the estimates for $g'$ and
    $g''$, one obtains a positive, finite number $\Delta_2$ determined by
    $\gamma$ such that
    \begin{gather*}
        | \mathbf{h}(M,z) | \leq \Delta_2 r^{-2}
        \quad
	\text{whenever $z \in M \cap \pp^{-1} \lIm \oball{0}{r} \without
	\cball{0}{r/2} \rIm$}.
    \end{gather*}
    Combining the preceding estimates, we obtain, recalling $r \geq 2$, that
    \begin{gather*}
	\| \delta V \| ( \rel^3 ) \leq 4 \unitmeasure 2 + \unitmeasure{2}
	\Delta_2(1 + \Delta_1^2)  .
    \end{gather*}
    Since $M \cap \pp^{-1}\lIm \oball{0}{r/2} \without \cball{0}{s} \rIm$ is
    a~piece of a~catenoid, \ref{miniremark:catenoid} implies
    \begin{gather*}
	    \tint{\pp^{-1} \lIm \oball{0}{r} \rIm \times \grass{3}{2}}{} |
	    \project{S} - \pp^\ast \circ \pp |^2 \ud V (z,S) \geq 8
	    \unitmeasure 2 ( \ach ( r/2) - \ach (s) ),
    \end{gather*}
    hence \ref{miniremark:arcosh} implies the conclusion.
\end{proof}

\begin{miniremark}
    \label{miniremark:open_cube}
    Occasionally, we denote the open cube with centre~$a$ and side length~$2r$
    by
    \begin{gather*}
        \ocube ar = \rel^\vdim \cap \{ (x_1,\ldots,x_\vdim) \with |x_i - a_i|
	< r \text{ for $i = 1,\ldots,\vdim$}\}
    \end{gather*}
    for $\vdim \in \nat$, $a = (a_1,\ldots,a_\vdim) \in \rel^\vdim$, and $0 < r
    < \infty$.
\end{miniremark}

\begin{example}
    \label{example:quadratic_tilt_excess}
    Suppose $\adim = 3$, $\pp$ and $\qq$ are related to $\adim$ as in
    \ref{miniremark:pqT}, $T = \im \pp^\ast$, $\varepsilon > 0$, and $\omega$
    is a modulus of~continuity.

    Then there exist $f : \rel^2 \to \rel$, $C \subset T$, and $V \in \IVar_2
    ( \rel^3 )$ satisfying
    \begin{gather*}
        \Lip f \leq \varepsilon, \quad \| V \| = \HM^2
        \restrict \im ( \pp^\ast + \qq^\ast \circ f ), \quad \| V \|
        ( C ) > 0, \\
        \| \delta V \| ( \rel^3 ) < \infty, \quad \text{$\| \delta V
        \|$ is absolutely continuous with respect to $\| V \|$}, \\
        \limsup_{r \to 0+} r^{-1} ( \log  (1/r) )^{-1/2}
        \omega(r)^{-1} \big ( \tfint{\cball{c}{r} \times
        \grass{3}{2}}{} | \project{S} - \project{T} |^2 \ud V (z,S)
        \big )^{1/2} > 0
    \end{gather*}
    whenever $c \in C$, here $0^{-1} = \infty$.
\end{example}
\begin{proof} [Construction]
    Take $G$ and $A$ as furnished by \ref{example:yet_another_cantor_set} with
    $\vdim$ and $\lambda$ replaced by $2$ and $1/2$, abbreviate $\Delta =
    \Gamma_{\ref{lemma:ball_catenoid}}$, $\lambda = \inf \{ 1/4, \varepsilon
    /(2\Delta) \}$, and let $C = \pp^\ast \lIm A \rIm$.  Define $W \in \IVar_2
    ( \rel^3 )$ by
    \begin{gather*}
	W(k) = \tint{T \without \pp^\ast \lIm \bigcup G \rIm}{} k (z,T) \ud
	\HM^2 z \quad \text{for $k \in \mathscr{K} ( \rel^3 \times
	\grass{3}{2} )$}.
    \end{gather*}
    Whenever $Q = \ocube at \in G$ and $t > \lambda$ let $f_Q : \rel^2 \to \rel$
    and $X_Q \in \IVar_2 ( \rel^3 )$ be defined by
    \begin{gather*}
        f_Q (x) = 0 \quad \text{for $x \in \rel^2$}, \qquad \| X_Q \|
        = \HM^2 \restrict \im ( \pp^\ast + \qq^\ast \circ f_Q
        ).
    \end{gather*}
    Whenever $Q = \ocube{a}{t} \in G$ and $t \leq \lambda$ apply
    \ref{lemma:ball_catenoid} with $s$ and $r$ replaced by $(2\lambda)^{-1}$
    and $1/t$ to construct $f_Q : \rel^2 \to \rel$ and $X_Q \in \IVar_2 (
    \rel^3 )$ such that $t^{-2} f_Q \circ \boldsymbol{\tau}_a \circ
    \boldsymbol{\mu}_{t^2}$ and $( \boldsymbol{\mu}_{t^{-2}} \circ
    \boldsymbol{\tau}_{-a} )_\# X_Q$ satisfy the conditions of
    \ref{lemma:ball_catenoid} in place of $h$ and $V$ implying
    \begin{gather*}
        \spt f_Q \subset \Clos Q ,
        \quad
        \Lip f_Q \leq \varepsilon,
        \quad
        \| X_Q \| = \HM^2 \restrict \im ( \pp^\ast + \qq^\ast \circ f_Q ),
        \\
	\| X_Q \| \restrict \pp^{-1} \lIm \rel^2 \without Q \rIm = \HM^2
	\restrict \pp^\ast \lIm \rel^2 \without Q \rIm, \\
        \| \delta X_Q \| ( \rel^3 ) \leq \Delta t^2,
        \quad
	\text{$\| \delta X_Q \|$ is absolutely continuous with respect to $\|
	X_Q \|$}, \\
	| \qq (z) | \leq 3 t^2 \log ( 1/t ) \quad \text{whenever $z \in \spt
	\| X_Q \|$}, \\
        \tint{\pp^{-1} \lIm Q \rIm \times \grass{3}{2}}{}
	| \project{S} - \project{T} |^2 \ud V (z,S) \geq t^4 \big ( \log (1/t)
	- \log (1/\lambda) \big ).
    \end{gather*}
    Recall that $G$ is disjointed and define $f : \rel^2 \to \rel$ and $V \in
    \IVar_2 ( \rel^3 )$ by
    \begin{gather*}
        f(x) = \sum_{Q \in G} f_Q (x) \quad \text{for $x \in \rel^2$},
        \qquad \| V \| = \HM^2 \restrict \im ( \pp^\ast + \qq^\ast \circ f ).
    \end{gather*}
    Note that $V = W + \sum_{Q \in G} X_Q \restrict ( \pp^{-1} \lIm Q \rIm
    \times \grass{3}{2} )$ and $\| V \| ( C ) \geq 1/2$. Observe
    \begin{gather*}
        \| \delta V \| ( \rel^3 ) < \infty, \quad \text{$\| \delta V
        \|$ is absolutely continuous with respect to $ \| V \|$}.
    \end{gather*}

    Suppose $c \in C$ and $\delta > 0$.

    Then there exist $r$ and $\ocube{a}{t} = Q$ such that
    \begin{gather*}
        0 < r \leq \delta, \quad Q \in G, \quad Q \subset
        \oball{a}{r}, \quad \LM^2 ( Q ) \geq \omega (r) r^2  ,
    \end{gather*}
    hence $t \leq r$. Since $(2t)^2 \geq \omega (r) r^2$ and
    \begin{gather*}
	t^4 \big ( \log (1/t) - \log (1/\lambda) \big ) \geq 2^{-4} \omega
	(r)^2 r^4 \big ( \log (1/r) - \log (1/\lambda) \big ),
    \end{gather*}
    the estimates for $X_Q$ imply the assertion.
\end{proof}


\section{A coercive estimate}
\label{sed:coercive}

In this section we provide, in \ref{lemma:coercive-estimate}, the first main
ingredient for the proof of the decay rates almost everywhere of the quadratic
tilt-excess of two dimensional integral varifolds whose first variation is
representable by integration, namely a coercive estimate. In this estimate the
quadratic tilt-excess is controlled by the variation measure of the first
variation and the height-excess. In order to be effective for the present
purpose, two aspects are crucial. Firstly, in the height-excess only the set of
points where the density ratio is bounded from below are taken into
account. Secondly, the height-excess term which is multiplied by a first
variation term is measured in the Orlicz space seminorm naturally corresponding
to square summable weak derivatives in two dimensions.

In the basic form of such coercive estimate all quantities are measured as
square integrals, see Allard~\cite[8.13]{MR0307015}. In \cite[5.5]{MR485012}
Brakke devised an interpolation procedure to obtain estimates in which the
variation is measured by its variation measure. This was further refined
in~\cite[4.14]{snulmenn.decay} by allowing the height-excess to be measured in
different Lebesgue spaces and in \cite[4.10]{snulmenn.decay} by restricting the
height-excess to the set of points where the density ratio is bounded from below
using a possibly discontinuous ``cut-off'' function, see
\ref{remark:tv-cutoff}. In \ref{lemma:coercive-estimate} we additionally refine
Brakke's interpolation procedure to include the relevant Orlicz space norm.

\begin{definition}
    \label{def:orlicz}
    If $\Phi : \{ t \with 0 \leq t < \infty \} \to \{ t \with 0 \leq t < \infty
    \}$ is a nondecreasing convex function with $\Phi(0)=0$ and $\lim_{t \to
      \infty} \Phi (t) = \infty$, $\mu$ measures $X$, and $Z$ is a Banach
    space, then one defines the seminorm $\mu_{(\Phi)}$ on $\mathbf{A} ( \mu,
    Z )$ by
    \begin{gather*}
        \Lpnorm{\mu}{\Phi}{f} = \inf \big \{ \lambda \with
        \text{$0 < \lambda \leq \infty$, $\orlicz{\Phi}{\lambda^{-1}f}{\mu} \leq 1$} \big \}
        \quad \text{for $f \in \mathbf{A} ( \mu, Z)$}.
    \end{gather*}
\end{definition}
\begin{remark}
    \label{remark:orlicz_eq}
    Notice that $\Lpnorm{\mu}{\Phi}{f} = 0$ if and only if $f (x) = 0$ for
    $\mu$ almost all $x$. Moreover, if $\Lpnorm{\mu}{\Phi}{f} > 0$ then
    $\orlicz{\Phi}{\lambda^{-1}f}{\mu} \leq 1$ for $\lambda =
    \Lpnorm{\mu}{\Phi}{f}$ and equality holds if $\tint{}{} \Phi \circ |
    s^{-1} f | \ud \mu < \infty$ for some $0 < s < \Lpnorm{\mu}{\Phi}{f}$.
\end{remark}
\begin{remark}
    The functions $\Phi$ and $\mu_{(\Phi)}$ are a ``Young's function'' and
    its corresponding ``Luxemburg norm'' in the terminology of \cite[Chapter
    4, 8.1, 8.6]{MR928802}.
\end{remark}
\begin{remark}
    \label{remark:basic_orlicz}
    Suppose $\Phi$, $\mu$, $X$ and $Z$ are as in \ref{def:orlicz}. Then the
    following basic properties hold.
    \begin{enumerate}
	\item
        \label{item:basic_orlicz:mult} 
        If $0 < c < \infty$ and $f \in \mathbf{A} (\mu,Z)$, then
        $\eqLpnorm{c\mu}{\Phi}{f} = \Lpnorm{\mu}{c \Phi}{f}$.
	\item
        \label{item:basic_orlicz:estimate}
        If $0 < \varepsilon \leq 1$ and $f \in \mathbf{A} ( \mu,Z )$, then
        $\varepsilon \Lpnorm{\mu}{\Phi}{f} \leq \Lpnorm{\mu}{\varepsilon
          \Phi}{f}$.
	\item
        \label{item:basic_orlicz:push}
        If $u : X \to Y$, $f : Y \to Z$, and $f \circ u \in \mathbf{A} (\mu,
        Z)$, then $f \in \mathbf{A} ( u_\# \mu, Z)$ and $\eqLpnorm{u_\#
          \mu}{\Phi}{f} = \Lpnorm{\mu}{\Phi}{f \circ u}$, see \cite[2.1.2,
        2.4.18\,(1)]{MR41:1976}.
    \end{enumerate}
\end{remark}
\begin{miniremark}
    \label{miniremark:kappa}
    Suppose $2 \leq \vdim \in \nat$ and $\kappa : \{ t \with 0 \leq t < \infty
    \} \to \{ t \with 0 \leq t < \infty \}$ satisfies
    \begin{gather*}
        \kappa (0) = 0,
        \qquad
        \kappa(t) = t \big ( 1 + (\log(1 + 1/t))^{1-1/\vdim} \big )
        \quad
        \text{for $0 < t < \infty$}.
    \end{gather*}
    Then one verifies that $\kappa$ is continuous increasing and concave, in
    particular $\kappa ( \tau t ) \leq \tau \kappa (t)$ for $1 \leq \tau <
    \infty$ and $0 \leq t < \infty$.
\end{miniremark}
\begin{miniremark}
    \label{miniremark:Phi}
    Suppose $2 \leq \vdim \in \nat$ and $\Phi : \{ t \with 0 \leq t < \infty
    \} \to \{ 0 \leq t < \infty \}$ is defined by
    \begin{gather*}
        \Phi (t) = \exp \big ( t^{\vdim/(\vdim-1)} \big ) -1 \quad \text{for
          $0 \leq t < \infty$}.
    \end{gather*}
    Then $\Phi$ satisfies the conditions of \ref{def:orlicz}, and $\Phi$ maps
    $\{ t \with 0 \leq t < \infty \}$ univalently onto $\{ t \with 0 \leq t <
    \infty \}$ with
    \begin{gather*}
        \Phi^{-1} (t) = ( \log  ( 1 + t ) )^{1-1/\vdim} \quad \text{for $0
          \leq t < \infty$}.
    \end{gather*}
    Therefore $\Lpnorm{\mu}{\Phi}{1} = 1 / \Phi^{-1}(1/\mu(X))$ whenever
    $\mu$ measures $X$ and $0 < \mu (X) < \infty$ by \ref{remark:orlicz_eq}.
    Notice, if $0 \leq \alpha < \infty$ then
    \begin{gather*}
        \inf \{ \alpha t + 1/\Phi(t) \with 0 < t < \infty \}
        \leq \kappa ( \alpha ),
    \end{gather*}
    where $\kappa$ is as in \ref{miniremark:kappa}; in fact, consider $t =
    \Phi^{-1} ( 1 / \alpha )$ if $\alpha > 0$.
\end{miniremark}
\begin{theorem}
    \label{lemma:coercive-estimate}
    Suppose $\vdim$, $\adim$, $p$, $U$, and $V$ are as in
    \ref{miniremark:situation_general}, $p = 1< \vdim$, $\kappa$ and $\Phi$ are
    related to $\vdim$ as in~\ref{miniremark:kappa} and~\ref{miniremark:Phi},
    $C$ and $K$ are compact subsets of $U$, $C \subset K$, $0 < r < \infty$, $H$
    is the set of all $z \in \spt \| V \|$ such that
    \begin{gather*}
    	\measureball{\| V \|}{\cball zs} \geq (40 \isoperimetric{\vdim} \vdim )^{-\vdim} s^\vdim
        \quad \text{whenever $0 < s < \infty$, $\cball zs \subset K$},
    \end{gather*}
    $c \in \rel^\adim$, $T \in \grass{\adim}{\vdim}$, and $h : U \to \rel$
    satisfies $h(z) = \dist (z-c,T)$ for $z \in U$.
    
    Then there holds
    \begin{multline*}
        r^{-\vdim} \tint{\{ z \with \oball{z}{r} \subset C\} \times \grass{\adim}{\vdim}}{}
        \| \project{S} - \project{T} \|^2 \ud V (z,S)
        \leq \Gamma \big ( r^{-\vdim} \| \delta V \| ( K)^{\vdim/(\vdim-1)}
        \\
        + \kappa ( r^{-\vdim} \| \delta V \| ( K )
        \eqLpnorm{\| V \| \restrict C \cap H}{r^{-\vdim} \Phi}{h} )
        + r^{-\vdim-2} \tint{C \cap H}{} |h|^2 \ud \| V \| \big ),
    \end{multline*}
    where $\Gamma$ is a positive, finite number depending only on $\vdim$.
\end{theorem}
\begin{proof}
    Assume $c = 0$, hence $h = \big |\perpproject{T}|U \big |$, and notice that
    in view of \ref{remark:basic_orlicz}, \cite[3.2(2) and 4.12(1)]{MR0307015},
    one may employ homotheties to reduce the problem to the case
    $r=1$. Abbreviate $\mu = \eqLpnorm{\| V \| \restrict C \cap H}{\Phi}{h}$ and
    denote $(\| V \|, \vdim )$ approximate differentials by ``$\ap \Der $''.

    Select $\phi \in \mathscr{D} (U,\rel)$ with
    \begin{gather*}
        0 \leq \phi \leq 1, \quad \spt \phi \subset C, \quad \{ z \with \oball
        z1 \subset C \} \subset \{ z \with \phi (z) = 1 \}, \quad | \Der  \phi |
        \leq 2.
    \end{gather*}
    Using \cite[4.7]{snulmenn.decay} with $\delta = \frac{1}{40}$, one obtains
    a Borel function $f : U \to \{ t \with 0 \leq t \leq 1 \}$ with $f| U
    \without K = 0$ such that the varifolds $V_1, V_2 \in \RVar_\vdim ( U )$
    defined by
    \begin{gather*}
        V_1 (A) = \tint{A}{\ast} f(z) \ud V(z,S) \quad \text{for $A \subset U
          \times \grass{\adim}{\vdim}$}
    \end{gather*}
    and $V_2 = V-V_1$ satisfy
    \begin{gather*}
        \text{$f(z) = 1$ and $\ap \Der f(z) = 0$ for $\| V \|$ almost all $z \in U
          \without H$}, \\
        \tint{}{} \phi(z)^2 \| \project{S} - \project{T} \|^2 \ud V_1 (z,S) \leq
        4 \| V_1 \| (K) \leq \Delta \| \delta V \| (K)^{\vdim/(\vdim-1)}, \\
        \| \delta V_2 \| \leq (1-f) \| \delta V \| + | \ap \Der f | \| V \|, \quad
        \| V \| ( | \ap \Der f| ) \leq (400)^\vdim \| \delta V \| ( K ),
    \end{gather*}
    where $\Delta = 4
    (400)^{\vdim^2/(\vdim-1)}(\isoperimetric{\vdim}\vdim)^{\vdim/(\vdim-1)}$;
    compare \cite[p.~24, l.~14--20]{snulmenn.decay}. In~particular, one infers
    \begin{gather*}
    	\| V_2 \| \leq \| V \| \restrict H, \quad \| \delta V_2 \| ( K ) \leq
        (800)^\vdim \| \delta V \| ( K  ).
    \end{gather*}

    Defining $g = \phi^2 ( \perpproject{T}|U)$, one derives
    \begin{gather*}
        \tint{}{} \phi (z)^2 \| \project{S} - \project{T} \|^2 \ud V_2 (z,S)
        \leq \sup \big \{ 16 \tint{}{} |\Der \phi|^2 |h|^2 \ud \| V_2 \|, 2 | (
        \delta V_2 ) ( g ) | \big \}
    \end{gather*}
    as in Brakke~\cite[5.5, p.~139, l.~1--14]{MR485012}, hence
    \begin{multline*}
    	\tint{\{ z \with \oball z1 \subset C \} \times \grass \adim \vdim}{}
        \| \project{S} - \project{T} \|^2 \ud V(z,S) \\
        \leq \Delta \| \delta V \| (K)^{\vdim/(\vdim-1)} + 2 | ( \delta V_2 )
        (g) | + 64 \tint{C \cap H}{} |h|^2 \ud \| V \|.
    \end{multline*}
    If $\mu = 0$, then $g(z) = 0$ for $\| V \|$ almost all $z \in H$, hence
    $\Der g(z)|\Tan^\vdim ( \| V \|, z ) = 0$ for $\| V \|$ almost all $z \in H$
    by \cite[2.10.19\,(4), 3.2.16]{MR41:1976} and $( \delta V_2 ) (g) = 0$.
    Therefore one may assume $\mu > 0$.

    In order to estimate $| ( \delta V_2 ) ( g ) |$, suppose $0 < t < \infty$,
    define $\eta : \{ s \with 0 \leq s < \infty \} \to \rel$ by
    \begin{gather*}
        \eta (0) = 1, \qquad \eta (s) = \inf \{ 1, t/s \} \quad \text{for $0
          < s < \infty$}.
    \end{gather*}
    Moreover, let $Z = \{ z \with t < h(z) \}$ and define Lipschitzian maps by
    \begin{gather*}
        g_1 = \phi^2 ( \eta \circ h ) \perpproject{T}|U, \quad g_2 =
        g-g_1.
    \end{gather*}
    Since $g_2 | U \without Z = 0$, one notices that
    \begin{gather*}
    	\ap \Der  g_2 (z) = 0 \quad \text{for $\| V \|$ almost all $z \in U
          \without Z$}
    \end{gather*}
    by \cite[2.10.19\,(4)]{MR41:1976}. Additionally, one computes
    \begin{gather*}
        | g_1 (z) | \leq t \quad \text{for $z \in U$}, \qquad \| \Der 
        g_2 (z) \| \leq 2 \phi(z)^2 + | \Der \phi(z)|^2 h(z)^2 \quad
        \text{for $z \in Z$},
    \end{gather*}
    see the case $r=1$ of \cite[4.10, p.~24, l.~26 -- p.~25,
    l.~12]{snulmenn.decay}. It follows that
    \begin{gather*}
        \| \Der g_2 (z) \| \leq 2 \phi (z)^2 \Phi (t/\mu)^{-1} \Phi (
        h(z)/\mu) + | \Der \phi(z)|^2 h(z)^2 \quad \text{for $z \in Z$}.
    \end{gather*}
    Therefore one estimates, using \cite[4.5\,(4)]{snulmenn.decay} and
    \ref{remark:orlicz_eq},
    \begin{align*}
        | ( \delta V_2 ) (g) | & \leq t \| \delta V_2 \| ( K ) + 2 \Phi
        (t/\mu)^{-1} \tint{C}{} \Phi \circ | \mu^{-1} h | \ud \| V_2 \| + 4
        \tint{C}{} |h|^2 \ud \| V_2 \| \\
        & \leq (800)^\vdim t \| \delta V \| (K) + 2 \Phi(t/\mu)^{-1} + 4
        \tint{C \cap H}{} |h|^2 \ud \| V \|,
    \end{align*}
    hence, noting \ref{miniremark:Phi} with $\alpha = \|\delta V\|(K) \mu$, one
    may take $\Gamma = \Delta + (1600)^\vdim$.
\end{proof}
\begin{remark} \label{remark:tv-cutoff}
    Notice that the function $f$ furnished by \cite[4.7]{snulmenn.decay} will
    necessarily be discontinuous in some cases, see
    \cite[4.8]{snulmenn.decay}. However, inspecting the proof of
    \cite[4.7]{snulmenn.decay} and using \cite[8.7]{snulmenn:tv.v2}, one is at
    least assured that $f$ is generalised weakly differentiable in the sense
    of \cite[8.3]{snulmenn:tv.v2} with
    \begin{gather*}
    	\derivative{V}{f} (z) = ( \| V \|, \vdim ) \ap \Der f(z) \circ \project {
	\Tan^\vdim ( \| V \|, z ) } \quad \text{for $\| V \|$ almost all $z$}.
    \end{gather*}
\end{remark}


\section{Approximation} \label{sec:approximation}

In this section we construct a real valued Lipschitzian auxiliary function
from an integral varifold in a cylinder whose first variation is representable
by integration, see \ref{lemma:mini-lip-approx}. This auxiliary function
captures information on the height-excess measured in Lebesgue and Orlicz
spaces and the quadratic tilt-excess of the varifold. In conjunction with the
basic coercive estimate in \ref{lemma:coercive-estimate} and the interpolation
inequalities of Section \ref{sec:interpolation} it will be used in
\ref{lemma:app-coercive} to obtain a coercive estimate involving an
approximate height quantity.

The auxiliary function is constructed as ``upper envelope of the
modulus'' of an approximating Lipschitzian $\qspace_Q ( \rel^\codim
)$~valued function constructed in
\cite[3.15]{snulmenn.poincare}. Approximations by $\qspace_Q (
\rel^\codim)$~valued functions are a powerful tool, originating from
Almgren \cite[3.1--3.12]{MR1777737}, whose handling is at times
complex. The fact that in the present setting we are able to
encapsulate their usage in the construction of the real valued
auxiliary function considerably simplifies our proof of decay rates
for the quadratic tilt-excess in \ref{theorem:quadratic-tilt-decay}.

\begin{definition}[see Almgren \protect{\cite[1.1\,(1)\,(3), 2.3\,(2)]{MR1777737}}] \label{miniremark:almgren1}
    Suppose $Q$ is a positive integer and $Y$ is a finite dimensional
    inner product space. Then
    \begin{gather*}
    	{\textstyle \qspace_Q (Y) = \left \{ \sum_{i=1}^Q \Lbrack y_i \Rbrack
	\with y_1, \ldots, y_Q \in Y \right \}}
    \end{gather*}
    is metrised by $\mathscr{G}$ such that, whenever $y_1,\ldots,y_Q \in Y$
    and $\upsilon_1,\ldots,\upsilon_Q \in Y$, $\mathscr{G} \big ( \sum_{i=1}^Q
    \Lbrack y_i \Rbrack , \sum_{i=1}^Q \Lbrack \upsilon_i \Rbrack \big)$
    equals the infimum of the set of numbers
    \begin{gather*}
	{\textstyle \left ( \sum_{i=1}^Q \left | y_i - \upsilon_{\pi(i)} \right
	|^2 \right)^{1/2}}
    \end{gather*}
    corresponding to all permutations $\pi$ of $\{ 1, \ldots, Q \}$.
\end{definition}

\begin{definition}[see Almgren \protect{\cite[1.1\,(9)\,(10)]{MR1777737}}]
    \label{miniremark:almgren2}
    Suppose $\vdim$ and $Q$ are positive integers and $Y$ is a finite
    dimensional inner product space.

    A function $f : \rel^\vdim \to \qspace_Q(Y)$ is called \emph{affine} if
    and only if there exist affine functions $f_i : \rel^\vdim \to Y$
    corresponding to $i = 1,\ldots,Q$ such that
    \begin{gather*}
        f (x) = {\textstyle\sum_{i=1}^Q} \Lbrack f_i (x) \Rbrack
        \quad \text{whenever $x \in \rel^\vdim$}
    \end{gather*}
    and in this case $\| f \| = \Lip f$. Moreover, if $a \in A \subset
    \rel^\vdim$ and $f : A \to \qspace_Q ( Y )$ then $f$ is \emph{affinely
    approximable at $a$} if and only if $a \in \Int A$ and there exists an
    affine function $g : \rel^\vdim \to \qspace_Q ( Y)$ such that
    \begin{gather*}
	g(a) = f(a) \quad \text{and} \quad \lim_{x \to a} \mathscr{G} ( f(x),
	g(x) ) / |x-a| = 0.
    \end{gather*}
    The function $g$ is unique and denoted by $\Aff f(a)$. The concept of
    \emph{approximate affine approximability} is obtained through replacement
    of the condition $a \in \Int A$ by $a \in A$ and replacement of $\lim$ by
    $\aplim$. The corresponding affine function is denoted by $\ap \Aff f(a)$.
\end{definition}
\begin{remark}
    \label{remark:apAf-apDf}
    In comparison to Almgren \cite[1.1\,(10)]{MR1777737}, the requirement
    ``$g(a)=f(a)$'' has been added.  Consequently, [approximate] affine
    approximability implies [approximate] continuity. Moreover, supposing $Q =
    1$ and denoting by $i : \qspace_1 (Y) \to Y$ the canonical isometry, the
    function $f$ is [approximately] affinely approximable at $a$ if and only if
    $i \circ f$ is [approximately] differentiable at $a$, see \cite[3.1.1,
    3.1.2]{MR41:1976}, and in this case
    \begin{gather*}
        i \circ \ap \Aff f(a) = i(f(a)) + \ap \Der (i \circ f)(a).
    \end{gather*}
\end{remark}
\begin{miniremark}
    \label{miniremark:Q-affine}
    Suppose $Q$ is a positive integer, $Y$ is a finite dimensional inner product
    space, $a \in \rel^\vdim$, and $f : \rel^\vdim \to \qspace_Q (Y)$ is
    affine. Then
    \begin{gather*}
        \Lip f = \limsup_{x \to a } |x-a|^{-1} \mathscr{G} (f(x),f(a))
        = \aplimsup_{x \to a} |x-a|^{-1} \mathscr{G} (f(x),f(a));
    \end{gather*}
    in fact, in view of \cite[1.1\,(9)]{MR1777737} only the last equation needs
    to be proven. For this purpose denote the approximate limit superior by
    $\lambda$ and define the Lipschitzian function $g : \rel^\vdim \to \rel$ by
    $g(x) = \mathscr{G} (f(x),f(a)) - \lambda |x-a|$ for $x \in \rel^\vdim$. One
    infers $\ap \Der  g^+ (a) = 0$, whence it follows $\Der g^+ (a) =0$ by
    \cite[3.1.5]{MR41:1976} with $C$, $B$, $f$, $\eta$, and $M$ replaced by
    $\rel^\vdim$, $\rel^\vdim$, $g^+$, $1$, and $\Lip g$; therefore $\limsup_{x
      \to a} |x-a|^{-1} \mathscr{G}(f(x),f(a)) \leq \lambda$. The reverse
    inequality follows since $a \in \Int \dmn f$.
\end{miniremark}
\begin{miniremark}
    \label{miniremark:Q-affine-approx}
    Suppose $Q$ is a positive integer, $Y$ is a finite dimensional inner product
    space, $a \in \rel^\vdim$, $f$ maps a subset of $\rel^\vdim$ into
    $\qspace_Q(Y)$, and $f$ is approximately affine approximable at~$a$. Then
    \ref{miniremark:Q-affine} implies
    \begin{gather*}
        \| \ap \Aff f(a) \| = \aplimsup_{x \to a} |x-a|^{-1} \mathscr{G}(f(x),f(a)).
    \end{gather*}
\end{miniremark}
\begin{lemma}
    \label{lemma:q-sup-function}
    Suppose $Q$ is a positive integer, $Y$ is a finite dimensional inner product
    space, and $\sigma : \qspace_Q ( Y ) \to \rel$ satisfies
    \begin{gather*}
        \sigma (S) = \sup \{ |y| \with y \in \spt S \}
        \quad \text{for $S \in \qspace_Q (Y)$}.
    \end{gather*}
    
    Then $\Lip \sigma \leq 1$.
\end{lemma}
\begin{proof}
    One may express $\sigma = p \circ \xi \circ g$, where $g : \qspace_Q(Y) \to
    \qspace_Q(\rel)$ denotes the push forward induced by the norm on $Y$ mapping
    $Y$ into $\rel$, and $\xi : \qspace_Q( \rel ) \to \rel^Q$ and $p : \rel^Q
    \to \rel$ are characterised by
    \begin{gather*}
        \xi \big ( \tsum{i=1}{Q} \Lbrack y_i \Rbrack \big ) = (y_1, \ldots, y_Q )
        \quad \text{if $y_i \leq y_{i+1}$ for $i = 1, \ldots, Q-1$},
        \\
        p ( y_1, \ldots, y_Q ) = y_Q
    \end{gather*}
    whenever $(y_1, \ldots, y_Q ) \in \rel^Q$. Clearly, $\Lip p \leq 1$.
    Moreover, one readily verifies $\Lip g \leq 1$.  Finally, $\Lip \xi \leq 1$
    by Almgren \cite[1.1\,(4)]{MR1777737}.
\end{proof}
\begin{lemma}
    \label{lemma:comp-Q-valued}
    Suppose $Q$ is a positive integer, $Y$ is a finite dimensional inner product
    space, $a \in \rel^\vdim$, $f$ maps a subset of $\rel^\vdim$ into
    $\qspace_Q(Y)$, and $\sigma : \qspace_Q (Y) \to \rel$ is Lipschitzian.

    Then the following two statements hold.
    \begin{enumerate}
    \item \label{item:comp-Q-valued:classical} If $f$ is affinely approximable
        at~$a$ and $\sigma \circ f$ is differentiable at~$a$, then $| \Der (\sigma
        \circ f)(a)| \leq \Lip (\sigma) \| \Aff f(a) \|$.
    \item \label{item:comp-Q-valued:approximate} If $f$ is approximately
        affinely approximable at~$a$ and $\sigma \circ f$ is approximately
        differentiable at~$a$, then $| \ap \Der  ( \sigma \circ f ) (a) | \leq
        \Lip(\sigma) \| \ap \Aff f(a) \|$.
    \end{enumerate}
\end{lemma}
\begin{proof}
    \eqref{item:comp-Q-valued:approximate} is a consequence
    of~\ref{miniremark:Q-affine-approx} together with~\ref{remark:apAf-apDf}
    and implies \eqref{item:comp-Q-valued:classical}.
\end{proof}
\begin{miniremark}
    \label{remark:unitmeasure}
	Notice that
	\begin{gather*}
		\sup \{ \unitmeasure{\vdim} \with \vdim \in \nat \} < 6;
	\end{gather*}
	in fact, using $3 < \boldsymbol{\Gamma} ( 1/2 )^2 < 3.2$ and $(\vdim+2)
    \unitmeasure{\vdim+2} = 2 \boldsymbol{\Gamma} (\frac{1}{2})^2
    \unitmeasure{\vdim}$ for $\vdim \in \nat$ by \cite[3.2.13]{MR41:1976}, one
    obtains $\unitmeasure{6} = \frac{1}{6} \boldsymbol{\Gamma}( \frac{1}{2} )^6
    < \frac{8}{15} \boldsymbol{\Gamma} ( \frac{1}{2} )^4 = \unitmeasure{5} < 6$
    and the supremum does not exceed $\unitmeasure{5}$.
\end{miniremark}
\begin{theorem}
    \label{lemma:mini-lip-approx}
    Suppose $\vdim, \adim, Q \in \nat$ and $1 <\vdim < \adim$.

    Then there exists a positive, finite number $\Gamma$ with the following
    property.

    If $0 < r < \infty$, $T = \im \pp^\ast$, $V \in \IVar_\vdim ( \rel^\adim
    \cap \oball{0}{4r} )$,
    \begin{gather*}
        (Q-1/2) \unitmeasure{\vdim} r^\vdim
        \leq \| V \| ( \cylinder{T}{0}{r}{r} )
        \leq (Q+1/2) \unitmeasure{\vdim} r^\vdim,
        \\
        \| V \| ( \cylinder{T}{0}{r}{2r} \without \cylinder{T}{0}{r}{r/2} )
        \leq (1/2) \unitmeasure{\vdim} r^\vdim,
        \\
	\measureball{\| V \|}{\oball{0}{4r}} \leq Q \unitmeasure{\vdim}
	(5r)^\vdim, \\
	\eta = \| \delta V \| ( \oball 0{4r} )^{\vdim/(\vdim-1)} + \tint{}{}
	\| \project{S} - \project{T} \|^2 \ud V (z,S),
    \end{gather*}
    $H$ consists of all $z \in \cylinder T0rr$ such that
    \begin{gather*}
        \measureball{\| V \|}{\cball{z}{s}} 
        \geq (40\isoperimetric{\vdim}\vdim)^{-\vdim} s^\vdim
        \quad \text{whenever $0 < s < 2r$},
    \end{gather*}
    and $\Phi$ is as in \ref{miniremark:Phi}, then there exists a Borel
    subset~$X$ of $\rel^\vdim \cap \cball{0}{r}$ and a function $f : X \to \rel$
    with $\Lip f \leq 1$ satisfying the following five conditions whenever $1
    \leq q \leq \infty$ and $A$ is a subset of~$X$:
    \begin{enumerate}
	\item \label{item:mini-lip-approx:bad-set} $\mathscr{L}^\vdim ( \cball 0r
        \without X ) \leq \Gamma \eta$.
	\item \label{item:mini-lip-approx:2-height-control} $\eqLpnorm{\| V \|
          \restrict H}{q} {\perpproject{T}} \leq \Gamma \big (
        \eqLpnorm{\mathscr{L}^\vdim \restrict X}{q}{f} + \eta^{1/q+1/\vdim} \big
        )$.
	\item \label{item:mini-lip-approx:Phi-height-control} $\eqLpnorm{\| V \|
          \restrict H}{r^{-\vdim} \Phi} { \perpproject{T} } \leq \Gamma \big (
        \eqLpnorm{\mathscr{L}^\vdim \restrict X}{r^{-\vdim} \Phi} f +
        \eta^{1/\vdim} \big )$.
	\item \label{item:mini-lip-approx:inverse-2-height-control}
        $\eqLpnorm{\mathscr{L}^\vdim \restrict A}{2}{f} \leq \eqLpnorm{\| V \|
          \restrict H \cap \pp^{-1} \lIm A \rIm}{2}{\perpproject{T}}$.
	\item \label{item:mini-lip-approx:tilt-control} $\eqLpnorm{\mathscr{L}^\vdim
          \restrict X}{2}{\ap \Der f} \leq (2Q \eta)^{1/2}$.
    \end{enumerate}
\end{theorem}
\begin{proof}
    Notice that $(\isoperimetric{\vdim} \vdim )^{-\vdim} \leq \unitmeasure
    \vdim$, see for instance~\cite[2.4]{snulmenn.isoperimetric}. Define $\beta
    = \vdim/(\vdim-1)$, and
    \begin{gather*}
        \Delta_1 = \varepsilon_{\text{\cite[3.15]{snulmenn.poincare}}} 
        \left( \codim, \vdim, Q, 1, 5^\vdim Q, \tfrac 12, \tfrac 12, \tfrac 12, \tfrac 14, 
        (40 \isoperimetric{\vdim}\vdim)^{-\vdim}/\unitmeasure{\vdim} \right),
        \\
        \Delta_2 = ( \log 2 )^{1/\beta},
        \quad
        \Delta_3 = 1/ \Phi^{-1} ( 1/(6(Q+1))),
        \quad
        \Delta_4 = ( 20 \isoperimetric{\vdim} \vdim )^{-\vdim} \Delta_1^\beta,
        \\
        \Delta_5 = \Delta_1^2 ( Q \unitmeasure{\vdim} \adim)^{-1} 
        ( 60 \isoperimetric{\vdim} \vdim )^{-2\vdim},
        \\
        \Delta_6 = \sup \{ 3 + 2Q + (12Q + 6) 5^\vdim, 8 (Q+2) \},
        \\
        \Delta_7 = (1/2) \unitmeasure{\vdim}
        \lambda_{\text{\cite[3.15\,(4)]{snulmenn.poincare}}} ( \vdim, 1/2, 1/4)^\vdim 6^{-\vdim},
        \\
        \Delta_8 = \sup \{ \Gamma_{\text{\cite[3.15\,(6)]{snulmenn.poincare}}} ( \vdim ),
        2 \unitmeasure{\vdim}^{-1/\vdim} \},
        \quad
        \Delta_9 = \Delta_6 \adim \besicovitch{\adim} \sup \{ 1,\Delta_1^{-2} \},
        \\
        \Delta_{10} = (12)^{\vdim+1} Q \sup \{ Q, \Delta_8 \Delta_9^{1+1/\vdim} \},
        \quad
        \Delta_{11} = \inf \{ 1, \Delta_4, \Delta_5, \Delta_9^{-1} \Delta_7 \},
        \\
        \Delta_{12} = 2 \Delta_{10} \Delta_2^{-1},
        \quad
        \Delta_{13} = \sup \{ 6 (Q+1) \Delta_{11}^{-1-1/\vdim}, \Delta_3 \Delta_{11}^{-1/\vdim} \},
        \\
        \Gamma = \sup \{ \Delta_9, \Delta_{10}, \Delta_{12}, \Delta_{13} \}.
    \end{gather*}
    Notice that $\Delta_2 < 1 \leq \Delta_9$.

    Suppose $r$, $T$, $V$, $\eta$, $H$, and $\Phi$ are related to $\vdim$,
    $\adim$, and $Q$ as in the body of the lemma. Since the statement of the
    lemma is invariant by replacing $V$, $f$ with
    $(\boldsymbol{\mu}_{(1/r)})_{\#} V$, $r^{-1} f \circ \boldsymbol{\mu}_{r}$,
    we can assume $r=1$.

    One may also assume $\eta \leq \Delta_{11}$ since otherwise
    \begin{gather*}
        \measureball{\mathscr{L}^\vdim}{\cball 01} 
        \leq 6 \leq \Delta_{13} \Delta_{11} \leq \Gamma \eta,
        \\
        \eqLpnorm{\| V \| \restrict \cylinder T011}{q}{ \perpproject{T} } 
        \leq 6 (Q+1) \leq \Delta_{13} \Delta_{11}^{1+1/\vdim} \leq \Gamma \eta^{1/q+1/\vdim},
        \\
        \eqLpnorm{\| V \| \restrict \cylinder T011}{\Phi}{ \perpproject{T} }
        \leq \Delta_3 \leq \Delta_{13} \Delta_{11}^{1/\vdim} \leq \Gamma \eta^{1/\vdim}
    \end{gather*}
    by \ref{remark:unitmeasure} and \ref{miniremark:Phi}, hence one may take
    $X = \varnothing$ and $f = \varnothing$.

    One applies \cite[3.15]{snulmenn.poincare} with
    \begin{gather*}
    	\text{$m$, $n$, $L$, $M$, $\delta_1$, $\delta_2$, $\delta_3$,
          $\delta_4$, $\delta_5$, $a$, $h$, $\mu$, and $\varepsilon_1$}
        \quad \text{replaced by} \\
        \text{$\codim$, $\vdim$, $1$, $5^\vdim Q$, $\tfrac 12$, $\tfrac 12$, $\tfrac 12$, $\tfrac 14$,
          $(40 \isoperimetric{\vdim}\vdim)^{-\vdim}/\unitmeasure{\vdim}$, $0$,
          $r$, $\| V \|$, and $\Delta_1$}
    \end{gather*}
    to obtain $\bar B$, $\bar{f}$ and $\bar{H}$ named $B$, $f$ and $H$ there.

    First, \emph{it will be shown that $H = \bar{H}$}; in fact, noting $\eta
    \leq \inf \{ \Delta_4, \Delta_5 \}$, one estimates
    \begin{gather*}
	\measureball{\| \delta V \|}{ \oball{z}{2} } \leq \eta^{1/\beta} \leq
	\Delta_4^{1/\beta} \leq \Delta_1 \| V \|( \oball{z}{2} )^{1/\beta}, \\
	\begin{aligned}
	    \tint{\oball{z}{2} \times \grass{\adim}{\vdim}}{} | \project{S}-
	    \project{T}| \ud V(z,S) & \leq \| V \| ( \oball 04)^{1/2}
	    \adim^{1/2} \eta^{1/2} \\
	    & \leq ( Q \unitmeasure{\vdim})^{1/2} 3^\vdim \adim^{1/2}
	    \Delta_5^{1/2} \leq \Delta_1 \measureball{\| V \|}{\oball z2}.
	\end{aligned}
    \end{gather*}
    whenever $z \in \cylinder T011$ and $\measureball{\| V \|}{ \oball z2 }
    \geq ( 20 \isoperimetric{\vdim} \vdim )^{-\vdim}$.

    Choose a Borel subset $X$ of $\dmn \bar{f}$ with $\mathscr{L}^\vdim ( (\dmn
    \bar f) \without X ) = 0$ and define $f : X \to \rel$ by
    \begin{gather*}
	f(x) = \sup \{ |y| \with y \in \spt \bar f(x) \} \quad \text{whenever
	$x \in X$}.
    \end{gather*}
    Clearly, $\Lip f \leq \Lip \bar f \leq 1$ by \ref{lemma:q-sup-function} and
    one infers
    \begin{gather*}
    	| \ap \Der f(x) | \leq \| \ap \Aff  \bar f (x) \| 
        \quad \text{for $\mathscr{L}^\vdim$ almost all $x \in X$} ;
    \end{gather*}
    in fact, $\bar f$ is approximately affinely approximable at $\LM^\vdim$
    almost all $x \in X$ by~\cite[3.15\,(7a)]{snulmenn.poincare} and $f$ is
    approximately differentiable by~\cite[2.8.18, 2.9.11, 3.1.8]{MR41:1976} at
    $\LM^\vdim$ almost all $x \in X$ so the assertion follows from
    \ref{lemma:q-sup-function} and
    \ref{lemma:comp-Q-valued}\,\eqref{item:comp-Q-valued:approximate}.

    Next, \emph{it will be proven that $\mathscr{L}^\vdim ( \cball 01 \without
    X ) \leq \Delta_9 \eta$}. For this purpose define sets $B_1$ and $B_2$
    consisting of those $z \in \cylinder T011$ satisfying
    \begin{gather*}
        \text{$
          \measureball{\| \delta V \|} { \cball zs } > \Delta_1 \| V \| ( \cball zs )^{1/\beta}
          $ for some $0<s<2 $},
        \\
        \text{$
          \tint{\cball zs \times  \grass \adim \vdim}{} \| \project{S} - \project{T} \|^2 \ud V(z,S)
          > \Delta_1^2 \adim^{-1} \measureball{\| V \|}{ \cball zs }
          $ for some $0 < s < 2$}
    \end{gather*}
    respectively. To estimate $\| V \| ( B_1 )$ we employ the
    Besicovitch-Federer covering theorem which provides disjointed families
    $F_1, \ldots, F_{\besicovitch{\adim}}$ of closed balls such that
    \begin{gather*}
    	{\textstyle B_1 \subset \bigcup \bigcup \{ F_i \with i = 1, \ldots, \besicovitch \adim \}
          \subset \oball 04},
        \\
        \| V \| (C) < \Delta_1^{-\beta} \| \delta V \| ( C )^\beta 
        \quad \text{whenever $C \in F_i$ and $i = 1, \ldots, \besicovitch{\adim}$},
    \end{gather*}
    and we obtain
    \begin{align*}
    	\| V \| ( B_1 ) & \leq \Delta_1^{-\beta}
        \tsum{i=1}{\besicovitch \adim} \tsum{C \in F_i}{} \| \delta V \| ( C )^\beta 
        \\
        & \leq \Delta_1^{-\beta} \tsum{i=1}{ \besicovitch \adim } \big(
        \tsum{C \in F_i}{} \| \delta V \| ( C ) \big)^\beta 
        \leq \Delta_1^{-\beta} \besicovitch{\adim} \| \delta V \| ( \oball 04 )^\beta.
    \end{align*}
    In a similar fashion we find another disjointed families $F_1, \ldots,
    F_{\besicovitch \adim}$ of closed balls such that
    \begin{gather*}
    	{\textstyle B_2 \subset \bigcup \bigcup \{ F_i \with i = 1, \ldots, \besicovitch \adim \}
          \subset \oball 04},
        \\
        \| V \| (C) < \Delta_1^{-2} \adim
        \tint{C \times \grass \adim \vdim}{} \| \project{S} - \project{T} \|^2 \ud V(z,S)
        \quad \text{for $C \in F_i$, $i = 1, \ldots, \besicovitch \adim$},
    \end{gather*}
    and in consequence
    \begin{align*}
        \| V \| ( B_2 ) & \leq \Delta_1^{-2} \adim \tsum{i = 1}{\besicovitch \adim}
        \tsum{C \in F_i}{} \tint{C \times \grass \adim \vdim}{} \| \project{S} - \project{T} \|^2 \ud V(z,S)
        \\
        & \leq \Delta_1^{-2} \adim \besicovitch \adim \tint{}{} \| \project{S} - \project{T} \|^2 \ud V(z,S).
    \end{align*}
    Verifying $\bar B \subset B_1 \cup B_2$ by means of H{\"o}lder's inequality,
    the asserted estimate follows from \cite[3.15\,(3)]{snulmenn.poincare}.

    Since in particular $\mathscr{L}^\vdim ( \cball 01 \without X ) \leq
    \Delta_9 \Delta_{11} \leq \Delta_7$, one applies
    \cite[3.15\,(6)]{snulmenn.poincare} with $S$ replaced by $Q \Lbrack 0
    \Rbrack$ to estimate, concerning
    \eqref{item:mini-lip-approx:2-height-control},
    \begin{align*}
        & \eqLpnorm{\| V \| \restrict H}{q}{ \perpproject T} \\
        & \qquad \leq (12)^{\vdim+1}
        Q \big ( Q^{1/2} \eqLpnorm{\mathscr{L}^\vdim \restrict X} qf 
        + \Delta_8 \mathscr{L}^\vdim ( \cball 01 \without X )^{1/q+1/\vdim} \big ) \\
        & \qquad \leq \Delta_{10} \big ( \eqLpnorm{\mathscr{L}^\vdim \restrict X} qf 
        + \eta^{1/q+1/\vdim} \big )
        \quad \text{for $1 \le q \le \infty$.}
    \end{align*}
    Consequently, concerning \eqref{item:mini-lip-approx:Phi-height-control},
    one notes that $\eta \leq 1$ and $\Phi ( \Delta_2 ) = 1$ and estimates
    \begin{align*}
        \tint{H}{} \Phi \circ | \gamma^{-1} \perpproject T | \ud \| V \| 
        & = \tsum{i=1}{\infty} i!^{-1} \gamma^{-i\beta} \tint{H}{} | \perpproject{T} |^{\beta i} \ud \| V \| 
        \\
        & \leq \tfrac 12 \tsum{i=1}{\infty} i!^{-1} ( 2 \Delta_{10}/ \gamma )^{\beta i}
        \big ( \tint{X}{} |f|^{\beta i} \ud \mathscr{L}^\vdim + \eta^{1+\beta i/\vdim} \big )
        \\
        & = \tfrac 12 \tint{X}{} \Phi \circ | 2 \Delta_{10} \gamma^{-1} f | \ud \mathscr{L}^\vdim
        + \tfrac 12 \eta \Phi ( 2 \Delta_{10} \eta^{1/\vdim} \gamma^{-1} ) \leq 1
    \end{align*}
    whenever $2 \Delta_{10} \Delta_2^{-1} \big ( \eqLpnorm{\mathscr{L}^\vdim
      \restrict X}{\Phi}{f} + \eta^{1/\vdim} \big ) < \gamma < \infty$, hence
    \begin{gather*}
    	\eqLpnorm{\| V \| \restrict H}{\Phi}{\perpproject{T}}
        \leq \Delta_{12} \big ( \eqLpnorm{\mathscr{L}^\vdim \restrict X}{\Phi}{f} 
        + \eta^{1/\vdim} \big ).
    \end{gather*}

    To prove \eqref{item:mini-lip-approx:inverse-2-height-control} and
    \eqref{item:mini-lip-approx:tilt-control}, recall
    \begin{gather*}
        H \cap \pp^{-1} \lIm \dmn \bar f \rIm 
        = \{ z \with \qq (z) \in \spt \bar f ( \pp (z)) \} 
        \subset \{ z \with \density^\vdim ( \| V \|, z) \in \nat \}
    \end{gather*}
    from \cite[3.15\,(2)\,(4)]{snulmenn.poincare} and observe: \emph{If $A$ is
      a subset of $X$, $g$ is an $\mathscr{L}^\vdim \restrict A$ measurable real
      valued function and $h$ is an $\| V \| \restrict H \cap \pp^{-1} \lIm A
      \rIm$ measurable real valued function such that
      \begin{gather*}
          \mathscr{L}^\vdim ( A \cap \{x \with g(x)> t \} \without \pp \lIm H
          \cap \{ z \with h(z) > t \} \rIm ) = 0 \quad \text{for $0 < t <
            \infty$},
      \end{gather*}
      then $\eqLpnorm{\mathscr{L}^\vdim \restrict A} q g \leq \eqLpnorm{\| V \|
        \restrict H \cap \pp^{-1} \lIm A \rIm} qh$;} in fact
    \begin{align*}
    	\mathscr{L}^\vdim ( A \cap \{ x \with g(x) > t \} ) 
        & \leq \mathscr{H}^\vdim ( H \cap \pp^{-1} \lIm A \rIm \cap \{ z \with h(z) > t \} )
        \\
        & \leq \| V \| ( H \cap \pp^{-1} \lIm A \rIm \cap \{ z \with h(z) > t \} )
    \end{align*}
    by \cite[2.10.35]{MR41:1976} and Allard \cite[3.5\,(1b)]{MR0307015}.

    One applies this observation with $g$ and $h$ replaced by $f$ and $\big |
    \perpproject T | U \big |$ to deduce
    \eqref{item:mini-lip-approx:inverse-2-height-control}. Recalling $|\ap
    \Der f(x)| \le \|\ap \Aff f(x) \|$ for $\LM^\vdim$ almost all $x \in X$ together
    with~\cite[3.15\,(7d)]{snulmenn.poincare}, one applies the observation once
    more, with
    \begin{gather*}
        \text{$g(x)$ and $h(z)$} 
        \quad \text{replaced by} \quad
        \text{$| \ap \Der f(x) |$ and $(2Q)^{1/2} \project{\| \Tan^\vdim ( \| V \|, z )} - \project{T} \|$}
    \end{gather*}
    to infer \eqref{item:mini-lip-approx:tilt-control}. 
\end{proof}


\section{Embedding results}
\label{sec:interpolation}

In the present section we formulate for convenient reference two embedding
results for Sobolev functions in Euclidean space which measure the lower order
term only on a set of suitably large Lebesgue measure.

\begin{lemma} \label{lemma:int-ineq}
	Suppose $2 \leq \vdim \in \nat$, $a \in \rel^\vdim$, $0 < r < \infty$,
	$0 < \varepsilon \le (\unitmeasure{\vdim}/2)^{1/\vdim}$, $A$ is an
	$\LM^\vdim$ measurable subset of $\oball ar$, $\LM^\vdim(\oball ar
	\without A) \le (\varepsilon r)^\vdim$, and $f \in \Sob{}{1}{1}(\oball
	ar)$.

	Then there holds
	\begin{multline*}
		r^{-1} \eqLpnorm{\mathscr{L}^\vdim \restrict \oball
		ar}{\vdim}{f} \\
		\leq \Gamma \big ( \varepsilon^{1/2}
		\eqLpnorm{\mathscr{L}^\vdim \restrict \oball ar}{\vdim}{\weakD
		f} + \varepsilon^{-1/2} r^{-1} \eqLpnorm{\mathscr{L}^\vdim
		\restrict A}{\vdim}{f} \big ),
	\end{multline*}
	where $\Gamma$ is a positive, finite number depending only on $\vdim$.
\end{lemma}
\begin{proof}
	By H{\"o}lder's inequality it is sufficient to prove the statement
	that results from replacing $r^{-1} \eqLpnorm{\mathscr{L}^\vdim
	\restrict \oball ar}{\vdim}{f}$ by $r^{-1/2}
	\eqLpnorm{\mathscr{L}^\vdim \restrict \oball ar}{2\vdim}{f}$. The
	latter is a special case of \cite[6.3]{snulmenn.decay} taking $\zeta =
	\frac 23 \vdim$, $\xi = \vdim$, $s = \vdim$ and $\lambda =
	(\varepsilon r)^\vdim$.
\end{proof}
\begin{theorem}
	\label{corollary:orlicz-poincare}
	Suppose $2 \leq \vdim \in \nat$, $\Phi$ is related to $\vdim$ as in
	\ref{miniremark:Phi}, $a \in \rel^\vdim$, $0 < r < \infty$, $A$ is an
	$\LM^\vdim$ measurable subset of $\oball ar$ with $\LM^\vdim(A) \ge
	\frac 12 \unitmeasure{\vdim} r^\vdim$, and $f \in \Sob{}{1}{1}(\oball
	ar)$,

	Then there holds
	\begin{gather*}
		\eqLpnorm{\LM^\vdim \restrict \oball ar}{r^{-m}\Phi}{f} \le
		\Gamma \big ( \eqLpnorm{\mathscr{L}^\vdim \restrict \oball
		ar}{\vdim}{\weakD f} + r^{-1} \eqLpnorm{\mathscr{L}^\vdim
		\restrict A}{\vdim}{f} \big ),
	\end{gather*}
	where $\Gamma$ is a positive, finite number depending only on $\vdim$.

\end{theorem}
\begin{proof}
    The problem may be reduced firstly to the case $A = \oball ar$ by
    \ref{lemma:int-ineq} and secondly to the case $a = 0$ and $r = 1$ using
    translations and homotheties. The remaining case is a special case of
    \cite[8.27]{MR2424078}.
\end{proof}


\section{Quadratic tilt-excess, decay rates}
\label{sec:quadratic_tilt_decay}
In this section we prove sharp decay rates of the quadratic tilt-excess for two
dimensional integral varifolds whose first variation is a Radon measure, see
\ref{theorem:quadratic-tilt-decay}. This result rests on two pillars.  Firstly,
on the second order rectifiability of such varifolds obtained
in~\cite[4.8]{snulmenn.c2}. Secondly, on an approximate coercive estimate by
which we mean an estimate of the tilt-excess in terms of the first variation,
the height-excess measured on a set of suitably large weight measure and small
contributions from the tilt-excess, see \ref{lemma:app-coercive}.

Accordingly, in order to derive the approximate coercive estimate
\ref{lemma:app-coercive} from the coercive estimate
\ref{lemma:coercive-estimate}, one needs to estimate the height-excess
occurring in~\ref{lemma:coercive-estimate} by approximate height-quantities
together with the variation measure of the first variation and quadratic
tilt-excess. The approximation \ref{lemma:mini-lip-approx} reduces such an
estimate to the case of a real valued Lipschitzian functions which has been
treated in \ref{lemma:int-ineq} and \ref{corollary:orlicz-poincare}.

Since currently no analogous estimates to \ref{lemma:int-ineq} and
\ref{corollary:orlicz-poincare} are available for real valued Lipschitzian
function over varifolds, the authors have chosen the path using the
approximation by $\qspace_Q ( \rel^\codim )$ valued functions leading to
\ref{lemma:mini-lip-approx}. Yet, it would be of interest to investigate whether
the embedding theory for Lipschitzian functions on varifolds can be extended so
as to yield a proof without such approximation. In a somewhat different vein
much of that theory has been extended to generalised weakly differentiable
Functions in~\cite[\S 10]{snulmenn:tv.v2}.

\begin{lemma}
    \label{lemma:app-coercive}
    Suppose $2 < \adim \in \nat$, $Q \in \nat$, $c \in \rel^\adim$, $0 < r <
    \infty$, $V \in \IVar_2 ( \oball c{8r} )$,
    \begin{gather*}
        \measureball{\| V \|}{\cball c{2r}} \geq (Q-1/2) \unitmeasure 2 (2r)^2,
        \quad
        \measureball{\| V \|}{ \oball c{8r} } \leq (Q+1/4) \unitmeasure 2 (8r)^2,
    \end{gather*}
    $T \in \grass \adim 2$, $0 < \varepsilon \leq 1$, $Z$ is $\| V \|$
    measurable, and
    \begin{gather*}
        \alpha = r^{-1} \measureball{\| \delta V\|}{ \oball c{8r} },
        \quad
        \beta = r^{-1} \big ( \tint{}{} \| \project{S} - \project{T} \|^2 \ud V(z,S) \big )^{1/2},
        \quad
        \alpha + \beta \leq \varepsilon,
        \\
        \gamma = r^{-2} \big ( \tint{Z}{} \dist (z-c,T)^2 \ud \| V \| z \big)^{1/2},
        \quad
        \| V \| ( \oball c{8r} \without Z ) \leq (\varepsilon r )^2,
    \end{gather*}
    and $\kappa : \{ t \with 0 \leq t < \infty \} \to \rel$ satisfies (see~\ref{miniremark:kappa})
    \begin{gather*}
    	\kappa (0) = 0,
        \qquad
        \kappa (t) = t \big ( 1 + (\log (1+1/t))^{1/2} \big )
        \quad \text{for $0 < t < \infty$}.
    \end{gather*}

    Then there holds
    \begin{gather*}
    	r^{-2} \tint{\oball cr \times \grass \adim 2}{}
        \| \project{S} - \project{T} \|^2 \ud V(z,S)
        \leq \Gamma \big ( \kappa ( \alpha ( \alpha + \beta + \gamma ) ) 
        + \varepsilon \beta^2 + \varepsilon^{-1} \gamma^2 \big ),
    \end{gather*}
    where $\Gamma$ is a positive, finite number depending only on $\adim$
    and~$Q$.
\end{lemma}
\begin{proof}
    Considering $(\boldsymbol{\mu}_{1/r})_{\#}V$ in place of $V$ one may assume
    $r = 1$ and, using isometries, one may assume $c = 0$ and $T = \im
    \pp^\ast$, see \ref{miniremark:pqT}. Moreover, one may assume $Z$ to be a
    Borel set.

    Define
    \begin{gather*}
    	\Delta_1 = 1 + \Gamma_{\ref{lemma:mini-lip-approx}} ( 2 ,\adim,Q),
        \quad
    	\Delta_2 = \inf \big \{ 1/3, 2  \Delta_1^{-1/2} \big \},
        \\
    	\Delta_3 = 2 (Q+1/2)^{1/2} (Q+3/8)^{-1/2},
        \quad
        \Delta_4 = (2Q)^{1/2}+\Delta_1^{1/2},
        \\
        \Delta_5 = \Delta_1^{5/4} \Gamma_{\ref{lemma:int-ineq}}(2) 2^{1/2},
        \quad
        \Delta_6 = \inf \big \{ 1, ( \Delta_3^2 - 4 )^{1/2}/4 \big \},
        \quad
        \Delta_7 = \Delta_5 ( 1 + \Delta_4 ),
        \\
        \Delta_8 = 4 \Delta_1 ( \Gamma_{\ref{corollary:orlicz-poincare}} ( 2) +1 ) ( \Delta_4 + 1 ),
        \\
        \Delta_9 = \sup \{ \Delta_2^{-1}, \Delta_6^{-2} \},
        \quad
	\Gamma = \sup \{ \Delta_9, \Gamma_{\ref{lemma:coercive-estimate}} ( 2
	)  ( 1 + 3 \Delta_7^2 + \Delta_8 ) \}.
    \end{gather*}
    If $\varepsilon > \Delta_2$ then $\beta^2 \le \Delta_2^{-1} \varepsilon
    \beta^2 \leq \Delta_9 \varepsilon \beta^2$ and if $\gamma > \Delta_6$,
    then $\beta^2 \le \varepsilon^2 \le 1 \le \Delta_6^{-2} \varepsilon^{-1}
    \gamma^2 \leq \Delta_9 \varepsilon^{-1} \gamma^2$.  Therefore one may
    assume $\varepsilon \leq \Delta_2$ and $\gamma \leq \Delta_6$.

    Abbreviate $C = \cylinder T022$, $K = \rel^\adim \cap \{ z \with \dist
    (z,C) \leq 4 \}$, and
    \begin{gather*}
	H = C \cap \{ z \with \measureball{\| V \|}{\cball zs} \geq (80
	\isoperimetric{ 2 } )^{-2} s^2 \text{ for $0 < s < 4$} \}.
    \end{gather*}
    Notice that
    \begin{gather*}
    	\oball 01 \subset \{ z \with \oball z1 \subset C \}, 
        \quad
        K \subset \oball 08,
        \\
        C \cap H_{\ref{lemma:coercive-estimate}} \subset H,
        \quad
        \text{where $H_{\ref{lemma:coercive-estimate}}$ denotes the set
          named ``$H$'' in \ref{lemma:coercive-estimate}}.
    \end{gather*}

    In order to apply \ref{lemma:mini-lip-approx} with $\vdim$ and $r$
    replaced by $2$ and $2$, one estimates
    \begin{gather*}
    	\measureball{\| V \|}{ \oball 08} \leq 100 Q \unitmeasure 2,
        \\
        \begin{aligned}
            \| V \| ( \cylinder T024 \without \cylinder T021 ) 
            & \leq \| V \| ( \oball 08 \cap \{ z \with \dist (z,T) \geq 1 \} )
            \\
            & \leq \| V \| ( \oball 08 \without Z ) + \gamma^2 
            \leq \varepsilon^2 + 1
            \leq 2 \unitmeasure{2},
        \end{aligned}
    \end{gather*}
    and, noting $2 < \Delta_3 < 8$ and
    \begin{gather*}
    	C \subset \oball 0{\Delta_3} \cup
        \big (
	  \oball 08 \cap \{ z \with \dist (z,T)^2 \geq \Delta_3^2 - 4 \} \big
	  ),
    \end{gather*}
    one employs the monotonicity identity (see for instance \cite[4.5,
    4.6]{snulmenn:tv.v2}) and the bounds on $\varepsilon$ and $\gamma$ to infer
    \begin{align*}
    	 \| V \| ( C ) 
	&\leq \measureball{\| V \|}{ \oball 0{\Delta_3} } + \| V \| ( \oball
	08 \cap \{ z \with \dist (z,T)^2 \geq \Delta_3^2-4 \} ) \\
        &\leq \Delta_3^2 \big (
          8^{-2} \measureball{\| V \|}{ \oball 08 } 
          + \tint 28 t^{-2} \measureball{\| \delta V \|}{ \oball 0t } \ud \LM^1t
          \big )
        \\
	& \qquad + \| V \| ( \oball 08 \without Z ) + ( \Delta_3^2-4)^{-1}
	\gamma^2 \\
        &\leq \Delta_3^2 \big (
          (Q+1/4) \unitmeasure 2 + \tfrac 38 \alpha 
          + \varepsilon^2 + ( \Delta_3^2-4)^{-1} \gamma^2
        \big )
        \\
        &\leq \Delta_3^2 ( Q + 3/8 ) \unitmeasure 2
        = 4 ( Q + 1/2 ) \unitmeasure 2.
    \end{align*}
    Using the hypotheses one also gets $\|V\|(C) \ge 4 ( Q - 1/2 )
    \unitmeasure{2}$. Therefore, applying \ref{lemma:mini-lip-approx} with
    $\vdim$ and $r$ replaced by~$2$ and $2$ in conjunction with Kirszbraun's
    theorem, see~\cite[2.10.43]{MR41:1976}, one obtains a Borel set $X$ and a
    function $f : \rel^2 \to \rel$ such that $f|X$ satisfies the conditions of
    \ref{lemma:mini-lip-approx} and $\Lip f \leq 1$, in particular $f$ is
    weakly differentiable with $\weakD f(x) = \Der f(x)$ for $\LM^2$ almost all
    $x$ by \cite[2.13, 2.14]{MR2003a:49002}. Define
    \begin{gather*}
        A = X \without \pp \lIm C \cap
        \{ z \with \density^2 ( \| V \|, z ) \in \nat \} \without Z \rIm
    \end{gather*}
    and notice that $A$ is $\LM^2$ measurable by \cite[2.55]{MR2003a:49002}
    and~\cite[2.2.13]{MR41:1976}. Since
    \begin{multline*}
    	\LM^2 ( \pp \lIm C \cap
        \{ z \with \density^2 ( \| V \|, z ) \in \nat \} \without Z \rIm )
        \\
        \leq \mathscr{H}^2 ( C \cap \{ z \with \density^2 ( \| V \|, z ) \in \nat \} \without Z )
        \leq \|  V \| ( C \without Z )
        \leq \varepsilon^2
    \end{multline*}
    by \cite[2.10.35]{MR41:1976} and Allard \cite[3.5\,(1b)]{MR0307015}, one
    infers from
    \ref{lemma:mini-lip-approx}\,\eqref{item:mini-lip-approx:bad-set} that
    \begin{gather*}
    	\LM^2 ( \oball 02 \without A )
        \le \LM^2 ( \cball 02 \without X ) + \LM^2 ( X \without A )
        \le \Delta_1 \varepsilon^2 
        \le 2 \unitmeasure 2.
    \end{gather*}
    Noting \cite[3.5\,(1c)]{MR0307015} and~\cite[2.8.17, 2.9.11,
    3.1.2]{MR41:1976} and observing that from the definition of $A$ it follows
    that $H \cap \pp^{-1} \lIm A \rIm \cap \{ z \with \density^2 ( \| V \|, z
    ) \in \nat \} \subset Z$, one applies
    \ref{lemma:mini-lip-approx}\,\eqref{item:mini-lip-approx:inverse-2-height-control}
    and~\ref{lemma:mini-lip-approx}\,\eqref{item:mini-lip-approx:bad-set}\,\eqref{item:mini-lip-approx:tilt-control}
    to obtain the following auxiliary estimates
    \begin{gather*}
    	\eqLpnorm{ \LM^2 \restrict A }{2}{f} \leq \gamma, \quad
        \eqLpnorm{ \LM^2 \restrict \oball 02}{2}{\Der f} \leq \Delta_4 (
        \alpha + \beta ).
    \end{gather*}

    Next, defining $\Phi$ as in \ref{miniremark:Phi}, it will be shown that
    \begin{align*}
        \eqLpnorm{\| V \| \restrict H} 2 {\perpproject T} & \leq \Delta_7 (
        \alpha + \varepsilon^{1/2} \beta + \varepsilon^{-1/2} \gamma ), \\
        \eqLpnorm{ \| V \| \restrict H}{\Phi}{\perpproject T} & \leq \Delta_8 (
        \alpha + \beta + \gamma ).
    \end{align*}
    To prove the first estimate, one notes $\alpha^2 \le \alpha$, $\beta^2 \le
    \varepsilon \beta \le \varepsilon^{1/2} \beta $ and applies
    \ref{lemma:mini-lip-approx}\,\eqref{item:mini-lip-approx:2-height-control}
    and \ref{lemma:int-ineq} with $r$ and $\varepsilon$ replaced by $2$ and
    $2^{-1} \Delta_1^{1/2} \varepsilon$ to deduce
    \begin{align*}
        & \eqLpnorm{\| V \| \restrict H}{2} { \perpproject T} 
        \leq \Delta_1 \big (
          \eqLpnorm{\LM^2 \restrict \oball 02} 2f + \alpha^2 + \beta^2
        \big ) \\
        & \qquad \leq \Delta_5 \big ( 
          \varepsilon^{1/2} \eqLpnorm{\LM^2 \restrict \oball 02}{2}{\Der f} 
          + \varepsilon^{-1/2} \eqLpnorm{\LM^2 \restrict A} 2f 
          + \alpha + \varepsilon^{1/2} \beta
        \big ),
    \end{align*}
    hence the estimate follows using the auxiliary estimates. To prove the
    second estimate, one employs
    \ref{remark:basic_orlicz}\,\eqref{item:basic_orlicz:estimate},
    \ref{lemma:mini-lip-approx}\,\eqref{item:mini-lip-approx:Phi-height-control},
    and \ref{corollary:orlicz-poincare} to infer
    \begin{align*}
        & \eqLpnorm{\| V \| \restrict H}{\Phi} { \perpproject{T} } 
        \leq 4 \Delta_1 \big (
          \eqLpnorm{\LM^2 \restrict \oball 02}{\Phi/4}{f} + \alpha + \beta
        \big )
        \\
        & \quad \leq 4 \Delta_1 (\Gamma_{\ref{corollary:orlicz-poincare}} ( 2 ) + 1 )
        \big (
          \eqLpnorm{\LM^2 \restrict \oball 02} 2 {\Der f} +
          \eqLpnorm{\LM^2 \restrict A}{2}{f} + \alpha + \beta
        \big ),
    \end{align*}
    hence the estimate follows from the auxiliary estimates.

    To conclude the proof, one employs \ref{lemma:coercive-estimate} to obtain
    \begin{align*}
        & \tint{ \oball 01 \times \grass \adim 2}{} \| \project{S} -
        \project{T} \|^2 \ud V(z,S) \\
        & \qquad \leq \Gamma_{\ref{lemma:coercive-estimate}} (2) \big (
        \alpha^2 + \kappa ( \Delta_8 \alpha ( \alpha + \beta + \gamma ) ) + 3
        \Delta_7^2 ( \alpha^2 + \varepsilon \beta^2 + \varepsilon^{-1}
        \gamma^2 ) \big ),
    \end{align*}
    hence, noting $\alpha^2 \leq \kappa ( \alpha^2 )$ and $\kappa ( \Delta_8
    \alpha ( \alpha + \beta + \gamma ) ) \leq \Delta_8 \kappa ( \alpha ( \alpha
    + \beta + \gamma ) )$ by \ref{miniremark:kappa}, the conclusion is now
    readily derived.
\end{proof}

\begin{theorem}
    \label{theorem:quadratic-tilt-decay}
    Suppose $2 < \adim \in \nat$, $U$ is an open subset of~$\rel^\adim$, $V
    \in \IVar_2(U)$, and $\| \delta V \|$ is a Radon measure.

    Then, for $V$ almost all $(z,T)$, there holds
    \begin{gather*}
        \lim_{r \to 0+} r^{-4} ( \log (1/r) )^{-1} 
        \tint{\cball zr \times \grass \adim 2}{}
        \| \project{S} - \project{T} \|^2 \ud V (\zeta,S) = 0.
    \end{gather*}
\end{theorem}
\begin{proof}
    Define $Z = U \cap \{ z \with \Tan^2 ( \| V \|, z ) \in \grass \adim 2 \}$
    and $\tau : Z \to \Hom ( \rel^\adim, \rel^\adim )$ by $\tau (z) =
    \project{\Tan^2 ( \| V \|, z )}$ for $z \in Z$. Recall that
    \begin{gather*}
        V(k) = \tint{Z}{} k(z,\tau(z)) \density^2 ( \| V \|, z ) \ud \mathscr{H}^2 z
        \quad \text{for $k \in \mathscr{K} (U \times \grass \adim 2 )$}
    \end{gather*}
    from Allard \cite[3.5\,(1b)]{MR0307015} and that there exists a countable
    collection $C$ of $2$ dimensional submanifolds of $\rel^\adim$ of class $2$
    such that $\| V\| (U \without \bigcup C ) = 0$ from
    \cite[4.8]{snulmenn.c2}. Notice that
    \begin{gather*}
        \Tan (M,z) = \Tan^2 ( \| V \|, z ) 
        \quad \text{for $\| V \|$ almost all $z \in U \cap M$}
    \end{gather*}
    for $M \in C$ by \cite[2.8.18, 2.9.11, 3.2.17]{MR41:1976} and Allard
    \cite[3.5\,(2)]{MR0307015}. In~particular, one may construct a sequence of
    functions $\tau_i : U \to \Hom ( \rel^\adim, \rel^\adim )$ of class $1$ such
    that the sets $Z_i = U \cap \{ z \with \tau(z) = \tau_i (z) \}$ cover $\| V
    \|$ almost all of~$U$. For $i \in \nat$, applying \ref{thm:O_o} with
    $\vdim$, $p$, $\omega(r)$, $Z$, $f$, and $q$ replaced by $2$, $1$, $r^2
    (1 + \log(1/r))$, $Z_i$, $\| \tau - \tau_i \|^2$, and $\infty$ one infers, for $\|
    V \|$ almost all $z \in Z_i$, that
    \begin{align*}
        & \limsup_{r \to 0+} r^{-4} ( \log (1/r) )^{-1} \tint{ \cball zr }{}
        \| \tau(\zeta)-\tau(z)\|^2 \ud \| V \| \zeta 
        \\
        & \qquad = \limsup_{r \to 0+} r^{-4} ( \log (1/r))^{-1} \tint{\cball zr}{} 
        \| \tau - \tau_i \|^2 \ud \| V \| \in \{ 0, \infty \}.
    \end{align*}
    
    Therefore it is sufficient to prove for $\|V\|$ almost all $c$ that
    \begin{gather*}
        \limsup_{r \to 0+} r^{-4} ( \log (1/r) )^{-1} \tint{\cball cr}{} 
        \| \tau (z)- \tau(c) \|^2 \ud \| V \| z < \infty.
    \end{gather*}
    
    For $\| V \|$ almost all $c \in U$ there exist $Q \in \nat$ and $M \in C$
    such that
    \begin{gather*}
        \density^2 ( \| V \|, c ) = Q,
        \quad
        \density^{\ast 2} ( \| \delta V \|, c ) < \infty,
        \quad
        \density^2 ( \| V \| \restrict U \without M, c) = 0,
        \\
        \lim_{r \to 0+} r^{-2} 
        \tint{\cball cr}{} \| \tau (z) - \tau (c) \| \ud \| V \| z = 0
    \end{gather*}
    
    by Allard \cite[3.5\,(1c)]{MR0307015} and \cite[2.8.18, 2.9.5, 2.9.9,
    2.9.11]{MR41:1976}. Considering such $c$, $Q$ and~$M$ and abbreviating $T =
    \project{\Tan(M,c)}$, it follows
    \begin{gather*}
        \tau (c) = \project T,
        \quad
        \limsup_{s \to 0+} s^{-6}
        \tint{ \cball cs \cap M}{} \dist (z-c,T)^2 \ud \| V \| z < \infty
    \end{gather*}
    since $M$ is a submanifold of class $2$. Defining
    \begin{gather*}
    	\eta = \sup \{ 1, \Gamma_{\ref{lemma:app-coercive}} ( \adim, Q ) \},
        \quad
        \varepsilon = 2^{-14} \eta^{-1},
    \end{gather*}
    one obtains the existence of $0 < r \leq 1/4$ and $1 \leq \xi < \infty$
    such that $\oball c {8r} \subset U$ and for $0 < s \le r$
    \begin{gather*}
        s^{-1} \measureball{\| \delta V \|}{ \oball c{8s} } + s^{-2}
        \big (
          \tint{ \oball c{8s} \cap M}{} \dist (z-c,T)^2 \ud \| V \| z
        \big )^{1/2} \leq \xi s,
        \\
        \text{and $V$ satisfies the hypotheses of \ref{lemma:app-coercive}}
        \\
        \text{ with $r$ and $Z$ replaced by $s$ and $\oball c{8s} \cap M$.}
    \end{gather*}
    Abbreviating
    \begin{gather*}
        f(s) = s^{-2} \tint{ \oball cs \times \grass \adim 2}{} 
        \| \project S - \project T \|^2 \ud V(z,S)
        \quad \text{for $0 < s \leq 8r$},
        \\
        \Delta = \sup \big \{
          2^{20} \eta^2 \varepsilon^{-1} \xi^2,
          2^6 f (8r) r^{-2} ( \log (1/(8r)))^{-1}
        \big \},
    \end{gather*}
    one inductively proves that
    \begin{gather*}
        f (s) \leq \Delta s^2 \log (1/s)
        \quad \text{whenever $0 < s \leq 8r$};
    \end{gather*}
    in fact, the inequality is evident if $r \leq s \leq 8r$ and if it holds
    with $s$ replaced by~$8s$ for some $0 < s \leq r$, then, recalling $r \le
    1/4$, one notes that
    \begin{gather*}
	s^2 \leq \xi s \big ( \xi s + 2^6 \Delta^{1/2} s ( \log (1/s ) )^{1/2}
	\big ) \leq 8^{-1} \eta^{-1} \Delta s^2(\log (1/s))^{1/2}, \\
        1 + ( \log ( 1 + 1/s^2 ) )^{1/2} \leq 4 ( \log (1/s) )^{1/2},
        \quad
        \eta \varepsilon^{-1} \xi^2 s^2 \le 4^{-1} \Delta s^2 \log(1/s),
    \end{gather*}
    to infer from \ref{lemma:app-coercive} that
    \begin{align*}
	f(s) & \leq \eta \Big ( \kappa \big ( \xi s ( \xi s + 2^6
	\Delta^{1/2} s ( \log (1/s) )^{1/2} ) \big ) + 2^{12} \varepsilon
	\Delta s^2 \log (1/s) + \varepsilon^{-1} \xi^2 s^2 \Big ) \\
        & \leq 8^{-1} \Delta s^2 ( \log ( 1/s ) )^{1/2}
        \big ( 1 + ( \log (1+1/s^2))^{1/2} \big ) + 2^{-1} \Delta s^2 \log (1/s)
        \\
        & \leq \Delta s^2 \log (1/s),
    \end{align*}
    where $\kappa$ is as in \ref{lemma:app-coercive}.
\end{proof}
\begin{remark}
    In view of \ref{example:quadratic_tilt_excess} the decay rate is sharp for
    integral varifolds. For curvature varifolds a stronger conclusion is
    attainable, see \cite[15.9]{snulmenn:tv.v2}.
\end{remark}
\begin{remark}
    It is an open problem whether the integrality hypothesis on $V$ could be
    replaced by the requirement ``$\density^\vdim ( \| V \|, z ) \geq 1$ for
    $\| V \|$ almost all $z$''.
\end{remark}


\section{Super-quadratic tilt-excess, an example}
\label{sec:super_quadratic_tilt}

In this section we provide examples of curvature varifolds satisfying the
conditions of \ref{miniremark:situation_general} with $p = \infty$, hence in
particular having bounded generalised mean curvature vector, for which there is
a~set of positive weight measure such that in arbitrarily small balls around the
points of that set there is a portion of relatively large measure where the tilt
is greater than~$1/3$. In~fact, the Hausdorff measure of the regions in the
affine tangent planes which are not covered by the varifold, i.e., the size of
``holes'', is large at these scales which is essentially a stronger statement,
see~\ref{remark:holes-tilt-large}. The power of the decay of the super-quadratic
tilt-excess exhibited by these varifolds is the smallest possible, see
\ref{remark:example-super-quadratic-tilt} and \ref{thm:positive_result}. In
\ref{example:quantitative_brakke_again}, the example is modified so as to yield
the largest possible size of points of small density permitted by the
approximate lower semicontinuity of the density,
see~\ref{remark:density-ap-lsc}.

The qualitative construction principle was described by Brakke in
\cite[6.1]{MR485012} for two dimensional integral varifolds. Our implementation
employs additionally the estimates obtained in \ref{lemma:bent_catenoid} for
certain varifolds, see Figure \ref{F:bent-catenoid}, and the sets constructed
in~\ref{example:cantor_set} and~\ref{example:yet_another_cantor_set}.

\begin{figure}[!htb]
  \centering
  \includegraphics[width=\textwidth, keepaspectratio=true]{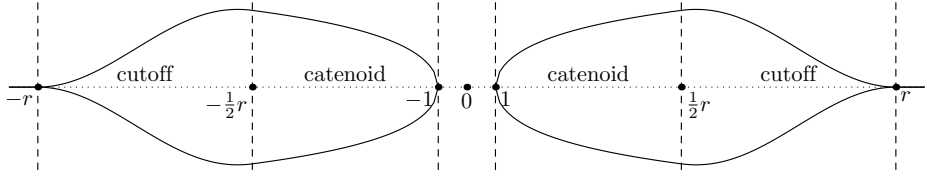}
  \caption{Rotating the solid line around the vertical axis illustrates the
    support of the varifold constructed in \ref{lemma:bent_catenoid}.}
  \label{F:bent-catenoid}
\end{figure}

\begin{miniremark}
    \label{miniremark:absorb}
    If $\phi$ is a measure, $A$ is $\phi$ measurable, $\phi (A) < \infty$, $f
    \in \Lp{\infty} ( \phi )$, $\varepsilon > 0$, and $\varepsilon \phi (A) \leq
    \tint A{} f \ud \phi$, then
    \begin{gather*}
    	( \varepsilon / 2 ) \phi (A) \leq \phi ( A \cap \{ x \with f(x) \geq
        \varepsilon / 2 \} ) \Lpnorm{\phi}{\infty}{f}.
    \end{gather*}
\end{miniremark}

\begin{lemma}
    \label{lemma:bent_catenoid}
    Suppose $\adim = 3$, $\pp$ and $\qq$ are related to $\adim$ as in
    \ref{miniremark:pqT}, and $4 \leq r < \infty$.

    Then there exists a~curvature varifold $V \in \IVar_2 ( \rel^3 )$
    satisfying
    \begin{gather*}
        \spt \| V \| \subset \im \pp^\ast \cup \pp^{-1} \lIm \oball 0r \rIm, 
        \\
        \| V \| \restrict \pp^{-1} \lIm \rel^2 \without \oball{0}{r} \rIm 
        = 2 \HM^2 \restrict \im \pp^\ast \without \oball{0}{r}, 
        \\
        \text{$\pp^{-1} \lIm \oball 0r \rIm \cap \spt \| V \|$
          is a two dimensional submanifold of $\rel^3$ of class $\infty$},
        \\
        \density^2 ( \| V \|, z ) = 1 
        \quad
        \text{for $z \in \pp^{-1} \lIm \oball 0r \rIm \cap \spt \| V \|$},
        \\
        0 \leq \| V \| \big ( \pp^{-1} \lIm \oball 0r \rIm \big )
        - 2 \measureball{\mathscr{L}^2}{ \oball 0r }
        \leq \Gamma ( \log r )^2,
        \\
        \| \delta V \| \leq \Gamma ( r^{-2} \log r ) \| V \|, 
        \quad \tint{}{} \| \mathbf{b} (V,z) \| \ud \| V \| z \leq \Gamma \log r, 
        \\
        \text{$| \pp (z) | \geq 1$ and $| \qq (z) | \leq 3 \log r$} 
        \quad
        \text{whenever $z \in \spt \| V \|$}, 
        \\
        V \big( \eqclassification{\rel^3 \times \grass{3}{2}}{(z,S)}
        {\| \project{S} - \pp^\ast \circ \pp \| \geq 1/3 } \big) \geq 1,
    \end{gather*}
    where $\Gamma$ is a universal positive, finite number.
\end{lemma}
\begin{proof}
    Choose $\gamma \in \mathscr{D} ( \rel, \rel )$ with $\{ t \with \gamma (t) >
    0 \} = \oball 01$ and
    \begin{gather*}
        0 \leq \gamma (t) \leq 1 \quad \text{for $t \in \rel$},
        \qquad
        \gamma (t) = 1
        \quad
        \text{for $-1/2 \leq t \leq 1/2$}.
    \end{gather*}
    Recalling \ref{miniremark:arcosh}, define $g : \classification{\rel}{t}{ 1
    < t < \infty} \to \rel$ by
    \begin{gather*}
        g(t) = \ach(t) \gamma (t/r) \quad \text{for $1 < t < \infty$},
    \end{gather*}
    hence there exists a positive, finite number $\Delta_1$ determined by
    $\gamma$ such that
    \begin{gather*}
        |g'(t)| \leq \Delta_1 r^{-1} \log r \quad \text{and} \quad |g''(t)| \leq
        \Delta_1 r^{-2} \log r \qquad \text{for $r/2 \leq t \leq r$}.
    \end{gather*}
    Defining $h : \rel^2 \without \cball{0}{1} \to \rel$ by $h(x) = g (|x|)$
    for $x \in \rel^2 \without \cball{0}{1}$, let
    \begin{gather*}
        M = \classification{\rel^3}{z}{ \qq (z) = h ( \pp (z) ) }.
    \end{gather*}
    Notice that $\mathbf{h} ( M,z ) = 0$ for $z \in M \cap \pp^{-1} \lIm
    \oball 0{r/2} \rIm$ by \ref{miniremark:catenoid} and
    \begin{gather*}
        M \cap \pp^{-1} \lIm \rel^2 \without \oball 0r \rIm = \im \pp^\ast
        \without \oball 0r, \\
    	\pp \lIm M \rIm \subset \rel^2 \without \oball 01, \quad \qq \lIm M
        \rIm \subset \cball 0{3 \log r}, \\
        M \cap \pp^{-1} \lIm \cball 02 \rIm \subset \{ z \with \| \project{\Tan
          ( M,z)} - \pp^\ast \circ \pp \| \geq 1/3 \}.
    \end{gather*}
    by \ref{miniremark:pqT}\,\eqref{item:pqT:graphs} and
    \ref{miniremark:arcosh}. Recalling $r \geq 4$, one may deduce from
    \ref{miniremark:arcosh} and \ref{miniremark:catenoid} that
    \begin{gather*}
        0 \leq \HM^2 \big ( M \cap \pp^{-1} \lIm \oball{0}{r/2} \rIm \big ) -
        \mathscr{L}^2 ( \oball 0{r/2} ) \leq 12 \log r.
    \end{gather*}
    Noting that $(1+s)^{1/2} \leq 1 + s/2$ for $-1 \leq s < \infty$, one
    estimates
    \begin{align*}
    	0 & \leq \HM^2 \big ( M \cap \pp^{-1} \lIm \oball 0r \without \oball
        0{r/2} \rIm \big ) - \mathscr{L}^2 ( \oball 0r \without \oball 0{r/2}
        ) \\
        & \leq 8 \tint{r/2}{r} \big ( ( 1 + g'(t)^2 )^{1/2} - 1 \big ) t \ud
        \mathscr{L}^1 t \leq 4 \tint{r/2}{r} g'(t)^2 t \ud \mathscr{L}^1 t \leq
        2 \Delta_1^2 ( \log r )^2.
    \end{align*}
    In view of the estimates for $g'$ and $g''$ and
    \ref{miniremark:radial_function}, one notes
    \begin{gather*}
        \| \mathbf{b} (M,z) \| \leq 2 \Delta_1 r^{-2} \log r \quad
        \text{whenever $z \in M$ and $r/2 \leq | \pp (z) | \leq r$},
    \end{gather*}
    hence, using \ref{miniremark:radial_function} and \ref{miniremark:arcosh},
    one obtains
    \begin{gather*}
    	\tint{M}{} \| \mathbf{b}(M,z) \| \ud \mathscr{H}^2 z \leq \Delta_2 \log r,
    \end{gather*}
    where $\Delta_2$ is a positive, finite number determined by $\gamma$.

    Employing the reflection $L : \rel^3 \to \rel^3$ given by $L(z)=
    \pp^\ast(\pp(z)) - \qq^\ast(\qq(z))$ for $z \in \rel^3$, one defines a
    curvature varifold
    $V \in \IVar_2 ( \rel^3 )$ by
    \begin{gather*}
        V(k) = \tint{M}{} k(z,\Tan(M,z)) + k (L(z), L \lIm \Tan (M,z) \rIm ) \ud \HM^2 z
    \end{gather*}
    for $k \in \mathscr{K} ( \rel^3 \times \grass{3}{2} )$. Hence one may take
    $\Gamma = 4 \sup \{ (3 + \Delta_1)^2, \Delta_2 \}$.
\end{proof}
\begin{example}
    \label{example:quantitative_brakke}
    Suppose $\vdim$ is an integer with $\vdim \geq 2$ and $\omega$ is a
    modulus of continuity satisfying the Dini condition.

    Then there exist $\varepsilon$, $C$, $M$, $R$, $T$, and $V$ satisfying
    \begin{gather*}
        \varepsilon > 0,
        \quad
        R \in \grass{\vdim+1}{\vdim-2},
        \quad
        T \in \grass{\vdim+1}{\vdim},
        \quad
        \text{$C$ is a Borel subset of $T$}, 
        \\
        \text{$M$ is an $\vdim$ dimensional submanifold 
          of $\rel^{\vdim+1}$ of class $\infty$},
        \\
        \text{$V \in \IVar_\vdim ( \rel^{\vdim+1} )$ is a~curvature varifold
          with $\density^\vdim ( \| V \|, z ) = 1$ for $z \in M$},
        \\
        \text{$C$, $M$, and $T$, are invariant under translations in
          directions belonging to $R$}, \\
        \spt \| V \| \subset M\cup T,
        \quad
        \| \delta V \| \leq \| V \|, 
        \quad
        \| V \| (C) > 0, \quad \density^\vdim ( \| V \|, c ) = 2,
        \\
        \| V \| ( \cball{c}{r} \cap \{ z \with \density^\vdim ( \| V \|, z) = 1 \} )
        \geq \omega (r) r^\vdim,
        \\
        \begin{aligned}
            & \inf \big \{
              V( \eqclassification{\cball cr \times \grass{\vdim+1}{\vdim}}{(z,S)}
                {\| \project{S} - \project{T} \| \geq 1/3 } ),
            \\
            & \phantom{\inf \big \{\ }
              \mathscr{H}^\vdim ( T \cap \cball cr \without \project T \lIm \spt \| V \| \rIm  )
            \big \}
            \geq \omega (r)^2 r^{\vdim+2} ( \log (1/r))^{-2}
        \end{aligned}
    \end{gather*}
    whenever $c \in C$ and $0 < r \leq \varepsilon$ and, if $\vdim > 2$, then
    there also exists a~curvature varifold $V' \in \IVar_2 ( \ker \project R )$
    such that
    \begin{gather*}
        V(k) = \tint{\rel^{\vdim+1} \times R}{} k(x+y, \im (\project{P} +
        \project R)) \ud V' \times \mathscr{H}^{\vdim-2} ((x,P),y)
    \end{gather*}
    whenever $k \in \mathscr{K} \big ( \rel^{\vdim+1}, \grass {\vdim+1} \vdim
    \big)$.
\end{example}
\begin{proof} [Construction]
    Suppose $\pp$ and $\qq$ as related to $\adim=\vdim+1$ as in
    \ref{miniremark:pqT} define $T = \im \pp^\ast$. Notice that it is sufficient
    to prove the assertion obtained by replacing closed balls ``$\cball cr$'' by
    open cubes ``$\ocube cr$'', see \ref{miniremark:open_cube}. Hence, in view
    of \ref{thm:product_varifold}, the construction may be reduced to the case
    $\vdim = 2$ by considering suitable products with $\vdim-2$ dimensional
    planes if $\vdim > 2$. Let $\Delta = 3 \sup \{
    \Gamma_{\ref{lemma:bent_catenoid}}, 3 \}$ and choose $0 < \eta \leq 1$ such
    that $\omega ( \eta ) \leq 2^{-6} \Delta^{-2}$.

    Define a modulus of continuity $\psi$ satisfying the Dini condition such
    that $\psi (r) = \sup \{ 8 \Delta \omega (r), 4r^2 \}$ for $0 \leq r \leq
    \eta$.  Apply \ref{example:cantor_set} with $\vdim$, $\omega$, and
    $\lambda$ replaced by $2$, $\psi$, and $1/2$ to obtain a number $\delta$,
    named ``$\varepsilon$'' there, as well as $G$ and $A$. Let
    \begin{gather*}
	\varepsilon = \inf \big \{ \delta, \Delta^{-1}, \eta \big \}, \quad B
	= \{ r \with 0 < r \leq \varepsilon \}.
    \end{gather*}

    Define $W \in \IVar_2 ( \rel^3 )$ by
    \begin{gather*}
        W ( k ) = 2 \tint{T \without \pp^\ast \lIm \bigcup G \rIm}{} k (z,T) \ud \HM^2 z
        \quad
        \text{for $k \in \mathscr{K} (\rel^3 \times \grass{3}{2} )$}.
    \end{gather*}
    Whenever $Q = \ocube{a}{s} \in G$ and $s > \Delta^{-1}$ let $X_Q \in \IVar_2
    ( \rel^3 )$ be defined by
    \begin{gather*}
        X_Q ( k ) = 2 \tint{T}{} k(z,T) \ud \HM^2 z
        \quad
        \text{for $k \in \mathscr{K} ( \rel^3 \times \grass{3}{2} )$}
    \end{gather*}
    and set $M_Q = \varnothing$. Whenever $Q = \ocube{a}{s} \in G$ and $s \leq
    \Delta^{-1}$ apply \ref{lemma:bent_catenoid} with $r$ replaced by $\Delta
    s^{-1} \log (1/s)$ to construct a~curvature varifold $X_Q \in \IVar_2 (
    \rel^3 )$ such that $( \boldsymbol{\mu}_{\Delta s^{-2} \log (1/s)} \circ
    \boldsymbol{\tau}_{-a})_\# X_Q$ satisfies the conditions of
    \ref{lemma:bent_catenoid} in place of $V$ implying
    \begin{gather*}
        \| X_Q \| \restrict \pp^{-1} \lIm \rel^2 \without \oball as \rIm
        = 2 \mathscr{H}^2 \restrict T \without \oball {\pp^\ast(a)}s,
        \\
        \text{$M_Q$ is a two dimensional submanifold of $\rel^3$ of class $\infty$},
        \\
        \density^2 ( \| X_Q \|, z ) = 1 \quad \text{for $z \in M_Q$}, 
        \\
        \| \delta X_Q \| \leq \| X_Q \|, 
        \qquad
        | \qq (z) | \leq s^2 \quad \text{for $z \in \spt \| X_Q \|$}, 
        \\
        \| X_Q \| \big ( \ocube{\pp^\ast(a)} s \cap
          \{ z \with \density^2 ( \| X_Q \|, z ) = 1 \} \big )
        \geq 2^{-1} \mathscr{L}^2 ( Q ),
        \\
        \| X_Q \| \big ( \pp^{-1} \lIm Q \rIm \big )
        \leq 2 \mathscr{L}^2 ( Q) + s^4,
        \quad
        \tint{}{} \| \mathbf{b} (X_Q,z) \| \ud \| X_Q \| z \leq s^2,
        \\
        \begin{aligned}
            & \inf \big \{ X_Q ( \eqclassification{\pp^{-1} \lIm Q \rIm \times
              \grass{3}{2}}{(z,S)}{\| \project{S} - \project{T} \| \geq 1/3 } ),
            \\
            & \phantom{ \inf\big \{ } \ 
            \mathscr{H}^2 ( \pp^\ast \lIm Q \rIm \without
            \project{T} \lIm \spt \| X_Q \| \rIm ) \big \}
            \geq \Delta^{-2} s^4 ( \log (1/s))^{-2},
        \end{aligned}
    \end{gather*}
    where $M_Q = \pp^{-1} \lIm \oball as \rIm \cap \spt \| X_Q \|$. Now, let
    $M = \bigcup \{ M_Q \with Q \in G \}$ and define $V \in \IVar_2 ( \rel^3
    )$ by
    \begin{gather*}
        V = W + \tsum{Q \in G}{} X_Q \restrict
        ( \pp^{-1} \lIm Q \rIm \times \grass 32 ).
    \end{gather*}
    Note that $\| V \| (\pp^\ast \lIm A \rIm) \geq 1$ and $\density^2 ( \| V \|,
    c ) = 2$ for $\| V \|$ almost all $c \in \pp^\ast \lIm A \rIm$ by Allard
    \cite[2.8\,(4a), 3.5\,(2)]{MR0307015}. Let $C = \classification{\pp^\ast
      \lIm A \rIm}{c}{\density^2(\|V\|, c) = 2 }$.  Moreover, observe that $V$
    is a~curvature varifold with $\| \delta V \| \leq \| V \|$. Finally, if $c
    \in \pp^\ast \lIm A \rIm$ and $r \in B$, then there exists $\ocube as = Q
    \in G$ with $Q \subset \oball{\pp(c)}{r}$ and $\LM^2( Q) = 4s^2 \geq \psi(r)
    r^2$, in particular $\ocube{\pp^\ast(a)} s \subset \ocube cr$, $s \leq
    \Delta^{-1}$, and
    \begin{gather*}
        \Delta^{-2} s^4 ( \log (1/s) )^{-2}
        \geq 2^{-6} \Delta^{-2} \psi (r)^2 r^4 ( \log (1/r) )^{-2}
        \geq \omega(r)^2 r^4 ( \log(1/r))^{-2}
    \end{gather*}
    since $r \geq s \geq \psi (r)^{1/2} r/2 \geq r^2$.
\end{proof}
\begin{remark}
    \label{remark:holes-tilt-large}
    Concerning the relation of the two terms occurring in the infimum, the
    following observation is particularly appropriate. \emph{If $n$, $Q$, $L$,
    $M$, $\delta_1$, $\delta_2$, $\delta_3$, $\delta_4$, $\varepsilon$, $m$,
    $s$, $S$, $U$, $V$, $\delta$, and $B$ are as in \cite[4.1]{snulmenn.c2},
    $p = \vdim$, $\psi$ is related to $\vdim$, $\adim$, $p$, $U$ and $V$ as in
    \ref{miniremark:situation_general}, and $\psi ( U )^{1/\vdim} \leq
    \delta$, then}
    \begin{gather*}
        \| V \| (B) \leq 2 \delta^{-1} \adim^{1/2} \besicovitch \adim V ( ( U
        \times \grass \adim \vdim ) \cap \{(z,R) \with | \project R - \project S
        | \geq \delta/2 \} );
    \end{gather*}
    in fact, this follows from \ref{miniremark:absorb}, Allard
    \cite[8.9\,(3)]{MR0307015}, and the Besicovitch-Federer covering theorem.
\end{remark}
\begin{remark}
    \label{remark:example-super-quadratic-tilt}
    Taking $\omega$ in \ref{example:quantitative_brakke} such that $\omega (t) =
    ( \log (1/t))^{-1} (\log(\log(1/t)))^{-2}$ for $0 < t \leq e^{-e}$, where
    $e$ denotes the Euler's number, one obtains
    \begin{gather*}
        \lim_{r \to 0+} r^{-\vdim-2-\delta}
        \tint{\cball{c}{r} \times \grass{\adim}{\vdim}}{}
        \| \project{S} - \project{T} \|^\iota \ud V (z,S) = \infty
    \end{gather*}
    whenever $c \in C$, $\delta > 0$, and $1 \leq \iota < \infty$. Taking $\iota
    > 2$ and $\delta = \iota -2$, one infers
    \begin{gather*}
        \tint{M \cap \oball cr}{} \| \mathbf{b} (M,z) \|^q \ud \mathscr{H}^\vdim z
        = \infty
    \end{gather*}
    whenever $c \in \spt ( \| V \| \restrict C )$, $0 < r < \infty$, and $1 < q
    < \infty$; in fact, the Cartesian product structure of $M$ and $V$ reduces
    the problem to the case $\vdim = 2$ in which, in view of
    \ref{lemma:second_fundamental_form}\,\eqref{item:second_fundamental_form:norm_equiv}
    and \ref{remark:curvature_varifold}, one may apply
    \cite[11.4\,(3)]{snulmenn:tv.v2} with $f(z)$ replaced by $\project{\Tan^2 (
    \| V \|,z)}$.
\end{remark}

\begin{remark}
    \label{remark:brakke_example}
    For comparison note the following well known proposition: \emph{If
      $\vdim$ and $\adim$ are positive integers, $\vdim \leq \adim$, $0 \leq
      K < \infty$, $V \in \IVar_\vdim ( \rel^\adim )$ with $\| \delta V \|
      \leq K \| V \|$ then there exists a relatively open, dense subset $A$
      of $\spt \| V \|$ such that for any $1 \leq q < \infty$ there holds
      \begin{gather*}
          \limsup_{r \to 0+} r^{-\vdim-q} \tint{\cball{a}{r}
            \times \grass{\adim}{\vdim}}{} | \project{S} - \project{T} |^q
          \ud V (z,S) < \infty
      \end{gather*}
      for $V$ almost all $(a,T) \in A \times \grass{\adim}{\vdim}$}; in fact,
    one may combine Allard \cite[8.1\,(1)]{MR0307015} with elliptic regularity
    theory as provided, e.g.,~in \cite[3.6, 3.21]{snulmenn.c2} and properties of
    Sobolev functions, see Calder\'on and Zygmund \cite[Theorem~12,
    p.~204]{MR0136849} or \cite[Theorem 3.4.2]{MR1014685}. In particular, the
    tangent plane behaviour exhibited in the preceding example may not occur at
    $V$ almost all points.
\end{remark}
\begin{remark}
    Example \ref{example:quantitative_brakke} is a refinement of the example
    described by Brakke in \cite[6.1]{MR485012}.
\end{remark}
\begin{example}
    \label{example:quantitative_brakke_again}
    Suppose $\vdim$ is an integer with $\vdim \geq 2$ and $\omega$ is a
    modulus of continuity.

    Then there exist $\varepsilon$, $B$, $C$, $M$, $T$, and $V$ satisfying
    \begin{gather*}
        \varepsilon > 0, \quad B \subset \rel \cap \{ t \with t > 0 \}, \quad
        R \in \grass{\vdim+1}{\vdim-2}, \quad T \in \grass{\vdim+1}{\vdim}, \\
        \inf B = 0, \quad \text{$C$ is a Borel subset of $T$}, \\
        \text{$M$ is an $\vdim$ dimensional submanifold of $\rel^{\vdim+1}$ of
          class $\infty$}, \\
        \text{$V \in \IVar_\vdim ( \rel^{\vdim+1} )$ is a~curvature varifold
          with $\density^\vdim ( \| V \|, z ) = 1$ for $z \in M$}, \\
        \text{$C$, $M$, and $T$ are invariant under translations in directions
          belonging to $R$}, \\
        \spt \| V \| \subset M \cup T, \quad \| \delta V \| \leq \| V \|,
        \quad \| V \| (C) > 0, \quad \density^\vdim ( \| V \|, c ) = 2, \\
        \| V \| ( \cball{c}{r} \cap \{ z \with \density^\vdim ( \| V \|, z) = 1
        \} ) \geq \omega (r) r^\vdim \quad \text{for $0 < r \leq
          \varepsilon$}, \\
        \begin{aligned}
	    & \inf \big \{ V( \eqclassification{\cball cr \times
	    \grass{\vdim+1}{\vdim}}{(z,S)} {\| \project{S} - \project{T} \|
	    \geq 1/3 } ), \\
	    & \phantom{\inf \big \{\ } \mathscr{H}^\vdim ( T \cap \cball cr
	    \without \project T \lIm \spt \| V \| \rIm ) \big \} \geq \omega
	    (r) r^{\vdim+2} ( \log (1/r))^{-2} \quad \text{for $r \in B$}
        \end{aligned}
    \end{gather*}
    whenever $c \in C$ and, if $\vdim > 2$, then there also exists a~curvature
    varifold $V' \in \IVar_2 ( \ker \project R)$ such that
    \begin{gather*}
        V(k) = \tint{\rel^{\vdim+1} \times R}{} 
        k(x+y, \im (\project{P} + \project R))
        \ud V' \times \mathscr{H}^{\vdim-2} ((x,P),y)
    \end{gather*}
    whenever $k \in \mathscr{K} \big ( \rel^{\vdim+1}, \grass {\vdim+1} \vdim
    \big)$.
\end{example}
\begin{proof} [Construction]
    Modify the construction of \ref{example:quantitative_brakke} by replacing
    its second paragraph by ``Define a modulus of continuity $\psi$ such that
    $\psi (r) = \sup \{ 8 \Delta \omega (r)^{1/2}, 4r^2 \}$ for $0 \leq r \leq
    \eta$. Apply \ref{example:yet_another_cantor_set} with $\vdim$, $\omega$,
    and $\lambda$ replaced by $2$, $\psi$, and $1/2$ to obtain a number
    $\delta$, named `$\varepsilon$' there, as well as $B$, $G$ and $A$. Let
    $\varepsilon = \inf \big \{ \delta, \Delta^{-1}, \eta \big \}$.'' and
    ``$\omega(r)^2$'' in the last displayed inequality by ``$\omega(r)$'', and
    adding ``If $c \in \pp^\ast \lIm A \rIm$ and $0 < r \leq \varepsilon$ there
    exists $H$ such that $H \subset G \cap \{ Q \with Q \subset \oball{\pp (c)}r
    \}$ and $\mathscr{L}^2 ( \bigcup H ) \geq \psi (r)r^2$, in particular
    $\ocube as \in H$ implies $\ocube{\pp^\ast(a)}s \subset \ocube cr$ and $s
    \leq \Delta^{-1}$.'' at the end, to obtain a construction for the present
    assertion.
\end{proof}
\begin{remark}
    The main modification of the construction of
    \ref{example:quantitative_brakke_again} in comparison to
    \ref{example:quantitative_brakke} is the usage of
    \ref{example:yet_another_cantor_set} in place of \ref{example:cantor_set}
    and that $B$ is a (countable) set constructed in
    \ref{example:yet_another_cantor_set} rather than an interval.
\end{remark}

\begin{remark}
    \label{remark:second_fundamental_form_large}
    As in \ref{remark:example-super-quadratic-tilt}, one obtains
    \begin{gather*}
	\tint{M \cap \oball cr}{} \| \mathbf{b} (M,z) \|^q \ud
	\mathscr{H}^\vdim z = \infty
    \end{gather*}
    whenever $c \in \spt ( \| V \| \restrict C )$, $0 < r < \infty$, and $1 <
    q < \infty$.
\end{remark}
\begin{remark} \label{remark:density-ap-lsc}
    Since $\density^\vdim ( \| V \|, c ) = 2$ for $c \in C$, the lower bound
    on $$\| V \| ( \cball cr \cap \{ z \with \density^\vdim ( \| V \|, z ) = 1
    \} )$$ is the largest one permitted by the approximate continuity of
    $\density^\vdim ( \| V \|, \cdot )$ with respect to $\| V\|$ and the
    standard Vitali relation, see \cite[2.8.18, 2.9.13]{MR41:1976}.
\end{remark}


\section{Super-quadratic tilt-excess, decay rates}
\label{sec:sqte-decay}

The present section concerns integral varifolds of at least two dimensions,
deferring the one dimensional case to Section~\ref{sec:one_dimensional_decay}.
Its purpose is to complement the examples concerning the decay rates of the
super-quadratic tilt-excess constructed in \ref{example:quantitative_brakke},
\ref{example:quantitative_brakke_again} and \cite[\S 1]{snulmenn.isoperimetric}
by positive results, see \ref{thm:positive_result}.  This yields a~sharp
dividing line in most cases, see
\ref{remark:sharpness_superquadratric}--\ref{remark:m=2,p=1}. Additionally, we
prove that the examples constructed in \ref{example:quantitative_brakke} and
\ref{example:quantitative_brakke_again} are essentially sharp also with respect
to the size of holes the varifolds contain, see \ref{corollary:holes}.

The positive results follow readily from the existing theory. For the
super-quadratic tilt-excess, these are the second order rectifiability and its
consequences for the decay of the quadratic tilt-excess in conjunction with the
differentiation theory both obtained in~\cite[4.8, 5.2]{snulmenn.c2} and
\cite[\S 3]{snulmenn.isoperimetric} respectively. Concerning the estimate for
the size of the holes, we additionally employ an approximation by $\qspace_Q (
\rel^\codim )$ valued functions, see~\cite[3.15]{snulmenn.poincare}, and more
basic results on the size of the set where the first variation is large
from~\cite[\S 2]{snulmenn.isoperimetric}.

\begin{theorem}
    \label{thm:positive_result}
    Suppose $\vdim$, $\adim$, $p$, $U$, and $V$ satisfy the hypotheses of
    \ref{miniremark:situation_general}, $V \in \IVar_{\vdim}(U)$, $2 < q <
    \infty$, and either
    \begin{enumerate}
    	\item $\vdim = 2$ and $p>1$, or
    	\item $\vdim>2$ and $p \geq 2\vdim/(\vdim+2)$.
    \end{enumerate}

    Then for $V$ almost all $(z,T)$ there holds
    \begin{gather*}
        \lim_{r \to 0+} r^{-\vdim-2} \tint{\cball{z}{r} \times
          \grass{\adim}{\vdim}}{} | \project{S} - \project{T} |^q \ud V
        (\zeta,S) = 0.
    \end{gather*}
\end{theorem}
\begin{proof}
    Assume $\vdim < \adim$. First, note that since the function mapping $S \in
    \grass{\adim}{\vdim}$ to $|\project S - \project T|$ is bounded for any $T
    \in \grass{\adim}{\vdim}$, we have
    \begin{gather*}
	\limsup_{r \to 0+} r^{-\vdim-2} \tint{\cball{z}{r} \times
	\grass{\adim}{\vdim}}{} | \project{S} - \project{T} |^q \ud V
	(\zeta,S) < \infty
    \end{gather*}
    for $V$ almost all $(z,T)$ by \cite[5.2\,(2)]{snulmenn.c2} and H\"older's
    inequality. Second, note that \cite[4.8]{snulmenn.c2} implies the
    existence of a sequence of functions $\tau_i : U \to \Hom ( \rel^\adim,
    \rel^\adim )$ of class $\class{1}$ such that
    \begin{gather*}
        \| V \| \big ( U \without {\textstyle\bigcup_{i=1}^{\infty}}
        Z_i \big ) = 0,
    \end{gather*}
    where $Z_i = \classification{U}{z}{ \tau_i(z) = \project{\Tan^\vdim ( \| V
    \|, z )} }$, hence
    \begin{gather*}
        \lim_{r \to 0+} r^{-\vdim-2} \tint{\cball{z}{r}}{} | \tau_i (\zeta) -
        \project{\Tan^\vdim ( \| V \|, \zeta )} |^q \ud \| V \|\zeta = 0
    \end{gather*}
    for $\| V \|$ almost all $z \in Z_i$
    by \cite[3.7\,(i)]{snulmenn.isoperimetric} with $Z$,
    $f$, $\alpha$, $r$, and $g$ replaced by $\Hom ( \rel^\adim,
    \rel^\adim)$, $\project{\Tan^\vdim ( \| V \|, \cdot )}$, $2/q$, $\infty$,
    and $\tau_i$. The conclusion then follows, since the functions $\tau_i$
    are of class~$1$.
\end{proof}
\begin{remark}
    The concept of proof is the same as in \cite[5.2\,(1)]{snulmenn.c2}.
\end{remark}
\begin{remark}
    \label{remark:sharpness_superquadratric}
    Note that the number $2$ in $r^{-\vdim-2}$ cannot be replaced by any larger
    number by \ref{example:quantitative_brakke} even if $\adim = \vdim+1$ and
    ``$\lim$'' is replaced by ``$\liminf$''.
\end{remark}
\begin{remark}
    Note the following proposition: \emph{If $\vdim$, $\adim$, $p$, $U$, and
    $V$ are as in \ref{miniremark:situation_general}, $V \in
    \IVar_{\vdim}(U)$, $\vdim > 2$, $p < 2\vdim/(\vdim+2)$, and $2 \leq q <
    \infty$, then
    \begin{gather*}
	 \lim_{r \to 0+} r^{-\vdim-\vdim p/(\vdim-p)} \tint{\cball{z}{r}
	 \times \grass{\adim}{\vdim}}{} | \project{S} - \project{T} |^q \ud V
	 (z,S) = 0
    \end{gather*}
    for $V$ almost all $(z,T)$;} in fact, it suffices to combine
  \cite[5.2\,(1)]{snulmenn.c2} with H\"older's inequality. Taking $\alpha_1 =
  \alpha_2$ slightly larger than $q^{-1} \vdim p(\vdim-p)^{-1}$ in
  \cite[1.2]{snulmenn.isoperimetric}, one infers that $\vdim p/(\vdim-p)$ cannot
  be replaced by any larger number in the preceding statement even if $\adim =
  \vdim+1$ and ``$\lim$'' is replaced by ``$\liminf$''.
\end{remark}
\begin{remark}
    \label{remark:m=2,p=1}
    Note the following proposition: \emph{If $\vdim$, $\adim$, $p$, $U$, and
    $V$ are as in \ref{miniremark:situation_general}, $V \in
    \IVar_{\vdim}(U)$, $\vdim=2$, $p=1$, $0 < s < 2$, and $2 \leq q < \infty$,
    then
    \begin{gather*}
	 \limsup_{r \to 0+} r^{-2-s} \tint{\cball{z}{r} \times
	 \grass{\adim}{\vdim}}{} | \project{S} - \project{T} |^q \ud V
	 (\zeta,S) < \infty
    \end{gather*}
    for $V$ almost all $(z,T)$;} in fact, again, it suffices to combine
  \cite[5.2\,(1)]{snulmenn.c2} with H\"older's inequality. Taking $\alpha_1 =
  \alpha_2$ slightly larger than $2q^{-1}$ in \cite[1.2,
  1.3]{snulmenn.isoperimetric}, one infers that $s$ cannot be replaced by any
  number larger than $2$ in the preceding statement and $s$ cannot be replaced
  by $2$ in case $q=2$ by \ref{example:quadratic_tilt_excess} both even if
  $\adim = \vdim+1$ and ``$\limsup$'' is replaced by ``$\liminf$''.  This leaves
  open the case $s=2$ and $q>2$. An affirmative answer to the latter case would
  be implied by interpolation if one would know
    \begin{gather*}
        \limsup_{r \to 0+} r^{-1} \phi (z,r,T) < \infty
        \quad
        \text{for $V$ almost all $(z,T)$},
    \end{gather*}
    where $\phi(z,r,T)$ abbreviates
    \begin{gather*}
        r^{-1} \sup \big\{ t V ( \eqclassification{\cball{z}{r} \times
          \grass{\adim}{\vdim} }{(\zeta,S)}{| \project{S} - \project{T} | >
          t})^{1/2} \with 0 < t < \infty \big \}.
    \end{gather*}
\end{remark}
\begin{remark}
    \label{remark:no-gehring-improvement}
    \emph{If $2 \leq \vdim \in \nat$, $1 \leq p < \infty$, and $2 < q <
    \infty$, then there exist $V \in \IVar_\vdim (U)$ related to $\vdim$,
    $\adim = \vdim+1$, $p$, and $U = \rel^\adim$ as in
    \ref{miniremark:situation_general} and $A$ with $V(A)>0$ satisfying
    \begin{gather*}
        \lim_{r \to 0+} \frac{\big ( r^{-\vdim} \tint{\cball zr \times \grass
            \adim \vdim }{} | \project S - \project T |^q \ud V(\zeta,S) \big
          )^{1/q}}{\big ( r^{-\vdim} \tint{\cball zr \times \grass \adim
            \vdim}{} | \project S - \project T|^2 \ud V(\zeta,S) \big )^{1/2} +
          r^{1-\vdim/p} \psi ( \cball zr)^{1/p}} = \infty
    \end{gather*}
    whenever $(z,T) \in A$, where $\psi$ is as in
    \ref{miniremark:situation_general};} in fact, choosing $\alpha$ such that
    $q^{-1} \vdim p(\vdim-p)^{-1} < \alpha < 2^{-1} \vdim  p ( \vdim-p )^{-1}$
    if $\vdim > 2$ and $p < 2 \vdim/(\vdim+2)$ and $2/q < \alpha < 1$
    otherwise, \ref{remark:sharpness_superquadratric}--\ref{remark:m=2,p=1}
    yield $V \in \IVar_\vdim (U)$ related $\vdim$, $\adim = \vdim+1$, $p$, and
    $U = \rel^\adim$ as in \ref{miniremark:situation_general} and $A$ with
    $V(A)>0$ such that
    \begin{gather*}
	\liminf_{r \to 0+} r^{-\alpha} \big ( r^{-\vdim} \tint{\cball zr
	\times \grass \adim \vdim}{} | \project S - \project T|^q \ud V (
	\zeta, S ) \big )^{1/q} > 0
    \end{gather*}
    whenever $(z,T) \in A$, hence \cite[5.2]{snulmenn.c2} and \cite[2.8.18,
    2.9.5, 2.9.8]{MR41:1976} imply the assertion. \emph{The same statement
    holds for $p = \infty$ if $\psi ( \cball zr )^{1/p}$ is replaced by
    $\eqLpnorm{\| V \| \restrict \cball zr}{\infty}{ \mathbf{h}(V,\cdot)}$.}
\end{remark}
\begin{theorem}
    \label{corollary:holes} Suppose $\vdim$, $\adim$, $p$, $U$, and $V$
    satisfy the hypotheses of \ref{miniremark:situation_general}, $V \in
    \IVar_{\vdim}(U)$, and either
    \begin{enumerate}
	\item $\vdim = 2$ and $p>1$, or
	\item $\vdim>2$ and $p \geq 2\vdim/(\vdim+2)$.
    \end{enumerate}

    Then for $V$ almost all $(c,T)$ there holds
    \begin{gather*}
        \lim_{r \to 0+} r^{-\vdim-2} \HM^\vdim ( H (T,c,r) ) = 0,
    \end{gather*}
    where $H (T,c,r) = T \cap \cball{\project{T}(c)}{r} \without \project{T}
    \lIm \cylinder Tcrr \cap \{ z \with \density^{\ast \vdim} ( \| V \|, z ) >
    0 \} \rIm$.
\end{theorem}
\begin{proof}
    Assume $1 < p < \vdim$. If $\vdim = \adim$, then $\delta V = 0$ by
    \cite[4.8]{snulmenn.c2}, hence the conclusion follows from Allard
    \cite[4.6\,(3)]{MR0307015}. Therefore assume $\vdim < \adim$.
    
    Suppose $Q$ is a positive integer. Recalling
    \cite[2.4]{snulmenn.isoperimetric}, define
    \begin{gather*}
        \lambda = \varepsilon_{\text{\cite[3.15]{snulmenn.poincare}}}
        ( \codim, \vdim, Q, 1, 5^\vdim Q, 1/4, 1/4, 1/4, 1/4, ( 2
        \isoperimetric{\vdim} \vdim)^{-\vdim} / \unitmeasure{\vdim} ),
        \\
        Z = U \cap \{ z \with \Tan^\vdim ( \| V \|, z ) \in
        \grass{\adim}{\vdim} \}
    \end{gather*}
    and $\tau \with Z \to \Hom ( \rel^\adim, \rel^\adim)$ by
    \begin{gather*}
        \tau (z) = \project{\Tan^\vdim ( \| V \|, z )} \quad
        \text{whenever $z \in Z$}.
    \end{gather*}
    Let $B_i$ consist of all $z \in \spt \| V \|$ such that either
    $\cball{z}{1/i} \not \subset U$ or
    \begin{gather*}
        \measureball{\| \delta V \|}{ \cball{z}{s} } > \lambda \,
        \| V \| ( \cball{z}{s} )^{1-1/\vdim} \quad \text{for some $0 <
          s < 1/i$}
    \end{gather*}
    whenever $i$ is a positive integer. Note that $B_{i+1} \subset B_i$.
    Moreover, let $D_i(c)$ denote the set of all $z \in U$ such that
    either $\cball{z}{1/i} \not \subset U$ or
    \begin{gather*}
        \tint{\cball{z}{s}}{} | \tau(\zeta) - \tau(c)| \ud \| V \| \zeta >
        \lambda \, \measureball{\| V \|}{\cball{z}{s}} \quad
        \text{for some $0 < s < 1/i$}
    \end{gather*}
    whenever $c \in Z$ and $i$ is a positive integer. Note that $D_{i+1}
    (c) \subset D_i(c)$.

    Next, the following assertion will be proven. \emph{For $\| V \|$ almost
    all $c$ there exists $i$ such that
    \begin{gather*}
	  \lim_{r \to 0+} r^{-\vdim-2} \| V \| ( B_i \cap \cball{c}{r} ) = 0,
	  \quad \lim_{r \to 0+} r^{-\vdim-2} \| V \| ( D_i(c) \cap
	  \cball{c}{r} ) = 0.
    \end{gather*}}
    Noting $\vdim p / (\vdim-p) \geq 2$ and applying \cite[2.9,
    2.10]{snulmenn.isoperimetric} with $m$, $n$, $\mu$, $s$, $\varepsilon$,
    and $\Gamma$ replaced by $\codim$, $\vdim$, $\| V \|$, $\vdim$, $\inf \big
    \{ ( 2 \isoperimetric{\vdim} )^{-p/(\vdim-p)}, \lambda^{p/(\vdim-p)}
    \big \}$, and $8 \isoperimetric{\vdim} \vdim$ yields the first equality.
    In view of \cite[5.2\,(2)]{snulmenn.c2}, applying
    \cite[3.7\,(ii)]{snulmenn.isoperimetric} with $n$, $m$, $\mu$, $Z$, $f$,
    $\alpha$, $q$, and $r$ replaced by $\vdim$, $\codim$, $\| V \|$, $\Hom (
    \rel^\adim, \rel^\adim )$, $\tau$, $1$, $2$, and $\infty$ one obtains the
    second equality.

    Note that for $V$ almost all $(c,T)$ with density $\density^\vdim ( \| V
    \|, c ) = Q$ the hypotheses of~\cite[3.15]{snulmenn.poincare} (Lipschitz
    approximation theorem) with $m$, $n$, $L$, $M$, $\delta_1$, $\delta_2$,
    $\delta_3$, $\delta_4$, $\delta_5$, $a$, $h$, and $\mu$ replaced by
    $\codim$, $\adim$, $1$, $5^\vdim Q$, $1/4$, $1/4$, $1/4$, $1/4$, $(2
    \isoperimetric{\vdim} \vdim )^{-\vdim} / \unitmeasure{\vdim}$, $c$, $r$,
    and $\| V \|$ are satisfied for all sufficiently small $r>0$. Therefore
    the conclusions (1)--(3) of \cite[3.15]{snulmenn.poincare} with
    $\varepsilon_1$ replaced by $\lambda$ in conjunction with the assertion of
    the preceding paragraph yield the conclusion.
\end{proof}
\begin{remark}
    Possibly up to logarithmic factors, the estimate obtained is sharp even in
    case $\codim = 1$ and $p = \infty$ by \ref{example:quantitative_brakke} and
    \ref{example:quantitative_brakke_again}.
\end{remark}


\section{The one dimensional case}
\label{sec:one_dimensional_decay}

For completeness, we consider in this section one dimensional integral varifolds
of locally bounded first variation. In that case we prove that there is a set,
almost equal to the support of the weight measure of the varifold, such that the
tangent map of the varifold is differentiable relative to this set almost
everywhere, see \ref{thm:decay-1d}. This implies that near almost all points the
varifold may be expressed by a finite sum of graphs of Lipschitzian functions,
see \ref{corollary:1-dim-Q-graph}.

The differentiability result for the tangent map mainly relies on an adaptation
of a coercive estimate of Allard and Almgren in \cite[\S 5]{MR0425741} in
conjunction with differentiability results of approximate and integral nature
obtained for that map in~\cite[\S 5]{snulmenn.c2}. The corollary then follows
from a suitable approximation by Lipschitzian $\qspace_Q ( \rel^{\adim-1} )$
valued functions, see~\cite[3.15]{snulmenn.poincare}, in combination with
a~structural result for such function, see Almgren \cite[1.10]{MR1777737}.

\begin{theorem}
    \label{lemma:sup-L1-est}
    Suppose $U$ is an open subset of $\rel^\adim$, $V \in \RVar_1 ( U )$, $M =
    \{ z \with 0 < \density^{\vdim} ( \| V \|, z ) < \infty \}$, $\mu$ is a
    Radon measure over $U$, $f : U \to \rel$ is a Lipschitzian function,
    \begin{gather*}
	\tint{}{} \ap \Der  f \bullet \ap \Der  \theta \ud \| V \| \leq \mu ( \theta )
	\quad \text{whenever $\theta \in \mathscr{D} ( U, \rel )$ and $\theta
	\geq 0$},
    \end{gather*}
    where ``$\ap$'' denotes approximate differentiation with respect to $(\| V
    \|, 1)$, and $\phi \in \mathscr{D} ( U, \rel )$ with $\phi \geq 0$.

    Then there holds
    \begin{gather*}
        \eqLpnorm{\HM^1 \restrict M}{\infty}{ \phi | \ap \Der f | \density^1 ( \| V \|, \cdot )}
        \leq \Lip ( \phi ) \tint{\spt \Der  \phi}{} | \ap \Der f| \ud \| V \| + \mu ( \phi ).
    \end{gather*}
\end{theorem}
\begin{proof}
    Abbreviate $\gamma = \Lip ( \phi) \tint{\spt \Der \phi}{} | \ap \Der f | \ud \| V \|
    + \mu ( \phi )$ and define
    \begin{gather*}
        h(t) = \tint{M \cap \{ z \with f(z) = t \}}{}
        \phi (z) | \ap \Der f(z) | \density^1 ( \| V \|, z ) \ud \mathscr{H}^0 z
        \quad \text{whenever $t \in \rel$}
    \end{gather*}
    and $T = \{ t \with h(t) \leq \gamma \}$. Since
    \ref{miniremark:rect_varifold}\,\eqref{item:rect:varifold:coarea} and
    Allard \cite[3.5\,(1b)]{MR0307015} imply
    \begin{gather*}
    	\ap \Der f(z) = 0
        \quad
        \text{for $\mathscr{H}^1$ almost all $z \in M \cap f^{-1} \lIm N \rIm$}
    \end{gather*}
    whenever $\mathscr{L}^1 (N) = 0$, it is sufficient to prove $\mathscr{L}^1
    ( \rel \without T ) = 0$.

    Approximating $\theta$ by convolution and using
    \cite[4.5\,(3)]{snulmenn.decay}, one obtains
    \begin{gather*}
	\tint{}{} \ap \Der f \bullet \ap \Der  \theta \ud \| V \| \leq \mu (\theta)
    \end{gather*}
    whenever $\theta : U \to \rel$ is a nonnegative Lipschitzian function with
    compact support. Employing
    \ref{miniremark:rect_varifold}\,\eqref{item:rect:varifold:coarea} with
    $g(z)$ replaced by $\phi (z) \psi'(f(z)) | \ap \Der  f (z) |$ and taking
    $\theta = \phi \cdot ( \psi \circ f )$ yields
    \begin{align*}
        \tint{}{} \psi' h \ud \mathscr{L}^1
        & = \tint{}{} \phi ( \psi' \circ f ) | \ap \Der f|^2 \ud \| V \|
        \\
        & = \tint{}{} \ap \Der f \bullet \ap \Der  \theta \ud \| V \| 
        - \tint{}{} ( \psi \circ f ) \ap \Der f \bullet \ap \Der  \phi \ud \| V \|
        \leq \gamma
    \end{align*}
    whenever $\psi \in \mathscr{E} ( \rel, \rel )$ and $0 \leq \psi \leq 1$.
    Letting $\psi$ approach the characteristic function of $\{ u \with t < u
    \}$ shows that $t \in T$ whenever $t$ is a Lebesgue point of $h$ and the
    conclusion follows from \cite[2.8.18, 2.9.8]{MR41:1976}.
\end{proof}

\begin{remark}
    \label{remark:sup-L1-est}
    If $\theta \in \mathscr{D} (U,\rel)$ and $f \in \Hom ( \rel^\adim, \rel)$,
    then $( \delta V ) ( \theta \grad f ) = \tint{}{} \ap \Der f \bullet \ap \Der 
    \theta \ud \| V \|$. Consequently, if $\| \delta V \|$ is a Radon measure,
    then
    \begin{gather*}
        \Lpnorm{\| V \|}{\infty}{ \phi | \ap \Der L | }
        \leq \adim \big ( \| L \| \| \delta V \| ( \phi ) 
        + \Lip ( \phi ) \tint{\spt \phi}{} | \ap \Der L | \ud \| V \| \big )
    \end{gather*}
    whenever $L \in \Hom ( \rel^\adim, \rel^\adim )$.
\end{remark}
\begin{remark} \label{remark:laplace-beltrami}
    The method of proof originates from Allard and Almgren
    \cite[5\,(6)]{MR0425741}. Adapting the terminology of \cite[p.~41, p.~188,
    p.~391]{MR1814364} to varifolds, our presentation views $f$ as a ``weak
    subsolution to Poisson's equation for the Laplace-Beltrami operator on
    $V$'', see also Allard \cite[7.5]{MR0307015}.
\end{remark}
\begin{theorem}
    \label{thm:decay-1d}
	Suppose $1 < \adim \in \nat$, $U$ is an open subset of $\rel^\adim$,
	$V \in \IVar_1 ( U )$, $\| \delta V \|$ is a Radon measure,
	\begin{gather*}
		C = \{ (z,\cball{z}{r}) \with \text{$z \in U$, $0 < r <
		\infty$} \}, \quad Z = U \cap \{ z \with \Tan^1 ( \| V \|, z )
		\in \grass{\adim}{1} \},
	\end{gather*}
	$\tau : Z \to \Hom ( \rel^\adim, \rel^\adim)$ satisfies $\tau(z) =
	\project{\Tan^1 ( \| V \|, z )}$ for $z \in Z$, and $A$ is the set of
	points in $\spt \| V \|$ at which $\tau$ is $( \| V \|, C )$
	approximately continuous.

	Then $\| V \| ( U \without A ) = 0$ and, for $\| V \|$ almost all $z
	\in A$, $\tau|A$ is differentiable relative to $A$ at~$z$ with
	\begin{gather*}
	    \Der (\tau|A) (z) = ( \| V \|, 1 ) \ap \Der  \tau (z) .
	\end{gather*}
\end{theorem}
\begin{proof}
	First, notice that $\| V \| ( U \without A ) = 0$ by \cite[2.8.18,
	2.9.13]{MR41:1976} and that, for $\| V \|$ almost all $z$, $\tau$ is
	$( \| V \|, 1 )$ approximately differentiable at $z$ and
	\begin{gather*}
		\limsup_{r \to 0+} r^{-2} \tint{\cball{z}{r}}{} | \tau (
		\zeta ) - \tau ( z )| \ud \| V \| z < \infty
	\end{gather*}
	by \cite[5.2\,(2), 5.5]{snulmenn.c2}. If $z$ additionally satisfies
	$\density^{\ast 1} ( \| \delta V \|, z ) < \infty$, as $\| V \|$
	almost all $z$ do by \cite[2.10.19\,(3)]{MR41:1976}, then
	\begin{gather*}
		\limsup_{A \owns \zeta \to z} | \tau (\zeta) - \tau (z)
		|/|\zeta-z| < \infty;
	\end{gather*}
	in fact, it is sufficient to take $L = \id{\rel^\adim} - \tau (z)$ and a
    suitable $\phi$ in \ref{remark:sup-L1-est} since $| \tau (\zeta) - \tau
    (z)|^2 = 2 | L \circ \tau (\zeta) |^2$ for $\zeta \in Z$ by Allard
    \cite[8.9\,(1)\,(2)]{MR0307015},

	Next, one obtains a sequence of functions $\tau_i : U \to \Hom (
	\rel^\adim, \rel^\adim )$ of class $1$ such that the sets $Z_i = Z
	\cap \{ z \with \tau (z) = \tau_i (z) \}$ cover $\| V \|$ almost all
	of $U$ and
	\begin{gather*}
		( \| V \|, 1 ) \ap \Der  \tau (z) = \Der  \tau_i (z) | \Tan^1 ( \| V
		\|, z) \quad \text{for $\| V \|$ almost all $z \in Z_i$}
	\end{gather*}
	by \cite[11.1\,(2)\,(4)]{snulmenn:tv.v2} and \cite[3.2.16]{MR41:1976}.
	Defining $f_i = ( \tau-\tau_i ) |A$, the preceding paragraph yields
	\begin{gather*}
		\limsup_{\zeta \to z} |f_i (\zeta)| / |\zeta-z| <
		\infty \quad \text{for $\| V \|$ almost all $z \in Z_i$},
	\end{gather*}
	hence, in view of \ref{thm:differentiation} and
	\cite[3.1.22]{MR41:1976}, it follows
	\begin{gather*}
		\lim_{\zeta \to z} | f_i (\zeta)|/|\zeta-z| = 0, \quad \Der  (
		\tau |A ) (z) = \Der  \tau_i (z) | \Tan (A,z)
	\end{gather*}
	for $\| V \|$ almost all $z \in Z_i$. Noting
	\begin{gather*}
		\Tan^1 ( \| V \|, z ) \subset \Tan (A,z) \subset \Tan (\spt \|
		V \|, z) \quad \text{for $z \in U$}
	\end{gather*}
	by \cite[3.2.16]{MR41:1976}, the conclusion now follows from
	\cite[11.3]{snulmenn:tv.v2}.
\end{proof}
\begin{corollary}
    \label{corollary:1-dim-Q-graph}
    Suppose $1 < \adim \in \nat$, $U$ is an open subset of $\rel^\adim$, $V
    \in \IVar_1(U)$, $\| \delta V \|$ is a Radon measure, and $\varepsilon >
    0$.

    Then, for $V$ almost all $(z,T)$, there exist $Q \in \nat$, $0 <r<
    \infty$, and $f_i : T \cap \cball{\project{T} (z)}{r} \to T^\perp \cap
    \cball{ \perpproject{T} (z)}{r}$ with $\Lip f_i \leq \varepsilon$ for $i =
    1, \ldots, Q$ such that
    \begin{gather*}
        \density^1 ( \| V\|, \zeta ) = \card \{ i \with f_i ( \project{T}
        ( \zeta ) ) = \perpproject{T} (\zeta) \}
    \end{gather*}
    for $\mathscr{H}^1$ almost all $\zeta \in \cylinder Tzrr$.
\end{corollary}
\begin{proof}
    Let $Z$, $\tau$ and $A$ be defined as in~\ref{thm:decay-1d}, in particular
    $\| V \| ( U \without A ) = 0$. In~view of~\ref{thm:decay-1d} and Allard
    \cite[3.5\,(1)]{MR0307015} it is sufficient to prove the conclusion at
    a~point $(z,T)$ such that $z \in A$, $T = \im \tau (z)$, $\| \delta V \| (
    \{ z \} ) = 0$, $\tau|A$ is continuous at~$z$, and, for some $Q \in \nat$,
    also
    \begin{gather*}
    	r^{-1} \tint{}{} k (r^{-1} (\zeta-z), S ) \ud V(\zeta,S) \to Q \tint
        T{} k ( \zeta, T ) \ud \mathscr{H}^1 \zeta \quad \text{as $r \to 0+$}
    \end{gather*}
    for $k \in \mathscr{K} ( \rel^\adim \times \grass \adim 1 )$. Define
    $\delta_1 = \delta_2 = \delta_3 = 1/2$, $\delta_4 = 1/4$, and $\delta_5 = (2
    \unitmeasure 1 \isoperimetric 1 )^{-1}$ and recall $\delta_5 \leq 1$, see
    e.g.~\cite[2.4]{snulmenn.isoperimetric}. From \cite[3.15]{snulmenn.poincare}
    one obtains
    \begin{gather*}
        \lambda = \varepsilon_{\text{\cite[3.15]{snulmenn.poincare}}} (
        \adim-1, 1, Q, \varepsilon, 5Q, \delta_1, \delta_2, \delta_3,
        \delta_4, \delta_5 ).
    \end{gather*}
    Choose $0 < r < \infty$ such that
    \begin{gather*}
        (Q - 1/2) \unitmeasure{1} r \le \|V\|(\cylinder Tzrr) \le (Q + 1/2)
        \unitmeasure{1} r , \\
        \|V\|(\cylinder Tzr{5r/4} \without \cylinder Tzr{r/2}) \le (1/2)
        \unitmeasure{1} r , \\
        \oball z{4r} \subset U, \quad \measureball{\|V\|}{\oball z{4r}} \le 5Q
        \unitmeasure{1} r , \quad \measureball{\|\delta V\|}{\oball z{4r}} \le
        \lambda, \\
        | \tau(\zeta) - \tau(z) | \le \lambda \quad \text{whenever $\zeta \in
          A \cap \oball z{4r}$}.
    \end{gather*}
    Applying \cite[3.15\,(1)--(3)]{snulmenn.poincare} with $m$, $n$, $L$, $M$,
    $a$, $h$, and $\varepsilon_1$ by $\adim-1$, $1$, $\varepsilon$, $5Q$, $z$,
    $r$, and $\lambda$ and noting that the set $B$ occurring there is empty,
    one infers the existence of a function $f$ with values in $\qspace_Q (
    T^\perp )$ with $\dmn f \subset T \cap \cball{ \project{T} (z) }{r}$,
    $\Lip f \leq \varepsilon$, and
    \begin{gather*}
        \spt f (x) \subset \cball { \perpproject{T} (z)} r \quad \text{for $x
          \in \dmn f$}, \\
        \density^1 ( \| V \|, \zeta ) = \density^0 ( \| f ( \project{T}
        (\zeta) ) \|, \perpproject{T} (\zeta) ) \quad \text{for $\zeta \in
          \cylinder Tzrr \cap \project{T}^{-1} \lIm \dmn f \rIm$}, \\
        \mathscr{H}^1 ( T \cap
        \cball{\project{T} (z)}{r} \without \dmn f) + \| V \| \big ( \cylinder
        Tzrr \without \project{T}^{-1} \lIm \dmn f \rIm \big ) = 0.
    \end{gather*}
    Consequently, \cite[2.10.19\,(4)]{MR41:1976} implies
    \begin{gather*}
    	\density^1 ( \| V \|, \zeta ) = 0 \quad \text{for $\mathscr{H}^1$
          almost all $\zeta \in \cylinder Tzrr \without \project{T}^{-1} \lIm
          \dmn f \rIm$}.
    \end{gather*}
    Defining $g = \Clos f$, one infers that $g$ is a function, $\dmn g = T \cap
    \cball {\project{T}(z)}r$, $\Lip g \leq \varepsilon$, and $\spt g(x) \subset
    \cball{\perpproject{T}(z)}{r}$ for $x \in \dmn g$. Now one readily verifies
    \begin{gather*}
	\density^0 ( \| g ( \project{T} (\zeta) ) \|, \perpproject{T} (\zeta))
	= 0 \quad \text{for $\mathscr{H}^1$ almost all $\zeta \in \cylinder
	Tzrr \without \project{T}^{-1} \lIm \dmn f \rIm$}
    \end{gather*}
    and the conclusion, both by means of Almgren \cite[1.10\,(2),
    1.10\,(1)\,(iii)]{MR1777737}.
\end{proof}
\begin{remark}
    The fact $\mathscr{H}^\vdim ( \{ (x,y) \with \text{$x \in X$ and $y \in \spt
      g(x)$} \} ) = 0$ whenever $\mathscr{L}^\vdim ( X ) = 0$ and $g : X \to
    \qspace_Q ( \rel^\codim )$ is Lipschitzian, deduced from Almgren
    \cite[1.10\,(2), 1.10\,(1)\,(iii)]{MR1777737} for $\vdim = 1$ in the
    preceding proof, clearly holds for arbitrary $\adim > \vdim \in \nat$ by
    Almgren \cite[1.5\,(11)\,(iii)\,(c)]{MR1777737} or
    \cite[2.5\,(1)]{snulmenn.poincare}.
\end{remark}


\addcontentsline{toc}{section}{\numberline{}References}
\bibliography{ohne_duplikate}
\bibliographystyle{alphaurl}

\medskip
{\small \noindent S{\l}awomir Kolasi{\'n}ski \\
Max Planck Institute for Gravitational Physics (Albert Einstein Institute), \\
Am M{\"u}hlen\-berg 1, D-14476 Golm, Germany \\
\texttt{Slawomir.Kolasinski@aei.mpg.de}}

\vspace{1em}

{\small \noindent Ulrich Menne \newline Max
Planck Institute for Gravitational Physics (Albert Einstein Institute),
\newline Am M{\"u}hlen\-berg 1, D-14476 Golm, Germany \newline
\texttt{Ulrich.Menne@aei.mpg.de} \newline
University of Potsdam, Mathematics Institute, \newline OT Golm,
Karl-Liebknecht-Stra\ss e 24--25, D-14476 Potsdam, Germany \newline
\texttt{Ulrich.Menne@uni-potsdam.de}}

\end{document}